\documentclass[12pt]{article}

\usepackage{amsmath}
\usepackage{amsfonts}
\usepackage{amssymb}
\usepackage{amsthm}

\usepackage[dvipsnames, usenames]{color}

\usepackage{natbib}

\usepackage{hyperref}

\hypersetup{
    colorlinks,
    linkcolor= {blue!90!black}, 
    citecolor={blue!60!black},
    urlcolor={blue!80!black}
}

\usepackage{soul} 

\usepackage{enumerate}
\usepackage{enumitem}

\usepackage{array,arydshln}

\usepackage{graphicx}
\usepackage{color}
\usepackage{xcolor}

 \theoremstyle{plain}
 
\newtheorem{theorem}{Theorem}[section]
\newtheorem{corollary}[theorem]{Corollary}

\newtheorem{lemma}{Lemma}[section] 
\newtheorem{proposition}{Proposition}[section]
\newtheorem{remark}{Remark}[section]

\newcommand{\prob}{\mathbb P}

\newcommand{\umbral}{u}

\newcommand{\esp}{\mathbb E}
\newcommand{\var}{\mathbb{V}ar}
\newcommand{\cov}{\mathbb{C}ov}

\newcommand{\pl}{\mathrm{PL}}

\newcommand\bX {\mathbf X}

\newcommand\bSi {\mbox{\boldmath $\Sigma$}}

\newcommand\itS {{\mathcal{S}}}
\newcommand\itU {{\mathcal{U}}}

\newcommand\indica {\mathbb{I}}
\newcommand{\trasp}{^\top}
\def\real{\mathbb{R}}
\def\natu{\mathbb{N}}

\newcommand{\convprob}{ \buildrel{p}\over\longrightarrow}
\newcommand{\convpp}{ \buildrel{a.s.}\over\longrightarrow}
\newcommand{\convdist}{ \buildrel{D}\over\longrightarrow}

\def\argmin{\mathop{\mbox{argmin}}}

\newcommand{\itI}{\mathcal{I}}

\begin{document}


%

\title{Threshold detection under a semiparametric regression model}

\author{Graciela Boente{$^{1}$}, Florencia Leonardi{$^{2}$},
Daniela Rodriguez{$^{1}$} and Mariela Sued{$^{1}$}}

\footnotetext[1]{Universidad de Buenos Aires and CONICET}
\footnotetext[2]{Universidade de S\~ao Paulo}

\date{}

\maketitle

\begin{abstract}
Linear regression models have been extensively considered in the literature. However, in some practical applications they may not be appropriate all over the range of the covariate. In this paper, a more flexible model is introduced by considering a regression model $Y=r(X)+\varepsilon$ where the regression function $r(\cdot)$ is assumed to be linear for large values  in the domain of the predictor variable $X$.  More precisely,   we  assume that $r(x)=\alpha_0+\beta_0 x$ for $x> \umbral_0$,  where the value $\umbral_0$ is identified as   the smallest value satisfying such a  property. A penalized procedure  is introduced to estimate the threshold $\umbral_0$. The considered proposal focusses on a semiparametric approach since no parametric model is assumed for the regression function for values smaller than  $\umbral_0$. Consistency properties of both the threshold estimator and   the estimators of $(\alpha_0,\beta_0)$ are derived, under mild assumptions. Through a numerical study,  the small sample properties of the proposed procedure and the importance of introducing a penalization are investigated. The analysis of a real data set  allows us to demonstrate the usefulness of the penalized estimators.

\end{abstract}
\section{Introduction}
Predicting or understanding the structural relationship between a response variable $Y$ based on a scalar explanatory variable $X$  is the primary goal of the so-called regression methods. From the foundational work of \citet{galton1886regression} where simple regression models were introduced, a wide variety of models and estimation procedures have been developed. In this paper, we consider the situation of additive errors, that is we assume that $Y=r(X)+\varepsilon$, where $X$ and $\varepsilon$ are independent random variables and  $\esp(\varepsilon)=0$ $\var(\epsilon)=\sigma^ 2<\infty$.

Parametric and and nonparametric regression models are two main branches where different procedures have been proposed to   estimate the regression function $r$. The former usually provide parameter estimators with root-$n$ rates of convergence and predictions easy to interpret, while the latter are more flexible since  only some degree of smoothness for the regression function is assumed. Among the first ones, linear regression is one of the most popular models considered among applied practitioners. However, in some situations, the linear assumption does not hold over the whole support of the covariates. To deal with this problem, threshold regression models are commonly used to model some non--linear relationships between the response and the explanatory variable  by introducing one or more threshold parameters, also known as change points. Compared with nonparametric regression, threshold regression models are a simple but interpretable alternative, allowing at the same time to provide  threshold estimates and inferences.
 Most of the literature consider that, within each interval defined by the thresholds, the regression function follows a parametric model, generally a linear one. Unlike these models, in this paper the regression function is modelled as linear beyond the  threshold, but we do not assume   a specific parametric form for the regression function  for values of the covariates below the change point. Up to our knowledge, this flexible approach  has not been previously considered in the literature. 

Even when, in some papers the word threshold has been  used to describe regression models where there is no effect on the response before of such a value, throughout this paper  no distinction is made between the words threshold or change point to refer to the value where   the regression function changes.

We briefly overview some of the contributions done to estimate the regression function when it shows change points. Both, the selection of topics and given
references, are far from being exhaustive and we meant to highlight the differences between the existent procedures and our approach to the problem. 

Among the pioneering papers in the field,   \citet{sprent1961some} and \citet{bacon1971estimating}   considered a two-phase regression model where the regression function is defined by two straight lines, whose slopes and/or intercepts  change before and after the threshold. The goal in two-phase regression models or bent line regression ones is to estimate both the coefficients of each linear functions and the threshold. When many thresholds are admitted, allowing for an easy and natural way to gain flexibility in modelling, these models are referred to as segmented regression and most of the proposals were given in piecewise polynomial regression functions with known or unknown change-points.   
Segmented regression models are very popular in economy, ecology, and medicine, since they  permit the practitioner to contemplate different regimes in a unique model, making possible at the same time to estimate where the transitions occur.   Their flexibility allowed  applications  in different disciplines, giving rise to many publications where these techniques have successfully been used.  In particular, change point models have provided interesting results modelling Covid data, as shown in \citet{dehning2020inferring, coughlin2021early}, among others.

Different estimation procedures for segmented regression models have been extensively developed and implemented, among others we can mention \citet{muggeo2008segmented, muggeo2017interval, fasola2018heuristic, fong2017chngpt}. We also refer to  \citet{muggeo2003estimating} where a description of  procedures to estimate the change points and the regression coefficients is reviewed.  As in other regression models, large values of the residuals, associated to the so--called vertical outliers, may affect the estimations. Robust proposals based on ranks have   been  recently studied by \citet{Zhanli2017} and \citet{shi2020robust} for bent line regression and piecewise linear regression models with multiple change points, respectively.  Since the first works, presented in \citet{hinkley1969inference} and \citet{mczgee1970piecewise}, segmented regression models have been extended to many other scenarios and constitute an active area of research. For instance,  \citet{muggeo2003estimating} and \citet{liang2008estimating} considered generalized linear segmented models where the natural parameter of the conditional distribution of the response given the explanatory variable is modelled by   two straight lines. More general change-point models have received  attention from the statistical community over the last decades. \citet{khodadadi2008change} include a detailed list of papers related to change-points, most of them dealing with situations where the regression is linear in each interval or it is assumed to be 0 before the threshold, meaning that the covariate has no effect on the response  if the threshold limit is not exceeded.   Change points in time series analysis were studied, among others, in \citet{chan:etal:2015} who considered a threshold autoregressive   model with multiple-regimes and introduced a procedure penalizing the autoregressive parameter differences through a LASSO penalty.  Threshold models were also studied in high--dimensional settings where variable selection is an important issue, see \citet{Lee:etal:2020} who assumes that the regression is linear before and after the change point and penalizes the regression parameters with an adaptive LASSO penalty. 

%
%

As mentioned above, in this paper we seek for the  threshold above which the regression function behaves linearly. Unlike other models considered in the literature, we do not assume a specific form for the regression function for values smaller than the threshold,   as   done when modelling through   piecewise polynomials or bent regression lines. It is worth mentioning that   we do not intend to provide estimators for the regression function $r(x)$ on the whole support   $[a,b]$  of the covariates, since the goal is to predict responses for large values of the covariates, where it can be assumed that the linear regression model provides a proper fit.  In this sense, our approach is more semiparametric than parametric, since it combines a region, let us say $(a, \umbral_0)$,  where no assumptions are made on the shape of the regression function and an interval $[\umbral_0, b)$ where it is known to be linear. More precisely, for values of the covariates smaller than $\umbral_0$,  we do not assume that $r(x)$ is known up to a finite number of parameters to be estimated as in the papers mentioned above. For that reason and due to their flexibility, we call the model to be introduced in Section \ref{sec:proposal}   \lq\lq\textsl{threshold semiparametric regression  model}\rq\rq.

Our motivating example consists of the \texttt{airquality} real data set available in \texttt{R}. The data set corresponds to 
153 daily air quality measurements in the New York region between May and September, 1973 \citep[see][]{Chambers:etal:1983}. The interest is
in explaining mean Ozone concentration (``Ozone'', measured in ppb) as a function of the wind speed (``Wind'', in mph). \citet{cleveland1985} considered the 111 cases that do not contain missing observations and found that the relationship between ozone concentration and wind speed is non-linear, with higher wind speeds associated to lower Ozone concentrations due to the ventilation that   higher speeds bring. We add to his analysis the flexibility of our approach, while maintaining a parametric framework for large values of the covariates. Figure \ref{fig:ozone-lowess} displays the observations together to a nonparametric regression estimator obtained using local polynomials through the function \texttt{lowess} in \texttt{R}, see \citet{cleveland:1979,cleveland:1981}.

\begin{figure}[ht!]
\centering
\includegraphics[width=2.2in]{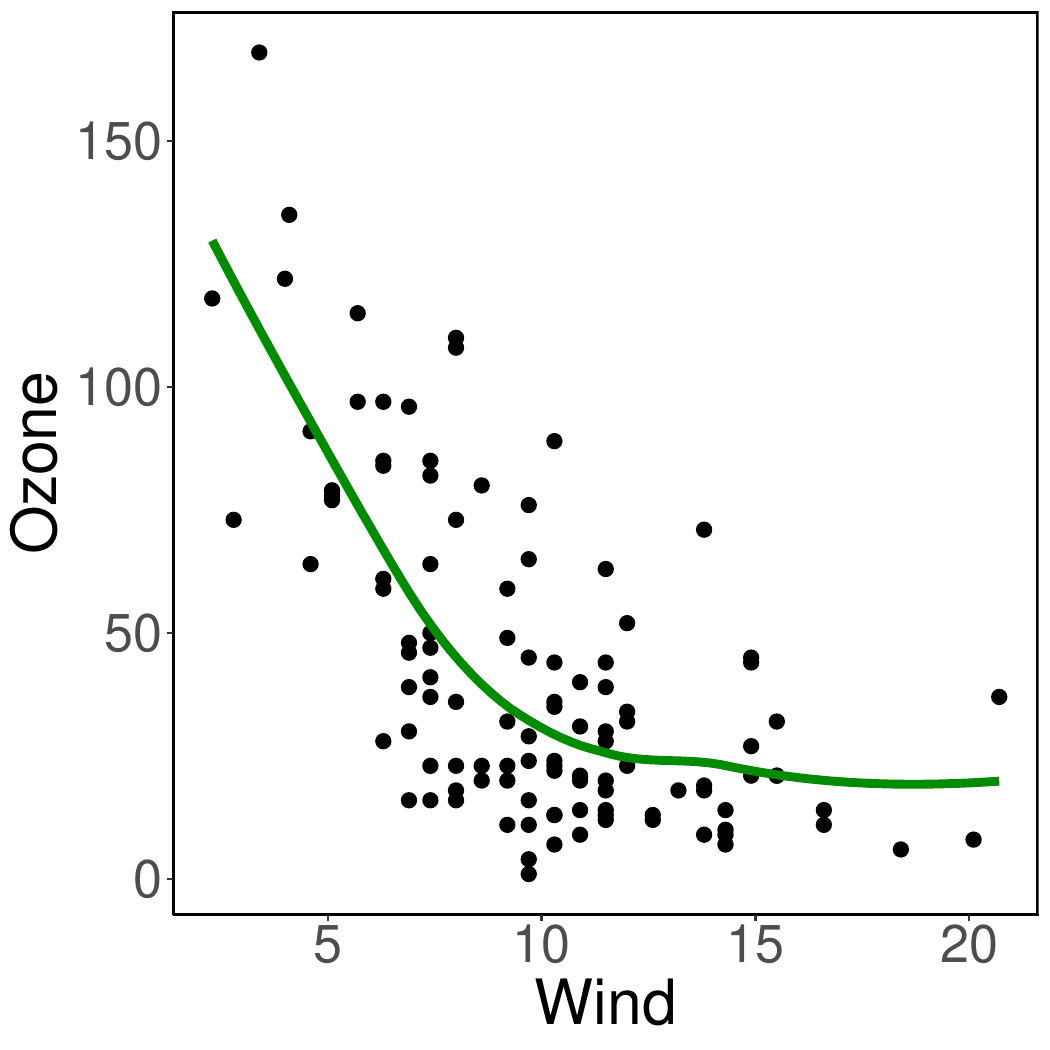} 
\vskip-0.1in
\caption{\small\label{fig:ozone-lowess} Air quality data (black points) and \texttt{lowess} fit. }
\end{figure}

The local smoother  in Figure \ref{fig:ozone-lowess} suggests that the linear fit is not appropriate all over the support of the covariates, but it may be adequate for large  values of the wind speed. In particular, the fit for speeds smaller than 10.5 seems far from a linear one. This motivates the need of providing appropriate estimators of the wind speed from which ozone concentration decreases linearly with wind speed to facilitate predictions.  Our approach through threshold semiparametric regression models suits this purpose, since it postulates that the regression function is linear beyond a threshold $\umbral_0$ without assuming any parametric form before the change point.

The procedure considered in this paper can easily be adapted to the situation where the model may be assumed to be linear up to an unknown point. This is the case in spectroscopy analysis, where the aim  is the determination of the concentration of a compound in a liquid sample, through the measurement of the light absorbed by the solution upon illumination with a lamp.  The degree of light absorbed, or absorbance, is related to the concentration of the compound. The Beer-Lambert law establishes that there is a linear relationship between the absorbance and the concentration. However, it is well known that deviations affect the Beer-Lambert equation at high concentrations. The region where these deviations occur are not always evident and depend on the nature of the  solution (compound and solvent) and the experimental setting. The identification of the interval with upper limit $\umbral_0$, where the linear relation may be assumed is therefore very important for an accurate quantification and it is usually determined graphically from a calibration plot. The  \textsl{threshold semiparametric regression model}  and the penalized threshold estimator to be considered provides a less visual dependent tool to estimate $\umbral_0$.

The \textsl{threshold regression semiparametric model}  provides  a  simple and easily interpretable model, since it  allows to obtain the threshold above which the regression function may assumed to be linear. In both examples described above obtaining this unknown change--point  becomes crucial.
To provide a consistent estimator of the threshold, we introduce a penalty function which penalizes large values of the threshold. This is another of the novelties of the paper, since up to our knowledge, penalties have only been considered to penalize the linear regression parameters but not   the change point candidates.   

In the rest of the paper we will introduce the model, the estimators of the threshold $\umbral_0$ and we will study
some of its properties. The paper is organized as follows. Section \ref{sec:modelo} describes the \textsl{threshold regression semiparametric model}  and provides a characterization of the threshold that will be a key point to define the estimators which are introduced in Section \ref{sec:estimators}. Section \ref{sec:asymptotics} summarize the asymptotic properties of the proposed estimators including  consistency results for the estimators of the threshold and of the intercept and slope  regression function assumed beyond $\umbral_0$. Asymptotic normality results are also derived for the   linear regression coefficient estimators obtained for values larger than $\widehat{\umbral}+\psi$, where $\widehat{\umbral}$ is the estimated threshold and $\psi>0$. Section \ref{sec:monte} reports the results of a Monte Carlo study  conducted to examine the small sample properties of the proposed procedure for different sample sizes according to the penalization constant, the values of the threshold and the regression smoothness. The usefulness of the proposed methodology is illustrated in Section
\ref{sec:realdata} on a real data set, while some final comments are presented in Section \ref{sec:coments}. All proofs are relegated to the Appendix.

\section{The threshold  and its estimation}\label{sec:proposal}
\subsection{The threshold regression semiparametric model}\label{sec:modelo}
Let $X$ and $Y$ be squared integrable  random variables taking values in $\real$, satisfying the regression model 
\begin{equation}
\label{eq:regression}
Y=r(X)+\varepsilon\;,  
\end{equation}
where $X$ and $\varepsilon$ are independent random variables and the error term $\varepsilon$ has mean zero and finite variance $\sigma^ 2$. Denote $\itS$ the support of the random variable  $X$ and assume that $\itS=[a,b]$, i.e., $\itS$ is a connected set, where $a$ and $b$ may be infinite.

In this paper, we deal with the situation where the regression function $r(x)$
is linear for values $x$ of the covariate large enough; namely,  $r(x)=\alpha+\beta \,x$ for $x> \umbral$, $ x\in \itS$.   Let $\umbral_0$ denotes   the smallest value of $\umbral$ above which $r(x)$ is linear. More precisely, define the threshold $\umbral_0$ as  
\begin{equation}
	\label{eq:u_0}
	\umbral_0=\inf\{\umbral\in \itS: \hbox{$u<b$ and for some} \; \alpha, \beta\in \mathbb R \quad \prob(r(X)=\alpha+\beta\, X\mid X>\umbral)=1 \}=\inf \itU\;. 
\end{equation}
If $\itU$ is empty,  then $\umbral_0=+\infty$. However,  assumption \ref{ass:u0finito} below avoids this case. {The possible values of $\umbral$ in \eqref{eq:u_0} are restricted to be in the support of the covariate to avoid an spurious threshold in the situation where the regression function is linear all over the support and $a\ne -\infty$. In such a case, with the above definition $\umbral_0=a$, while if the set $\itU$   allows for any $\umbral\in \real$, the infimum of $\itU$ would have been $-\infty$.}

It is also worth mentioning that if $\widetilde{\umbral}\in \itU$, then any $\umbral$ such that  $\umbral> \widetilde{\umbral}$ is also an element of $\itU$ and the regression function related to $\umbral$ will have the same coefficients as the one defined by $\widetilde{\umbral}$.
 This fact implies that, when $\itU$ is not empty,  whether $\umbral_0=-\infty$ or $\umbral_0$ belongs to $\itU$ meaning that the infimum in \eqref{eq:u_0} is indeed a minimum. {For covariates with unbounded support, the situation $\umbral_0=-\infty$ arises when the regression model is linear, i.e., when there exists $\alpha_0,\beta_0\in \real$ such that $\prob(r(X)=\alpha_0+\beta_0\,X)=1$. When  $\umbral_0\ne -\infty$, taking into account that $\umbral_0$ belongs to $\itU$, we also denote  $\alpha_0,\beta_0\in \real$ the values such that $\prob(r(X)=\alpha_0+\beta_0\, X\mid X>\umbral_0)=1$, implying that $\prob(r(X)=\alpha_0+\beta_0\, X\mid X>\umbral)=1$, for any { $\umbral\ge \umbral_0$.}}

The aim of this work  is to estimate the threshold  $\umbral_0$ when it is finite  on the basis of an i.i.d.  sample $ {\mathcal{Z}}_n=\{Z_i=(X_i,Y_i), 1\leq i\leq n\}$,  distributed as $(X,Y)$, following a model selection approach.
To do so, we first define a function allowing to characterize the threshold  $\umbral_0$.  In the sequel, we assume that $\umbral_0<+\infty$ and that the  distribution function $F_X$ of $X$ is continuous. In this way, probabilities and expectations conditioned on $\{X>\umbral\}$ or on $\{X\geq \umbral\}$ agree. Furthermore, to simplify the notation, when  $\umbral_0=-\infty$  conditioning on $\{X>\umbral_0\}$ or on  $\{X\geq \umbral_0\}$ means that  no conditions are imposed when conditioning. 

We first begin by defining the best linear predictor of the responses when considering only large values of the covariates. More precisely, for each $\umbral\in \real$ with $\umbral<b$,  define $(\alpha_{\umbral}, \beta_{\umbral})$ as the best coefficients to linearly predict  $Y$ based on $X$, when $X\geq \umbral$, that is, 
\begin{equation}
\label{eq:abu}
(\alpha_{\umbral}, \beta_{\umbral})=\argmin_{\alpha, \beta} \;\esp\left[\left\{Y-(\alpha+\beta \,X)\right \}^2\mid X\geq \umbral\right]\,.
\end{equation}
{Some facts should be highlighted regarding the coefficients $(\alpha_{\umbral}, \beta_{\umbral})$. When $\umbral_0\ne -\infty$, given   $\umbral\geq \umbral_0$, under \ref{ass:momento} below, we have that  $(\alpha_{\umbral}, \beta_{\umbral})$ is constant and equal  to $(\alpha_0,\beta_0)$. Moreover,   when $\umbral_0=-\infty$, we also have that $(\alpha_0,\beta_0)=(\alpha_{\umbral}, \beta_{\umbral})$ for any $\umbral<b$, since  $\prob(r(X)=\alpha_0+\beta_0\,X)=1$.}

{To show the first assertion, recall that, when $\umbral_0\ne -\infty$,   $\prob(r(X)=\alpha_0+\beta_0 \,X\mid X\ge\umbral)=1$,  for any $\umbral\geq \umbral_0$, since we are assuming that    $F_X$  is continuous. Hence,  when $\esp(\epsilon)=0$ as stated in \ref{ass:momento}, the independence between the errors and the covariates entail that 
\begin{align*}
\esp\left[\left\{Y-(\alpha+\beta \,X)\right \}^2\mid X\geq \umbral\right] &= \esp\left[\left\{r(X)-(\alpha+\beta \,X)\right \}^2\mid X\geq \umbral\right]+ \esp(\epsilon^2)\\
&= \esp\left[\left\{\alpha_0+\beta_0 \,X-(\alpha+\beta \,X)\right \}^2\mid X\geq \umbral\right]+ \esp(\epsilon^2)\,.
\end{align*} 
 Therefore,    for $\umbral\geq \umbral_0$,  the minimum of $\esp\left[\left\{Y-(\alpha+\beta \,X)\right \}^2\mid X\geq \umbral\right]$ is attained at $(\alpha_{\umbral}, \beta_{\umbral})$ when  $ \prob\left((\alpha_0- \alpha_{\umbral})+(\beta_0-\beta_{\umbral} ) \,X =0\mid X\geq \umbral\right)=1$. The continuity  of $F_X$, allow to conclude  that   $(\alpha_{\umbral}, \beta_{\umbral})=(\alpha_0,\beta_0)$, as desired.}

{To characterize the threshold $\umbral_0$, consider   the loss function   defined, for $\umbral<b$, as}
\begin{equation}
\label{eq:ell_pob}
\ell(u)=\esp [\{Y-(\alpha_{\umbral} +\beta_{\umbral} X)\}^2\mid X\geq  \umbral)]=\frac{\esp\left[\{Y-(\alpha_{\umbral} +\beta_{\umbral} X)\}^2\indica_{\{X\geq \umbral\}}\right]}{\prob(X\geq \umbral)}.
\end{equation}
It is worth noticing that $\ell(u)$ depends on $F$,  the joint distribution of $Z=(X, Y)$, i.e, we should  have used $\ell(u, F)$ to reinforce this dependence in lieu of $\ell(u)$, but we have decided to omit the dependence on $F$ to simplify the notation. 

One important issue to be considered when defining an estimator is if it is indeed estimating the target quantity, in our case, $(\umbral_0, \alpha_0, \beta_0)$. This property,   known as Fisher--consistency, is usually a first step before obtaining consistency results. When the estimator can be written as a functional  applied to the empirical distribution, this property has an easy representation, see for instance \citet{huber:ronchetti:2009}. 
 Even  in our case, where penalized estimators will be considered, the first step before defining them is to ensure that the function $\ell(\cdot)$ allows to characterize $\umbral_0$; this fact is related to the Fisher--consistency property described above. However, in our situation the function  $\ell(\cdot)$ will not have a unique minimum but will be constant for values larger than $\umbral_0$. Hence, the estimators to be defined below need to take into account this fact and to search for the \lq\lq smallest\rq\rq ~value where the minimum of the empirical counterpart of  $\ell(u)$ is attained. The penalty function to be introduced will be important to achieve this goal.

\vskip0.1in
The following set of assumptions will be needed in the sequel
\begin{enumerate}[label=\textbf{C\arabic*}]
\item\label{ass:momento} The errors and covariates are such that $\esp (X^2) <\infty$,   $\esp ( \varepsilon^2) <\infty$ and  $\esp (\varepsilon)=0$.
 \item\label{ass:dist-continua}  The distribution function $F_X$ of the random variable $X$ is continuous.
 \item\label{ass:u0finito} The threshold $\umbral_0$   defined  in \eqref{eq:u_0} is finite, i.e., $\umbral_0<+\infty$.

  \item\label{ass:u0mayorquea} The threshold $\umbral_0$   defined  in \eqref{eq:u_0} is such that $\umbral_0\ne a$, that is, for any $\alpha, \beta\in \real$,  $\prob(r(X)=\alpha+\beta\,X)<1$.
 \end{enumerate}

 Conditions \textbf{C3} and \textbf{C4} establish the semiparametric nature of the model  postulated in this paper: above the threshold $\umbral_0$ the regression is assumed to be linear, but  below it, no specific form is assumed for the relation between the response and the covariate.

Lemma \ref{lema-basico} states the characterization of the threshold $\umbral_0$ in terms of the loss function $\ell(u)$. 

\vskip0.1in
\begin{lemma}\label{lema-basico}
Let $(Y,X)$ be a random vector satisfying \eqref{eq:regression} and let $\umbral_0$ be defined as in \eqref{eq:u_0}. Assume that  \ref{ass:momento} to \ref{ass:u0finito}  hold. Then, 
\begin{enumerate}
\item[a)] The function  $\ell:(-\infty,b)\to \real$ defined by \eqref{eq:ell_pob} is continuous. 
\item[b)] Define $\ell(\umbral_0)=\esp(\varepsilon^2)$ when $\umbral_0=-\infty$. Then, 
\begin{equation}
\label{eq:l_minimiza}
\ell(\umbral_0)=\ell(\umbral)\; \hbox{for all $\umbral$ such that $\umbral_0\leq \umbral<b$}.
\end{equation}
\item[c)] If in addition \ref{ass:u0mayorquea} holds, we also have that
\begin{equation}
\label{eq:l_minimiza2}
\ell(\umbral )>\ell(\umbral_0) \hbox{ for all $\umbral <\umbral_0$ }\,.
\end{equation}
\item[d)] If in addition \ref{ass:u0mayorquea} holds, for any $\delta>0$, we have that \begin{equation}
\label{eq:lmayorlu0}
\inf_{u\le \umbral_0-\delta} \ell(u) > \ell(\umbral_0)\,.
\end{equation}
\end{enumerate}
\end{lemma}

\subsection{The estimators}\label{sec:estimators}
The characterization of the threshold $\umbral_0$ presented in equations \eqref{eq:l_minimiza} and \eqref{eq:l_minimiza2} of Lemma \ref{lema-basico}  implies that, for any $\lambda>0$ and any positive and increasing   function $f$, the minimum of $\ell(\umbral)+ \lambda\, f(\umbral)$ on $[\umbral_0, b)$ is achieved when $\umbral=\umbral_0$. This suggests that the threshold can be estimated combining an estimator of the loss function $\ell(\cdot)$ with a penalization term. The behaviour of $\ell(\umbral)+ \lambda\, f(\umbral)$ on $[\umbral_0, b)$ clarifies why the penalization is needed to avoid estimating the threshold  with large values, which  could be  induced as a natural consequence of an over-fitting phenomena.  

Throughout the paper, to shorten the expressions below, we will use a notation closely related to the one used in empirical processes.  Let $(X_i, Y_i)$, $1\le i\le n$, be i.i.d. observations,  $\prob_n$   stands  for the empirical mean operator
$$
\prob_n(g(X,Y))=\frac{1}{n}\sum_{i=1}^n g(X_i,Y_i)\,,
$$ 
which is the empirical counterpart of $\prob(g(X,Y))=\esp (g(X,Y))$.

Given a sample $ \mathcal Z_n=\{Z_i=(X_i,Y_i), 1\leq i\leq n\}$   distributed as  $(X, Y)$, to define an estimator of  $\ell$, we first introduce  estimators
$(\widehat{\alpha}_{\umbral} ,\widehat{\beta}_{\umbral})$ of $(\alpha_{\umbral}, \beta_{\umbral})$ by considering the least squares regression estimators for the regression of $Y$ on $X$ using only those observations with $X_i\geq  \umbral$, that is,
\begin{equation}
\label{eq:alfabetahat}
(\widehat{\alpha}_{\umbral} ,\widehat{\beta}_{\umbral})= \argmin_{\alpha, \beta} \prob_n \left\{\left(Y -  \alpha -\beta  X \right)^2 \indica_{\{X \geq \umbral\}}\right\}\,.
\end{equation} 
 We now consider the  estimator  $\widehat{\ell}(\umbral, {\mathcal{Z}}_n)$  of $\ell(\umbral)$ given by 
\begin{equation}
\label{eq:ell_empirica}
\widehat{\ell} (\umbral, {\mathcal{Z}}_n)=\frac{\prob_n \left[\left\{Y-(\widehat{\alpha}_{\umbral} +\widehat{\beta}_{\umbral} X)\right\}^2 \indica_{\{X\geq \umbral\}}\right]}{\prob_n\left(\indica_{\{X\geq\umbral\}}\right)}\,.
\end{equation}

\begin{remark}
Two facts should be mentioned regarding the estimator $\widehat{\ell}$.
 The first observation  is  related to its numerical evaluation. Notice that,  as a function of $\umbral$,  $\widehat{\ell}(u,{\mathcal{Z}}_n)$ is constant between consecutive  sample predictor values. Therefore, to compute the loss function  it is enough to evaluate $\widehat{\ell}(u,{\mathcal{Z}}_n)$ only on the sample points and not on the entire domain of the distribution, reducing the numerical complexity.  Furthermore, for any $\umbral< \min_{1\le i\le n}X_i$, we have that $\widehat{\ell} (\umbral, {\mathcal{Z}}_n)=\sum_{i=1}^n r_i^2/n$, where $r_i=Y_i-(\widehat{\alpha}  +\widehat{\beta}  X_i)$ and 
  $(\widehat{\alpha} ,\widehat{\beta})$ is the least squares estimators of the coefficients of a simple regression model using the complete sample.

 The second one regards the  proper definition of $\widehat{\ell}$. When estimating the denominator 
$\prob(X\geq\umbral)$ in $\ell$ by its empirical counterpart $\prob_n\left(\indica_{\{X\geq \umbral\}}\right)=\#\{X_i\geq \umbral\}/n$, we have to take into account that some fraction of points  lying at the right-hand side of $u$ is needed to guarantee that $\prob_n(\indica_{\{X\geq \umbral\}})>0$. Hence, only values of $u$ smaller than the maximum of the sample are allowed. This is similar to the fact that  $\ell(u)$ is finite for values of $u$ smaller than $b$. However, to define a consistent estimator of the threshold using the characterization provided by \eqref{eq:l_minimiza}, we will need to upper bound the possible sets of values for $u$. More precisely, we will assume below that for some $0<\eta_1<1$, the threshold $\umbral_0$ is small than the $1-\eta_1$ quantile  of the covariates distribution, $\gamma=F_X^{-1}(1-\eta_1)$. Furthermore, we will denote as $\gamma_n$   the corresponding  empirical quantile, that is, we ensure that  $\prob_n(\indica_{\{X\geq \umbral\}})\geq \eta_1$ for all $u\leq \gamma_n$.
 \end{remark}
  
To define the objective function allowing to estimate $\umbral_0$, we   add a penalization term to the empirical loss function $\widehat{\ell}(\umbral, {\mathcal{Z}}_n)$ and define  
\begin{equation}
\label{eq:PL}
\pl(\umbral  , {\mathcal{Z}}_n)=  \widehat{\ell}(\umbral,{\mathcal{Z}}_n)+ \lambda_n f(\umbral)\,, 	
\end{equation}	
where $f$ is a  non--negative and non--decreasing   function. 
The threshold estimator $\widehat{\umbral}_n$ is defined as
\begin{equation}
\label{eq:def_hat_umbral}
\widehat{\umbral}_n=  \argmin_{\umbral:\umbral \leq \gamma_n}\, 	\pl(\umbral  , {\mathcal{Z}}_n). 
\end{equation}
Observe that  $\widehat{\umbral}_n$ satisfies 
\begin{eqnarray*}
	\pl(\widehat{\umbral}_n, {\mathcal{Z}}_n)\; \leq \; \pl(\umbral  , {\mathcal{Z}}_n)\;,\quad\hbox{for all $\umbral \leq \gamma_n$}. 
\end{eqnarray*}

\section{Asymptotic properties}\label{sec:asymptotics}
To derive consistency results for  the threshold estimators $\widehat{\umbral}_n$ and for the linear regression  estimators $ (\widehat{\alpha}_{\widehat{\umbral}},\widehat{\beta}_{\widehat{\umbral}})$, where $(\widehat{\alpha}_{\umbral} ,\widehat{\beta}_{\umbral})$ is defined in \eqref{eq:alfabetahat},
we will need the following additional assumptions.

\begin{enumerate}[label=\textbf{C\arabic*}]
\setcounter{enumi}{4} 
\item\label{ass:cuantil} For some $0<\eta_1<1$, $\gamma=F_X^{-1}(1-\eta_1)> \umbral_0$ and $\gamma_n$ is the $1-\eta_1$ empirical quantile of $X_1,\dots, X_n$.
\item\label{ass:fcreciente}  The function $f:\real \to \real$ is a non--negative, non--decreasing function, strictly increasing in a neighbourhood of  $\umbral_0$.
\item\label{ass:momento4}The covariates and responses are such that $\esp ( X^2 Y^2) <\infty$.
\end{enumerate}

\begin{remark}
As mentioned above assumption \ref{ass:cuantil} prevents the denominator $\prob_n\left(\indica_{\{X\geq \umbral\}}\right)$ to be equal to 0. Clearly, under  \ref{ass:cuantil} assumption \ref{ass:u0finito} holds, for that reason in the results   below we omit assumption \ref{ass:u0finito} in their statements. It is also worth mentioning that  \ref{ass:u0finito} entails that $\prob(r(X)=\alpha_0+\beta_0\, X\,\mid\, X\geq\umbral_0)=1$, moreover, $\prob(r(X)=\alpha_0+\beta_0\, X\,\mid\, X\geq\umbral)=1$ for any $\umbral\ge \umbral_0$.

Assumption \ref{ass:fcreciente}  holds when $f$ is strictly increasing on the support $\itS$ of $X$. One possible choice for $f$ is to consider $f(u)=\arctan(u)$. When $a\ne -\infty$, we can also select $f(u)= u-a$ for $u\ge a$ and $0$ otherwise, that corresponds to the penalizing function selected in our numerical study and in {the real data analysis.}

Condition  \ref{ass:momento} is a standard assumption when considering least squares regression estimators. Assumption  \ref{ass:momento4} is needed to ensure that $\widehat{\beta}_{\umbral}-  {\beta}_{\umbral}$ has a root-$n$ rate of convergence, uniformly on $\umbral\le F_{X}^{-1}(1-\eta)$, with $\eta>0$, as stated in Lemma \ref{lema:todo}. Note that for values of $\umbral<  \umbral_0$, the regression function is not linear and for that reason a stronger moment requirement than the usual for linear regression models is needed.  From the Cauchy Schwartz inequality we obtain that assumption \ref{ass:momento4} holds when  $\esp (X^4) <\infty$ and   $\esp ( Y^4) <\infty$. 

It is also worth mentioning that \ref{ass:momento4} is  a consequence of \ref{ass:momento}   when the regression function $r$ is bounded on the support of $X$. For instance, if $\itS$ is bounded and the regression function $r$ is   continuous, assumption \ref{ass:momento} implies assumption \ref{ass:momento4}. Hence, if $r$ is bounded on $\itS$,  \ref{ass:momento4} may be omitted   in the statement of Theorems \ref{teo:tiro_del_final} and \ref{teo:distriasint} and in Corollary \ref{coro:au}.
\end{remark}

 Theorem \ref{teo:tiro_del_final} below states that the least squares  penalized loss  gives  rise to a consistent estimator of the threshold $\umbral_0$. However, consistency results for threshold estimator are not restricted to the specific loss function given by \eqref{eq:ell_pob} and its empirical counterpart given by  \eqref{eq:ell_empirica} if a penalty term is included as in \eqref{eq:PL}. Any continuous loss function characterizing the threshold through \eqref{eq:l_minimiza} and any reasonable  estimator of such a loss function, such as its empirical counterpart, would work as well, as shown in Proposition \ref{prop:consistencia_ver2}.

\begin{proposition}\label{prop:consistencia_ver2}  
Let $(X,Y)$ be a random vector of squared integrable variables  satisfying the regression model \eqref{eq:regression} and let $\umbral_0$ be defined as in \eqref{eq:u_0}.  
Consider  a continuous loss function $\ell:(-\infty,b) \to \real$ satisfying \eqref{eq:l_minimiza}, \eqref{eq:l_minimiza2} and \eqref{eq:lmayorlu0}.
Define $\pl(\umbral  , {\mathcal{Z}}_n)$ as in \eqref{eq:PL}, where the function $f$ satisfies \ref{ass:fcreciente} and  the function $\widehat{\ell}(u, {\mathcal{Z}}_n)$ is such that  
\begin{enumerate}
\item[(a)]  $ \sup_{\{\umbral\colon  \umbral\leq \umbral_0\}} \vert\widehat{\ell}(u, {\mathcal{Z}}_n)-\ell (u)\vert \convprob 0$.
\item[(b)] For some sequence $a_n$ converging to zero and for any $\widetilde{\gamma}=F_X^{-1}(1-\eta)$ with $\eta>0$
$$
\lim_{n\to \infty} \prob\left(\sup_{u\in [\umbral_0,\widetilde{\gamma}]} \vert\widehat{\ell}(u, {\mathcal{Z}}_n)-\ell (u)\vert\leq a_n\right)=1\,.
$$
\end{enumerate}
 Let $\widehat{\umbral}_n$ be defined through \eqref{eq:def_hat_umbral}. Then, under \ref{ass:momento},  \ref{ass:dist-continua} and \ref{ass:cuantil}, if  the sequence  $\lambda_n>0 $ is such that $\lambda_n\to 0$ and $a_n/\lambda_n \to 0$ when $n\to\infty$, we have that $\widehat{\umbral}_n\convprob \umbral_0$.
\end{proposition}

As mentioned above Theorem \ref{teo:tiro_del_final} establishes the consistency of the threshold estimators when using the squared loss function.

\begin{theorem}\label{teo:tiro_del_final}
Assume $\{(X_i, Y_i): i\geq 1\}$ are i.i.d., distributed as  $(X, Y)$ satisfying the regression model \eqref{eq:regression}  and let $\umbral_0$ be defined as in \eqref{eq:u_0}.   Assume that \ref{ass:momento},  \ref{ass:dist-continua}  and \ref{ass:u0mayorquea} to  \ref{ass:momento4}  hold. 
Let $\widehat{\umbral}_n$ be defined as in \eqref{eq:def_hat_umbral}  and  $\widehat{\ell}(\umbral, {\mathcal{Z}}_n)$  defined in \eqref{eq:ell_empirica}. Then, if  $\lambda_n \to 0$,  and there exists $\varphi_n\to +\infty$, such that 
$n^{1/2}\lambda_n/ \varphi_n\to +\infty$, we   have that  $\widehat{\umbral}_n\convprob\umbral_0$.
\end{theorem}
Note that as a consequence of Theorem \ref{teo:tiro_del_final} $\widehat{\umbral}_n\convprob\umbral_0$, when the penalty parameter $\lambda_n$ converges to $0$ with almost root-$n$ rates, as in the following two cases:  $\lambda_n=cn^{-\xi}$ for some $0<\xi<1/2$,  or  $\lambda_n= c\,n^{-1/2}\,\log(n)$, since it is enough to choose $\varphi_n=n^{-\omega}$ with $\xi+\omega<1/2$ or $\varphi_n=\log(\log n)$, respectively.

An important issue is to show that the least squares estimators obtained using only the observations larger the consistent threshold $\widehat{\umbral}_n$ provide consistent estimators of the true parameters $\alpha_0,\beta_0$. Following the notation introduced in \eqref{eq:alfabetahat},  $ \widehat{\alpha}_{\widehat{\umbral}}$ and $\widehat{\beta}_{\widehat{\umbral}}  $ stand for the least squares estimators of the intercept and slope when the sub-sample $\{(X_i, Y_i): X_i\ge \widehat{\umbral}_n\}$ is used in the estimation process.

\begin{corollary}{\label{coro:au}}
Assume $\{(X_i, Y_i): i\geq 1\}$ are i.i.d., distributed as  $(X, Y)$ satisfying the regression model \eqref{eq:regression} and let $\umbral_0$ be defined as in \eqref{eq:u_0}.    Under \ref{ass:momento},  \ref{ass:dist-continua}  and \ref{ass:u0mayorquea} to  \ref{ass:momento4}, we have that $(\widehat{\alpha}_{\widehat{\umbral}},\widehat{\beta}_{\widehat{\umbral}} )\convprob(\alpha_0,\beta_0)$.
\end{corollary}

Theorem \ref{teo:distriasint} provides an asymptotic normality result for the intercept and slope regression coefficients obtained using the observations with $X_i\ge \widehat{\umbral}+\psi$, for any $\psi>0$. More precisely, it shows that given $\psi > 0$, $(\widehat{\alpha}_{\widehat{\umbral}+\psi}, \widehat{\beta}_{\widehat{\umbral}+\psi})$ asymptotically behaves as the estimators computed using the threshold $\umbral_0+\psi$, which are asymptotically normal according to the well known results for weighted least squares regression estimators. First recall that for any $\psi>0$, $\alpha_{\umbral_0+\psi}=\alpha_{\umbral_0}$ and $\beta_{\umbral_0+\psi}=\beta_{\umbral_0}$, so Lemma \ref{lema:todo} and the continuity of $\alpha_{\umbral}$ and $\beta_{\umbral}$ stated in the proof of Lemma \ref{lema-basico} implies that $(\widehat{\alpha}_{\widehat{\umbral}+\psi}, \widehat{\beta}_{\widehat{\umbral}+\psi})\convprob (\alpha_0,\beta_0)$.

\begin{theorem}\label{teo:distriasint}
Assume $\{(X_i, Y_i): i\geq 1\}$ are i.i.d., distributed as  $(X, Y)$ satisfying the regression model \eqref{eq:regression} and let $\umbral_0$ be defined as in \eqref{eq:u_0}.  Assume that \ref{ass:momento},  \ref{ass:dist-continua}  and \ref{ass:u0mayorquea} to  \ref{ass:momento4}   hold and that $\widehat{\umbral}_n\convprob\umbral_0$. Then, for any fixed positive constant $\psi$ we have that
$$\sqrt{n}\{(\widehat{\alpha}_{\widehat{\umbral}+\psi},\widehat{\beta}_{\widehat{\umbral}+\psi} )-(\alpha_0,\beta_0)\}\convdist {\mathcal N}(0, \bSi_{\umbral_0+\psi})\,,$$
where   $\bSi_{\umbral_0+\psi}=\sigma^2 \left\{\esp\left(\indica_{\{X\geq \umbral_0+\psi\}} \widetilde{\bX}\widetilde{\bX}\trasp\right)\right\}^{-1}$ with $\widetilde{\bX}= (1, X)\trasp$. 
\end{theorem}

\section{Monte Carlo study}\label{sec:monte}
In this section, we report the results of a small simulation study conducted to numerically explore the finite sample behaviour of the estimators defined in Section \ref{sec:proposal}. We aim to compare our proposal with the estimators obtained minimizing $\widehat{\ell}(\umbral,{\mathcal{Z}}_n)$ to illustrate that some degree of penalization is needed to obtain a consistent procedure. One of the goals of this numerical study is also to show the effect that the penalization term, the errors scale and the smoothness of the regression function may have on the   threshold estimators. 

The   threshold estimator was   computed  using the definition given through \eqref{eq:def_hat_umbral}, taking as $\gamma_n$  the 95\% empirical quantile of the predictors. In all scenarios,  the penalty parameter  was defined as $\lambda_n = cn^{-0.4}$ for different values of the constant $c> 0$ and the function $f$ equals the identity function in $[0,1]$, i.e., we chose $f(u) = u$ for $u\ge 0$ and $f(u)=0$ for $u<0$. Note that the non--penalized case  corresponds to $c=0$ and is included to show that, in this case,  the threshold  estimator does not converge to the true  regression function threshold. It is worth mentioning that according to Theorem \ref{teo:tiro_del_final}, the penalty parameter $\lambda_n$ should converge to $0$ with a rate slower than $\sqrt{n}$ to ensure consistency, for that reason we have chosen $\lambda_n = cn^{-0.4}$.  

To provide a broad comparison, we considered different regression functions depending on the threshold value $\umbral_0$ and also on a constant $\delta$ which allows us to include functions which are not differentiable  at $\umbral_0$. We also studied different  scenarios according to the errors standard deviations and to the sample size. In all cases, we performed $Nrep=1000$ replications.

To define the regression function, let us denote as $g:[0,1]\to \real$ the non-linear function defined as
$$
g(x)=\left\{
\begin{array}
[c]{lll}%
4x^2(3-4x) &  & \text{if $0\leq x\leq 0.5$}\,,\\
&  & \\
\displaystyle{\frac{4}{3}}\,x\,(4x^2-10x+7)-1 &  & \text{if $0.5\leq x\leq 1$}\,.
\end{array}
\right.
$$
The   family of  regression functions will be labelled according to the threshold value $\umbral_0$ and to the smoothing parameter $\delta\in \real$ involved in its definition, that is, we denote as $r_{\umbral_0,\delta}:[0,1]\to \real$, the function
$$
r_{\umbral_0,\delta}(x)=\left\{
\begin{array}
[c]{lll}%
g(x)  &  & \text{if $0\leq x \leq \umbral_0$}\,,\\
&  & \\
\{g^\prime(\umbral_0)+\delta\}(x-\umbral_0)+g(\umbral_0) &  & \text{if $\umbral_0\leq x\leq 1$}\,.
\end{array}
\right.
$$
It is worth mentioning that regardless of the value of  $\delta$, the function $r_{\umbral_0,\delta}(x)$ is continuous. However, it is only  differentiable at $\umbral_0$ when $\delta=0$. Note that when $\umbral_0=0.5$, $g(\umbral_0)=1$ and $g^\prime(\umbral_0)=0$. Hence for  $\umbral_0=0.5$ and $\delta=0$, the function $r_{\umbral_0,\delta}(x)$ is constant on $[\umbral_0,1]$, while for any $\delta$, on $[0,\umbral_0]$, $r_{\umbral_0,\delta}(x)=4x^2(3-4x)$. 
Figure \ref{fig:r_delta} shows the plots of $g(x)$ and the  regression function $r_{\umbral_0,\delta}(x)$ for different values of $\umbral_0$ and $\delta$. 

\begin{figure}[t]
\begin{center}
\renewcommand{\arraystretch}{0.1}
\newcolumntype{G}{>{\centering\arraybackslash}m{\dimexpr.33\linewidth-1\tabcolsep}}
\begin{tabular}{GGG}
\multicolumn{1}{c}{$g(x)$} & \multicolumn{1}{c}{$\umbral_0=0.5$, $\delta=0$}& \multicolumn{1}{c}{$\umbral_0=0.5$, $\delta=\,-\, 1$}\\[-3ex]
\includegraphics[scale=0.34]{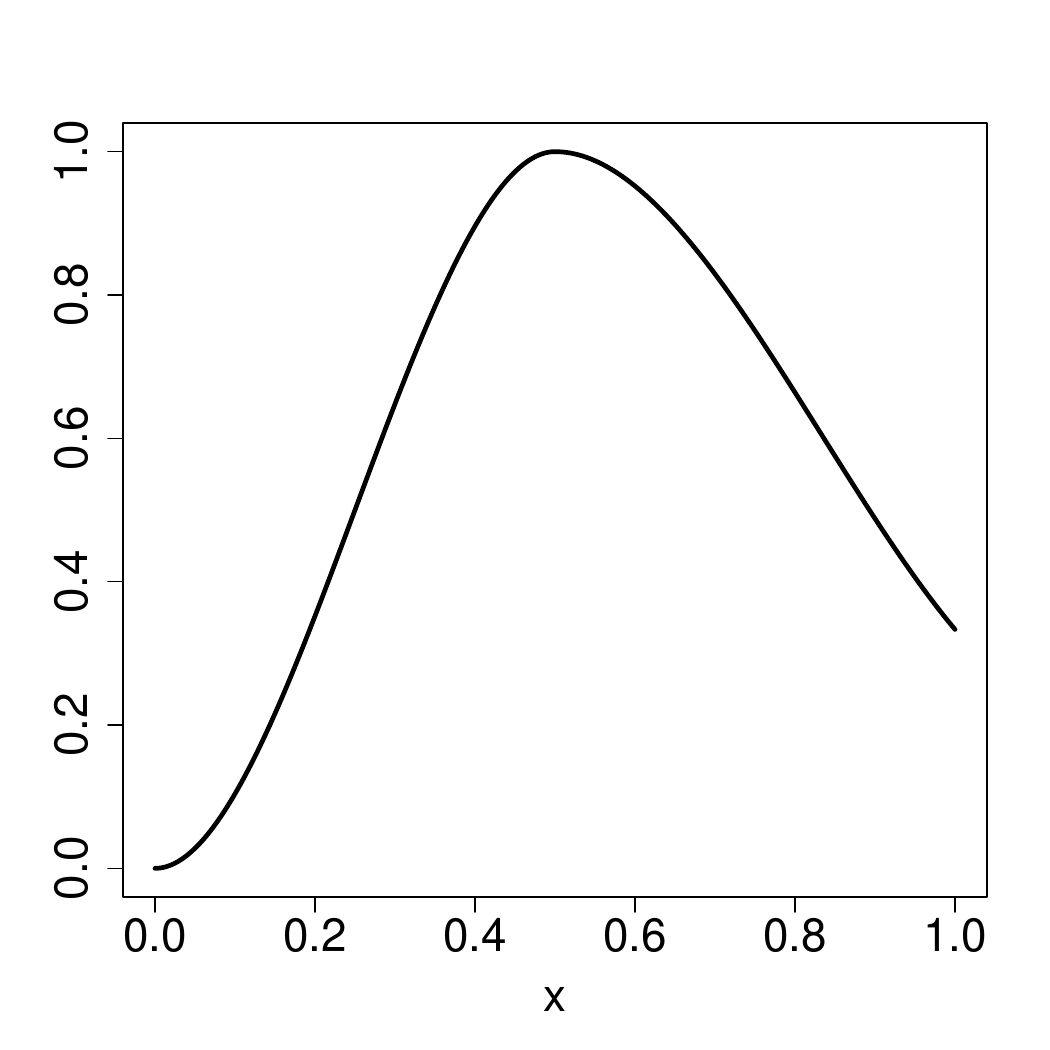}&
\includegraphics[scale=0.34]{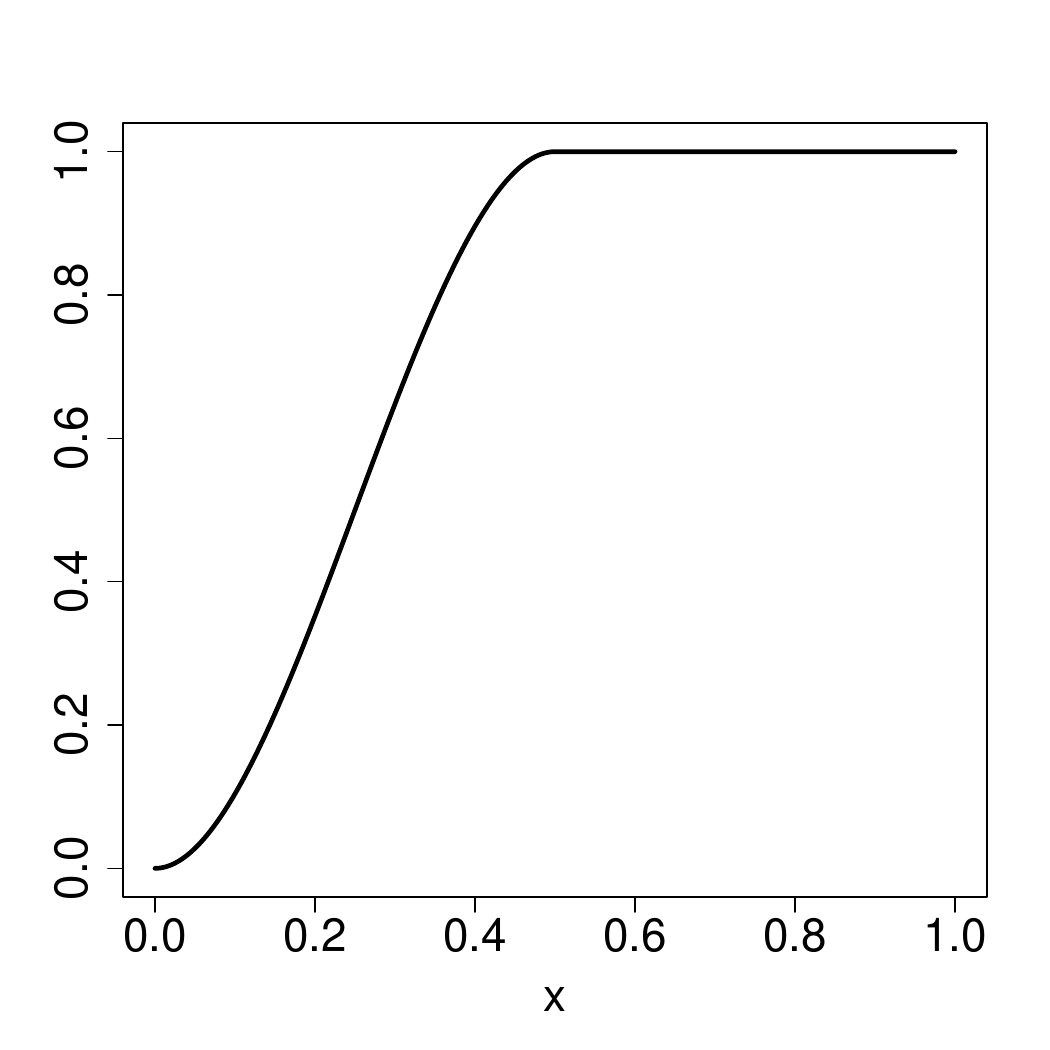}&
\includegraphics[scale=0.34]{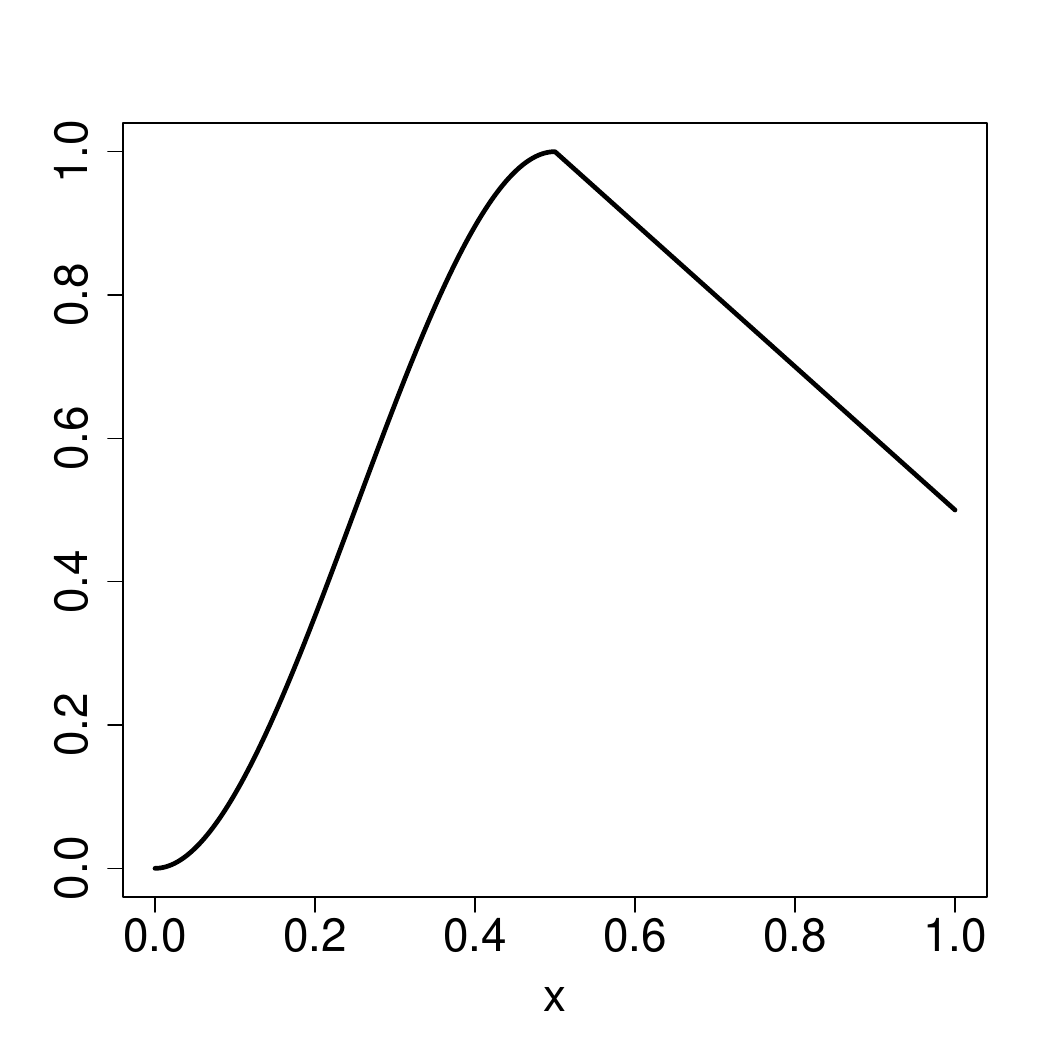}\\
\multicolumn{1}{c}{$\umbral_0=0.75$, $\delta=1$} & \multicolumn{1}{c}{$\umbral_0=0.75$, $\delta=0$}& \multicolumn{1}{c}{$\umbral_0=0.75$, $\delta=\,-\, 1$}\\[-3ex]
\includegraphics[scale=0.34]{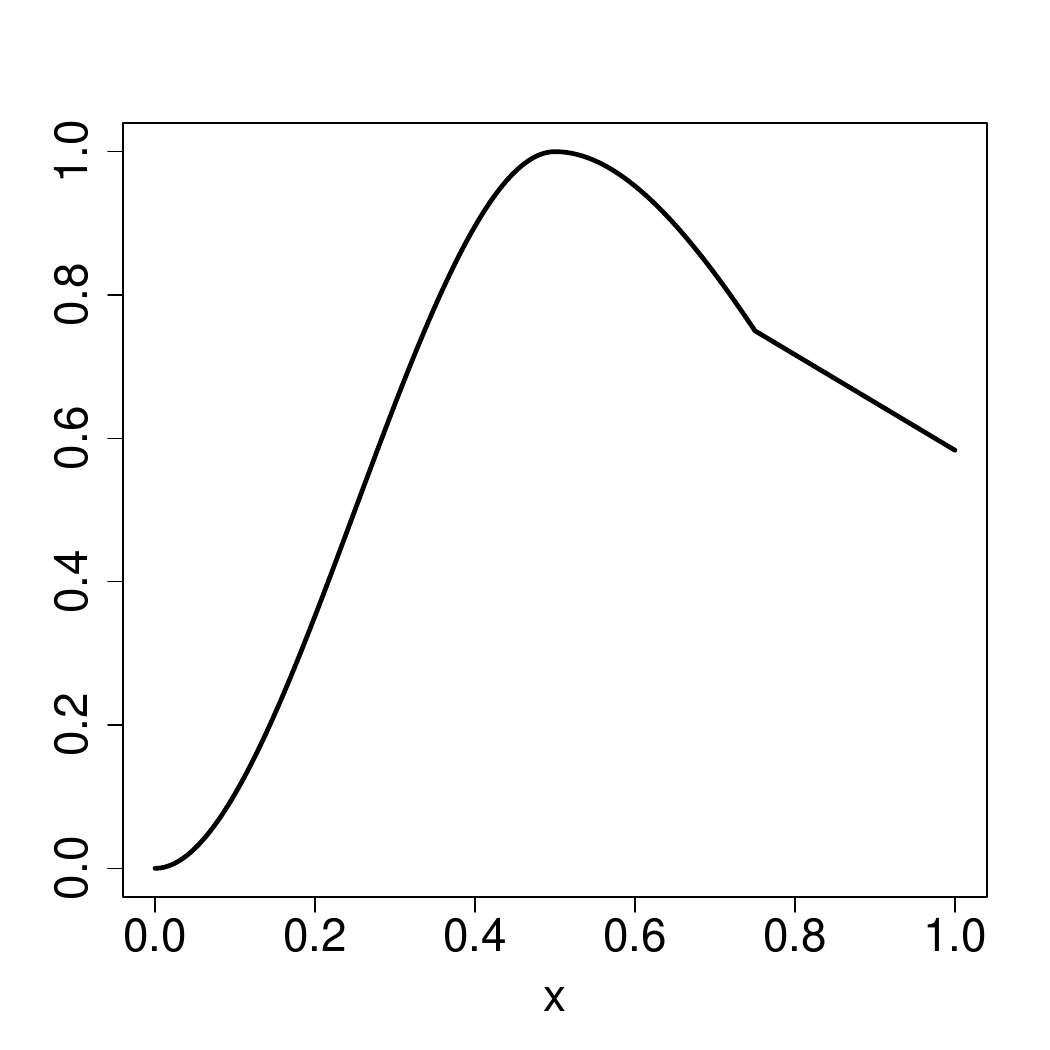}&
\includegraphics[scale=0.34]{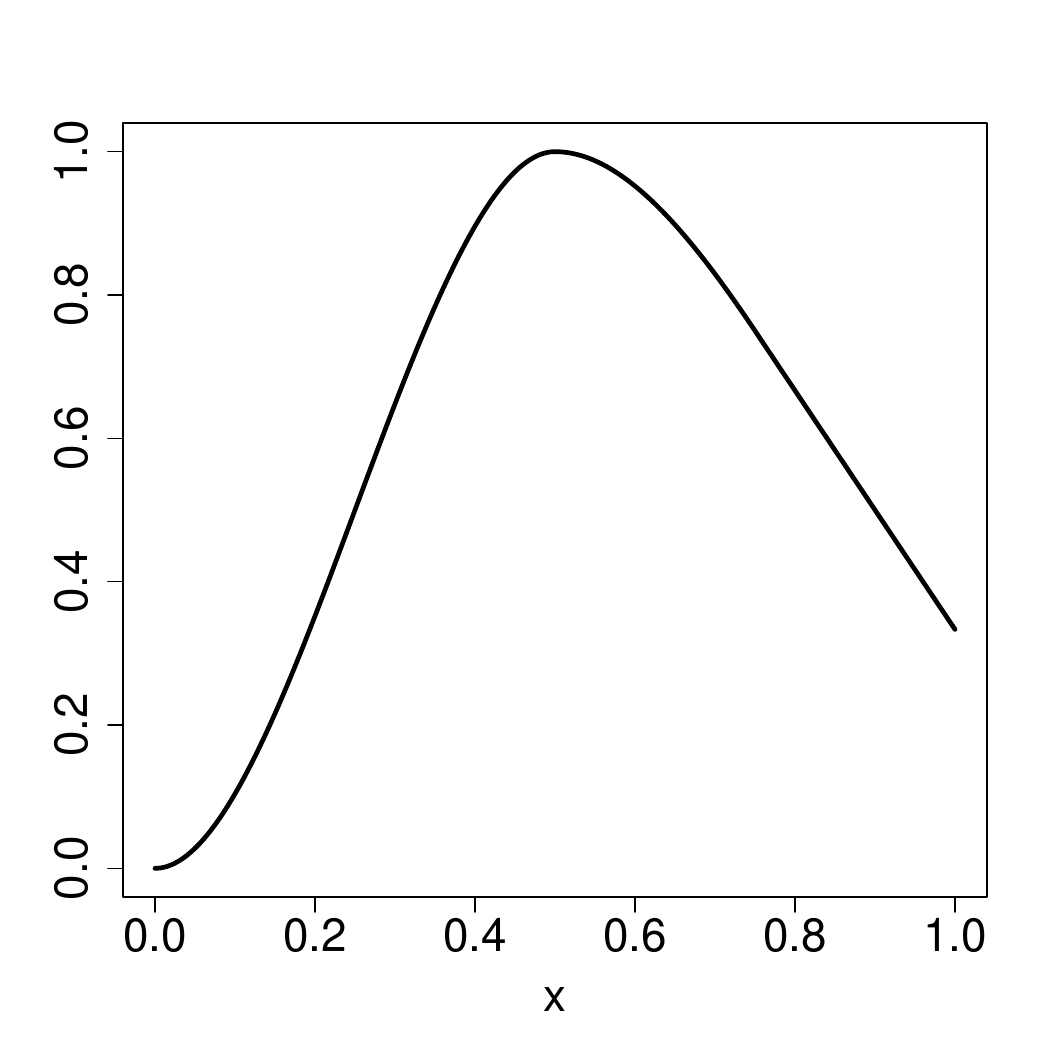}&
\includegraphics[scale=0.34]{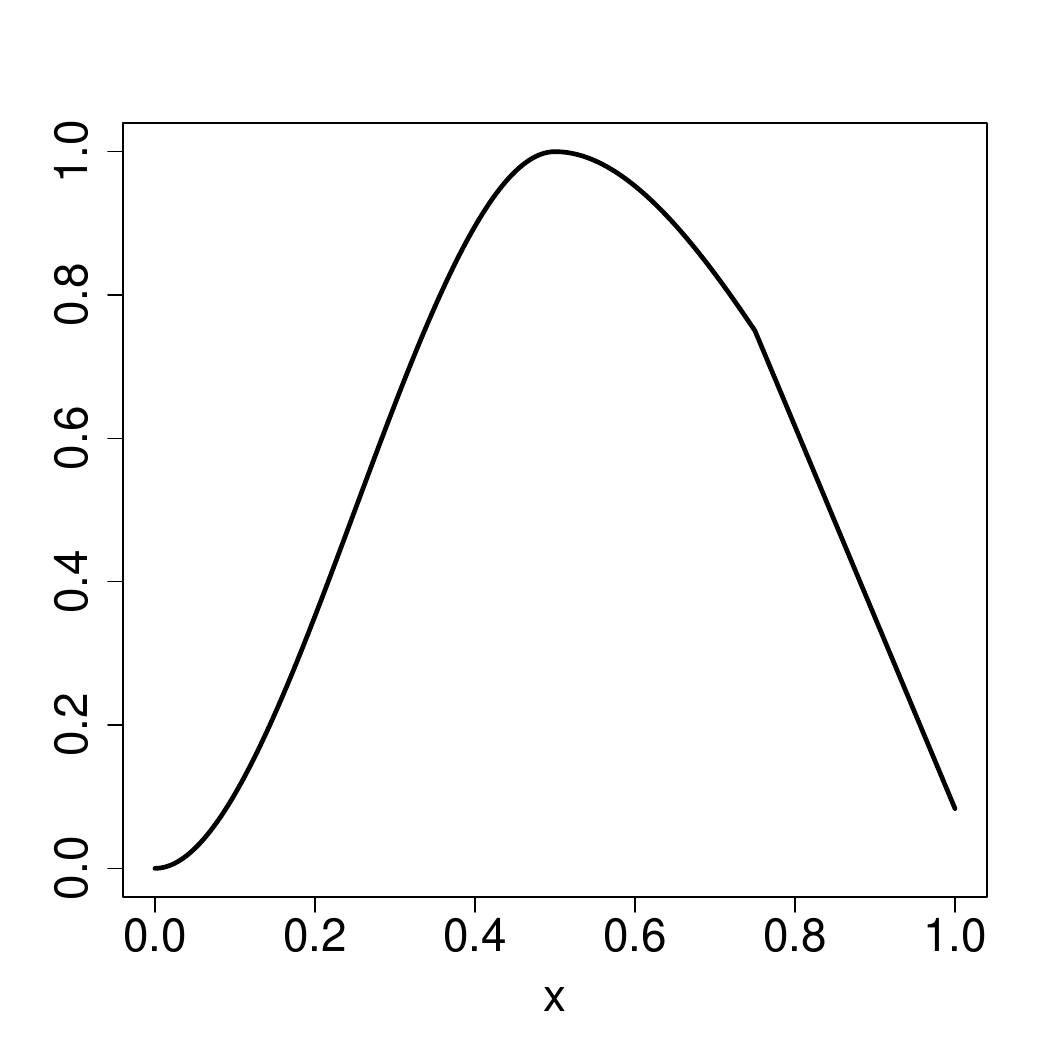}
\end{tabular}
\vskip-0.1in
\caption{\small \label{fig:r_delta}	Plot of the function $g(x)$ (upper left panel) and of the regression function $r_{\umbral_0,\delta}(x)$ for different values of the threshold $\umbral_0$ and smoothing parameter $\delta$.}  
\end{center}
\end{figure}

For each replication, we generate independent samples $(X_i, Y_i)\sim (X,Y)$, $i=1,\dots, n$, where $X\sim {\itU}(0,1)$ and $Y=r_{\umbral_0,\delta}(X)+\varepsilon$, with $\varepsilon\sim {\mathcal N}(0,\sigma^2)$, independent of $X$.   Note that for any value of $\delta$, the  observations satisfy the postulated model \eqref{eq:regression} with $b= g^\prime(\umbral_0)+\delta$, $a= g(\umbral_0)- b\, \umbral_0$ and threshold at  $\umbral_0$ as defined in \eqref{eq:u_0}. Across the different scenarios, we vary the sample sizes  from $n=100$ to $n=2000$ with step=100. Two values for  $\umbral_0$ were selected to illustrate the  performance of the proposed procedure $\umbral_0=0.5$ and $\umbral_0=0.7$ combined with three values for  $\delta$, $\delta= -1,0$ and $1$. In this way we get 
six different scenarios for the regression function $r_{\umbral_0,\delta}$, defined in \eqref{eq:regression}.  Finally,  three values for the errors standard deviation were selected, $\sigma=0.01, 0.05$ and $0.1$. 

To illustrate the structure of the different data sets obtained as $\umbral_0$ and $\delta$ vary, Figure \ref{fig:puntos} displays one of the pairs $(X_i, Y_i)$, $i=1,\dots, n$, obtained  when  $n=50$ together with the regression function $r_{\umbral_0,\delta}$ used to generate them. 
Note that when $\umbral_0=0.75$ the linear relationship becomes more evident  as   the absolute value of $\delta$ increases,. As shown below, this fact will be also confirmed by our numerical results.

\begin{figure}[t]
\begin{center}
\renewcommand{\arraystretch}{0.1}
\newcolumntype{G}{>{\centering\arraybackslash}m{\dimexpr.33\linewidth-1\tabcolsep}}
\begin{tabular}{GGG}
\multicolumn{1}{c}{$\umbral_0=0.5$, $\delta=0$} & \multicolumn{1}{c}{$\umbral_0=0.5$, $\delta=\,-\, 1$}& \multicolumn{1}{c}{$\umbral_0=0.5$, $\delta=1$}\\ 
\includegraphics[scale=0.34]{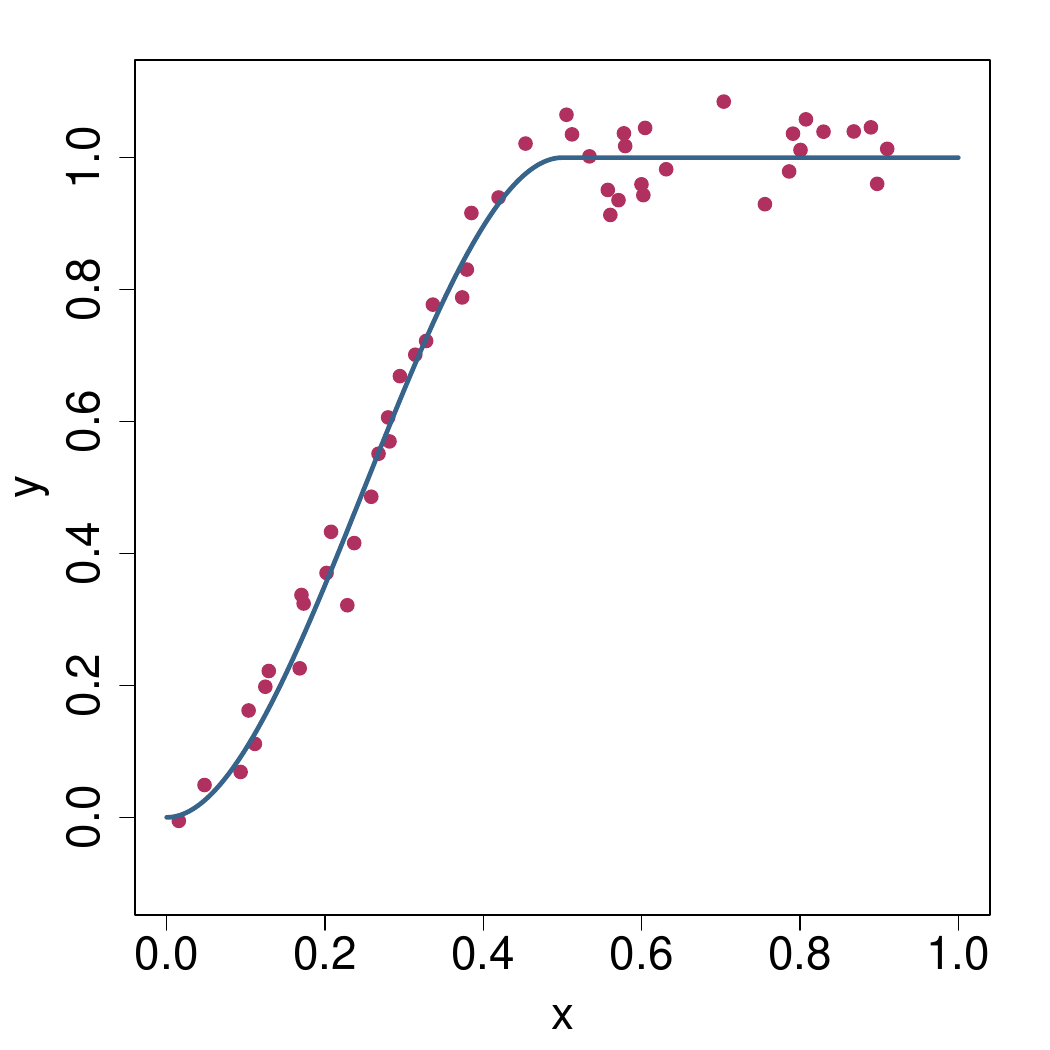}&
\includegraphics[scale=0.34]{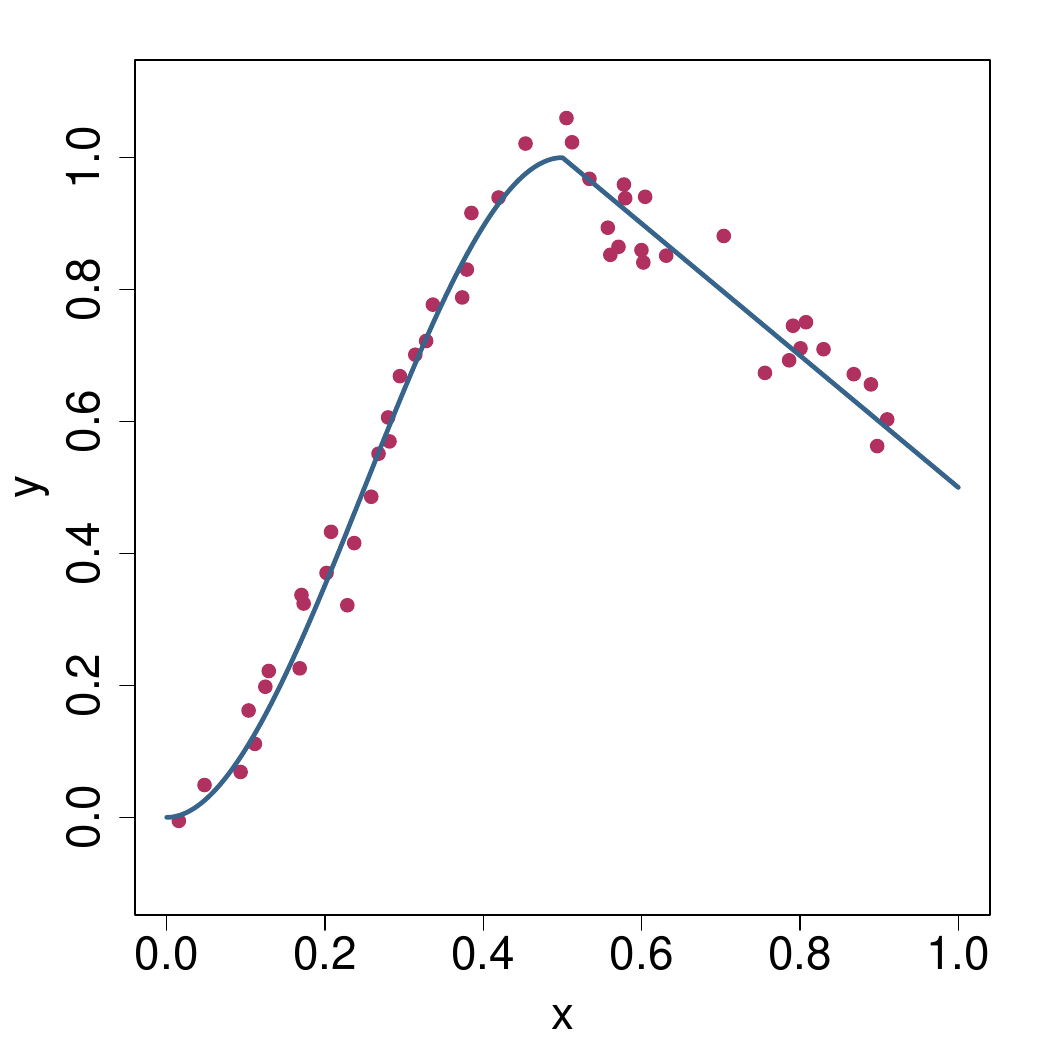}&
\includegraphics[scale=0.34]{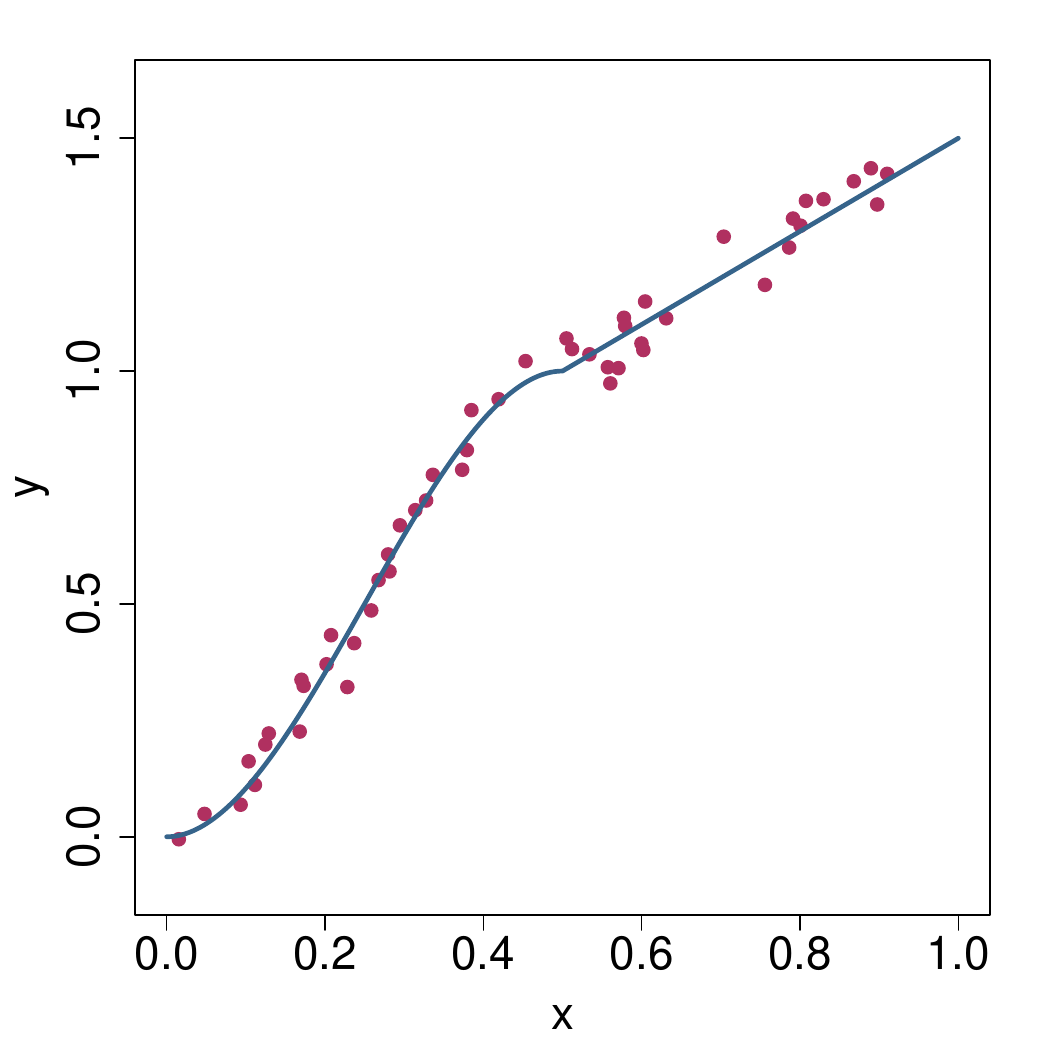}\\[1ex]
\multicolumn{1}{c}{$\umbral_0=0.75$, $\delta=0$} & \multicolumn{1}{c}{$\umbral_0=0.75$, $\delta=\,-\, 1$}& \multicolumn{1}{c}{$\umbral_0=0.75$, $\delta=1$}\\ 
\includegraphics[scale=0.34]{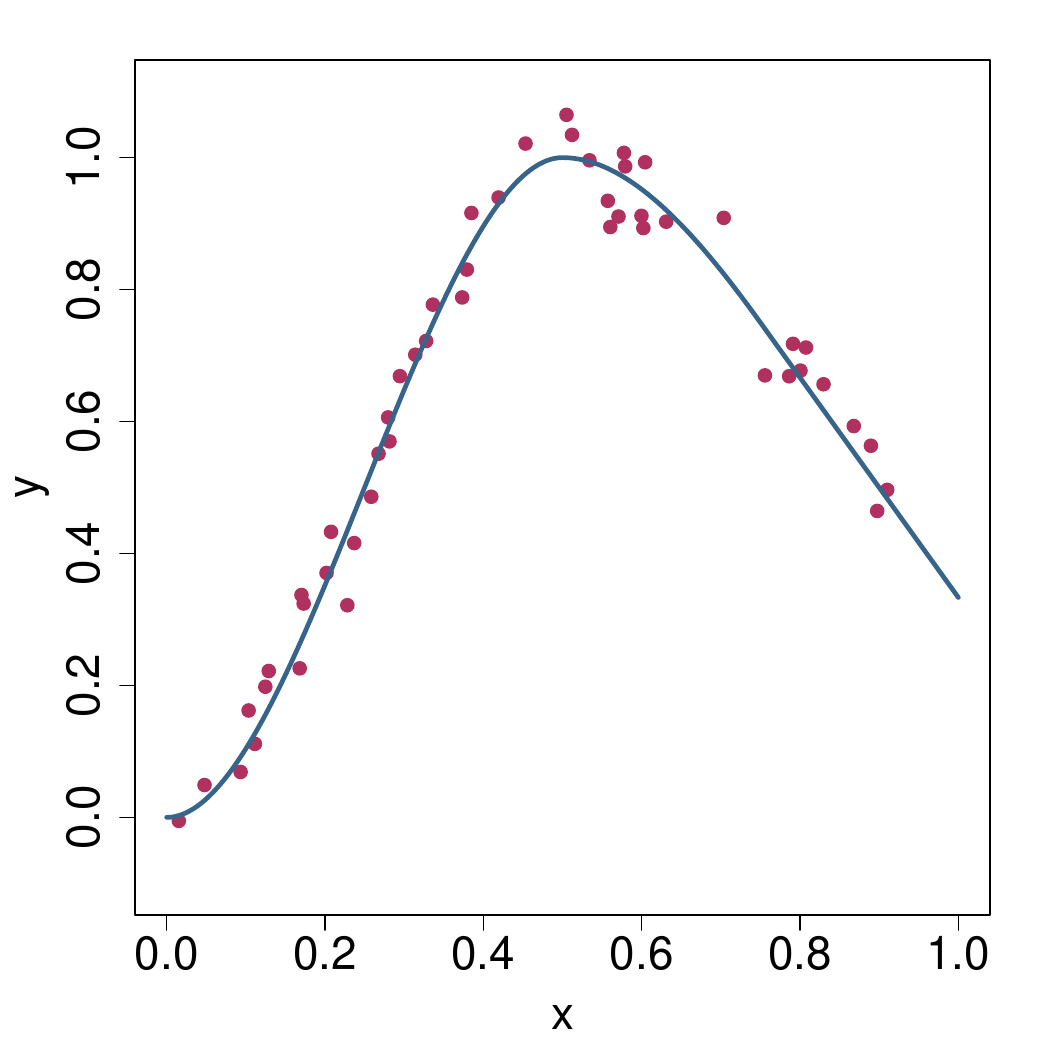}&
\includegraphics[scale=0.34]{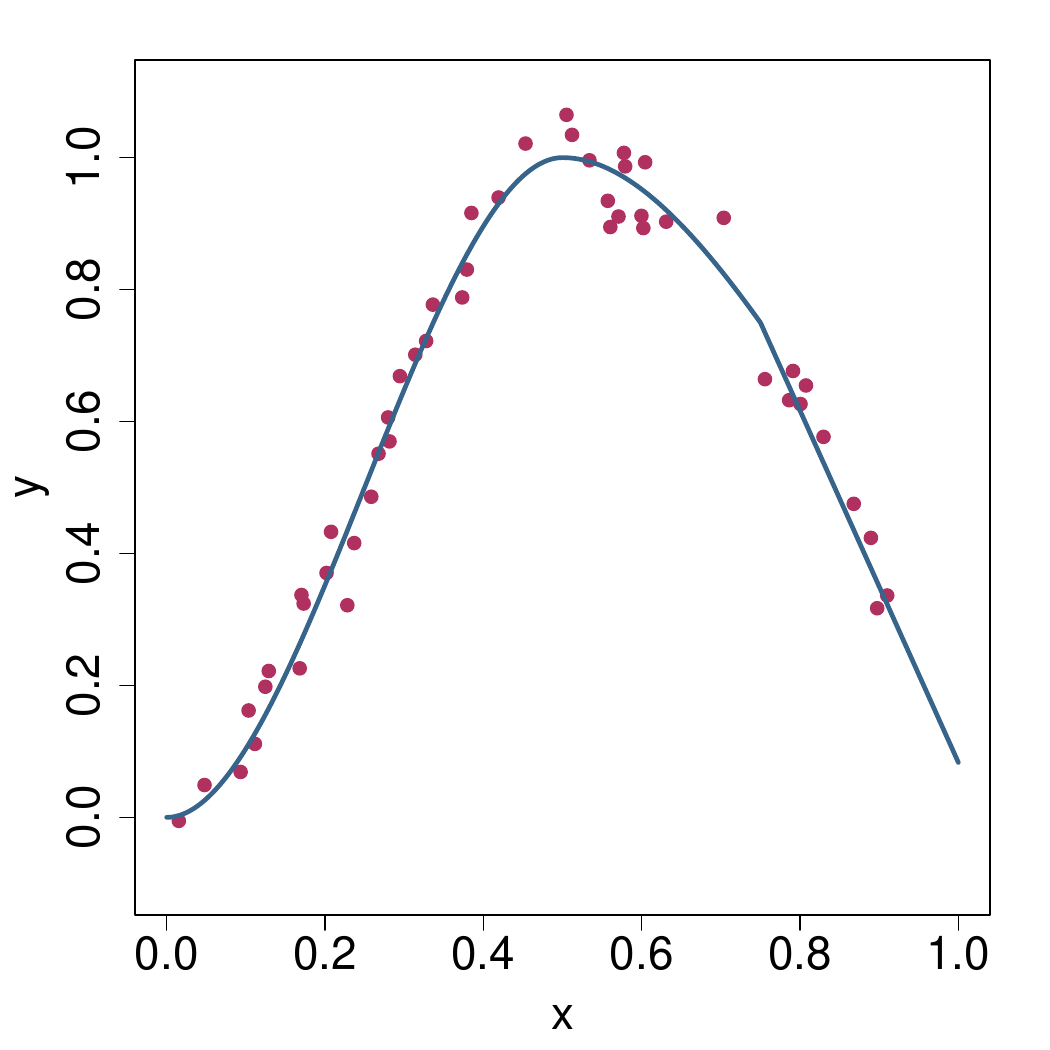}&
\includegraphics[scale=0.34]{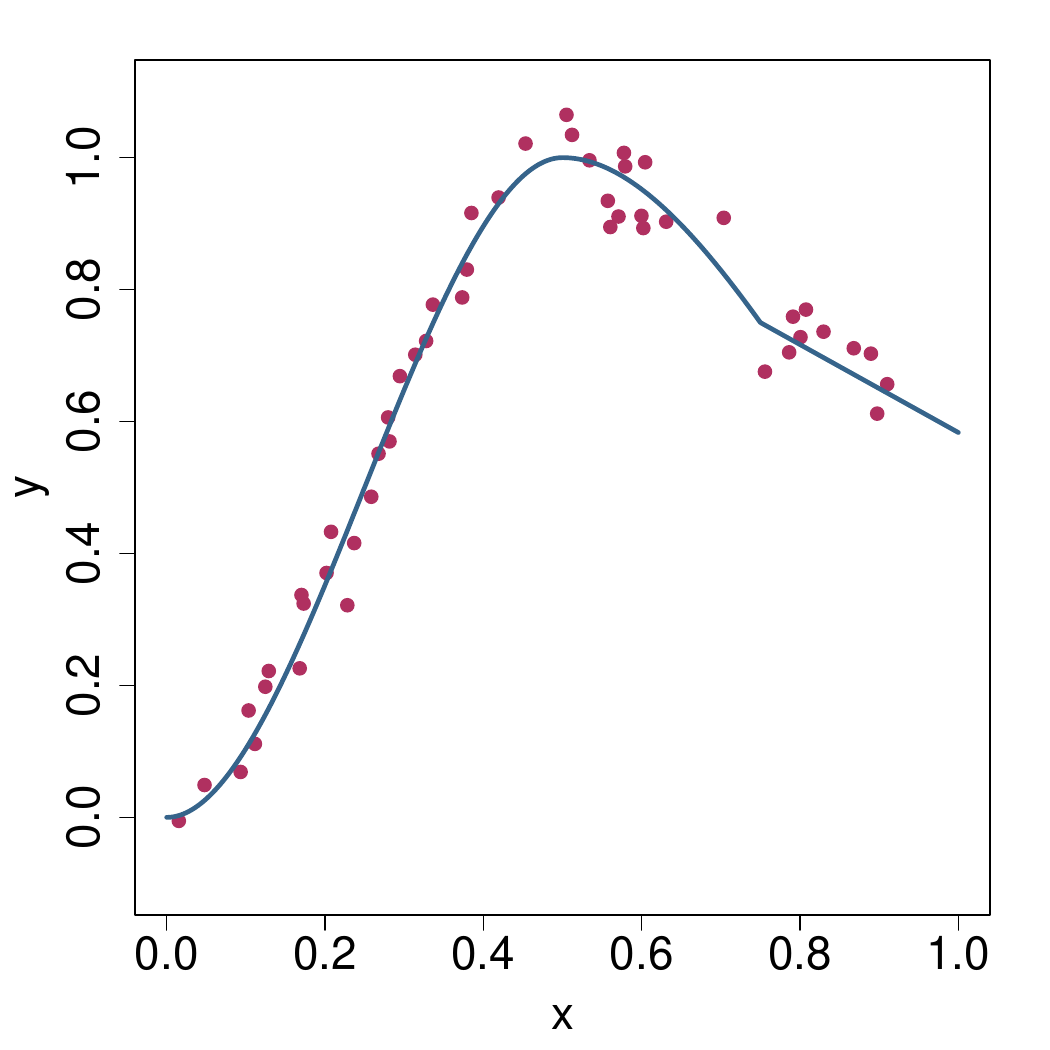}
\end{tabular}
\vskip-0.1in
\caption{\small \label{fig:puntos} Datasets with $n=50$ observations generated through the model \eqref{eq:regression} with different regression functions $r_{\umbral_0,\delta}(x)$ from Figure \ref{fig:r_delta} and $\epsilon\sim\mathcal N(0,0.05^2)$.} 
\end{center}
\end{figure}

To evaluate the performance of the estimators, we computed the empirical mean absolute  error (\textsc{EMAE}) of the estimator over replications defined as   
\begin{equation}
\label{MSE}
\hbox{\textsc{EMAE}}=\frac{1}{Nrep}\sum_{s=1}^{Nrep} |\widehat{\umbral}^{s}_n-\umbral_0|,
\end{equation}
where $\widehat{\umbral}_n^{s}$,  $s=1\ldots, Nrep$, stands for the  estimate obtained at  the $s$-th replication when the sample size equals $n$ and when the regression function equals $r_{\umbral_0,\delta}$. Clearly, the obtained value of  \textsc{EMAE}  varies across  the studied scenarios depending on each combination of   $\umbral_0$, $\delta$, $\sigma$, $n$ and  $c$.

\begin{figure}[ht!]
\begin{center}
\renewcommand{\arraystretch}{0.1}
\newcolumntype{G}{>{\centering\arraybackslash}m{\dimexpr.33\linewidth-1\tabcolsep}}
\begin{tabular}{cc}
\multicolumn{1}{c}{$c=0$} & \multicolumn{1}{c}{$c=0.005$}\\[-1ex]
\includegraphics[scale=0.35]{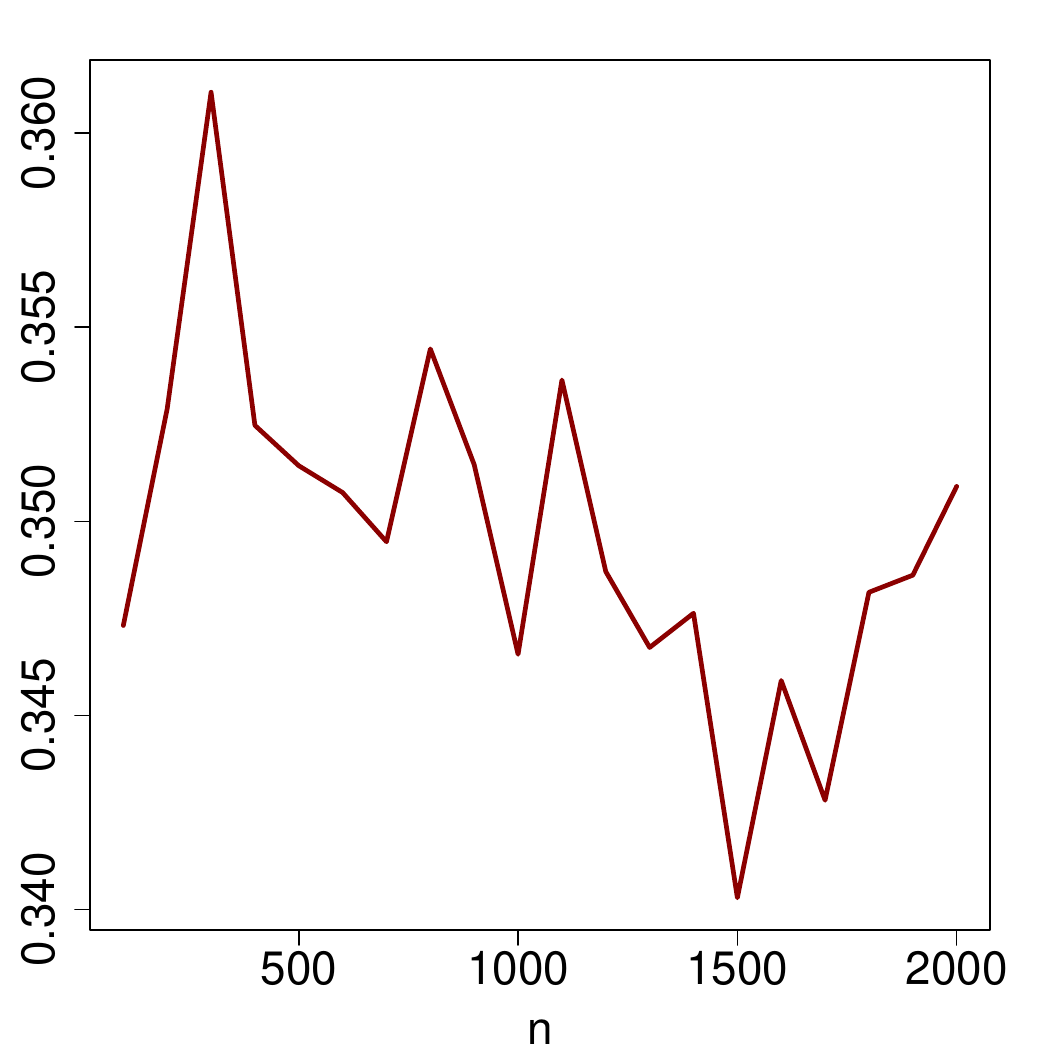}&
\includegraphics[scale=0.35]{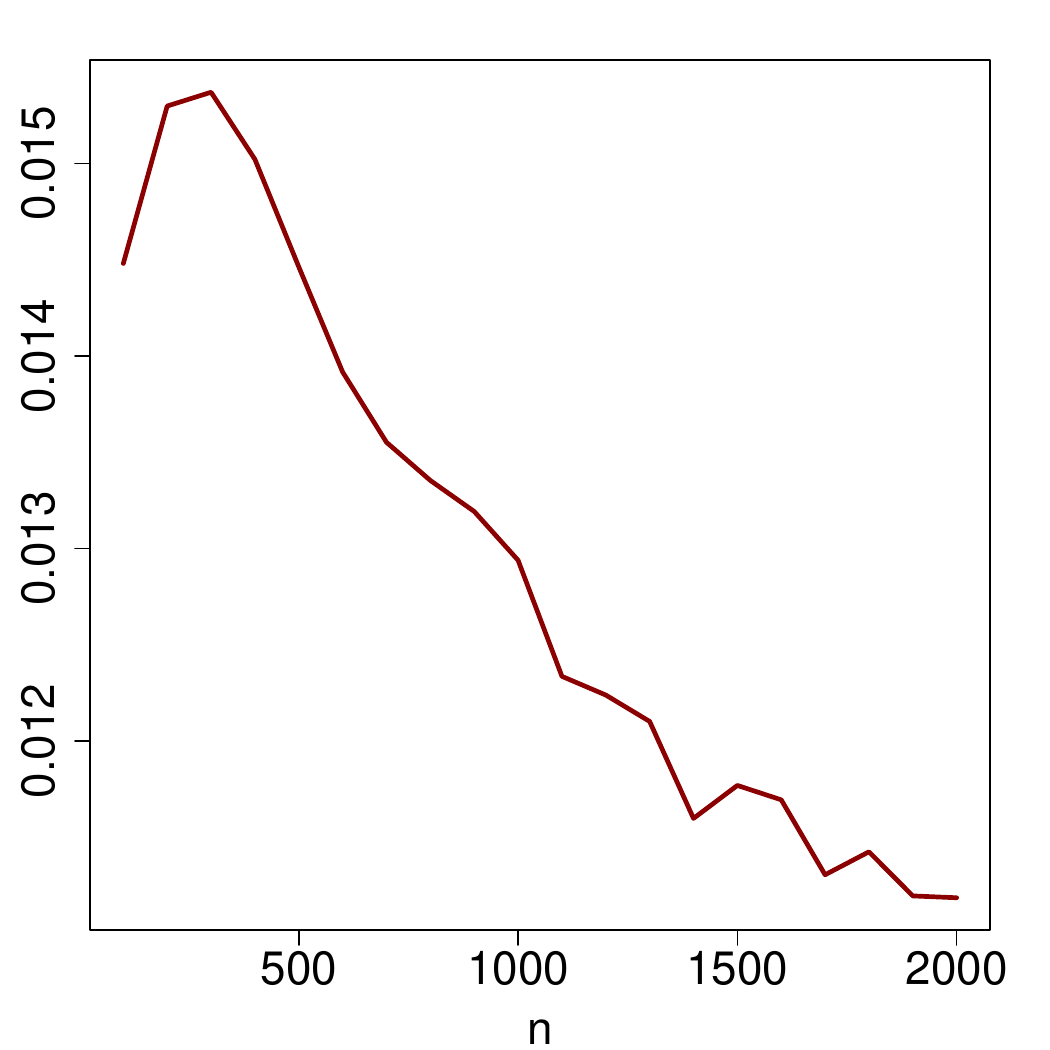}\\[2ex]
\multicolumn{1}{c}{$c=0.10$} & \multicolumn{1}{c}{$c=0.05$}\\[-1ex]
\includegraphics[scale=0.35]{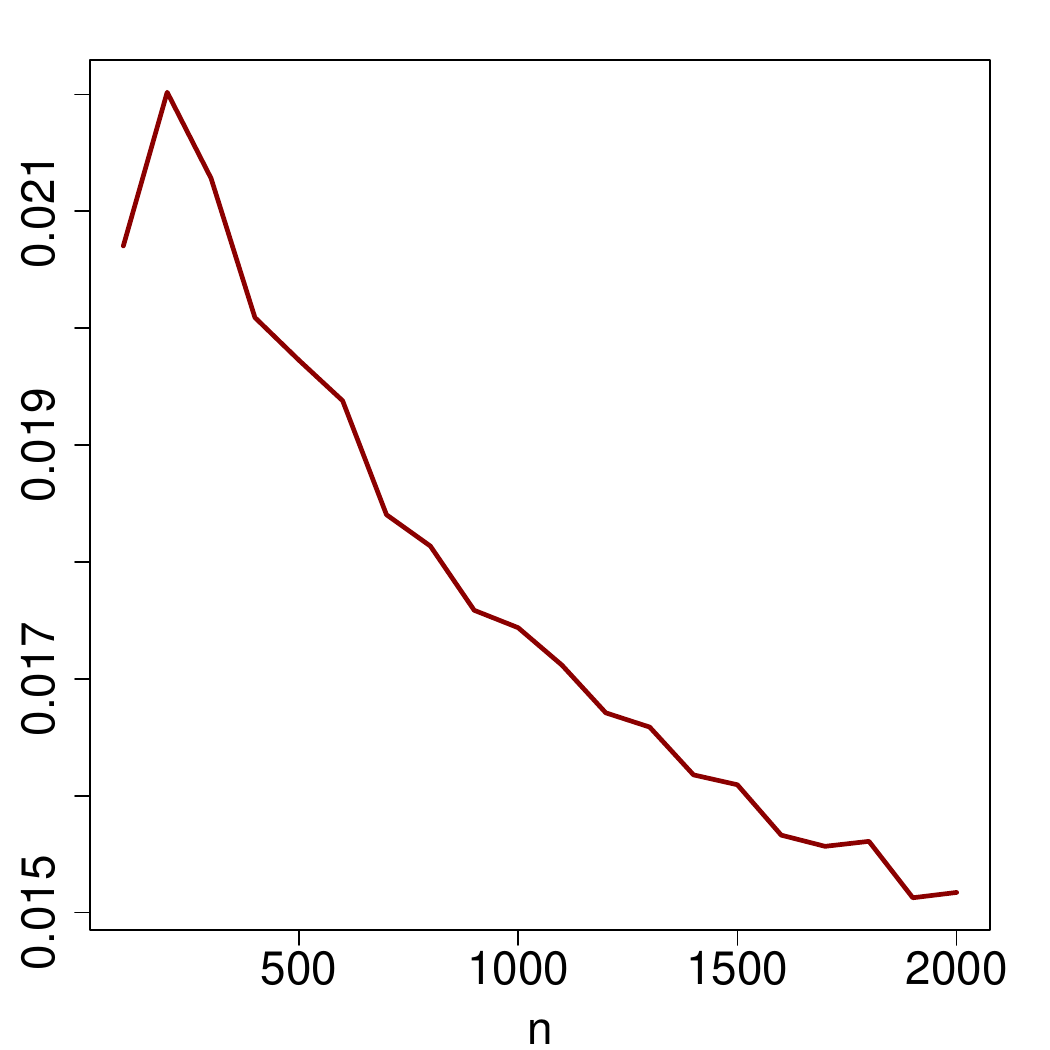}&
\includegraphics[scale=0.35]{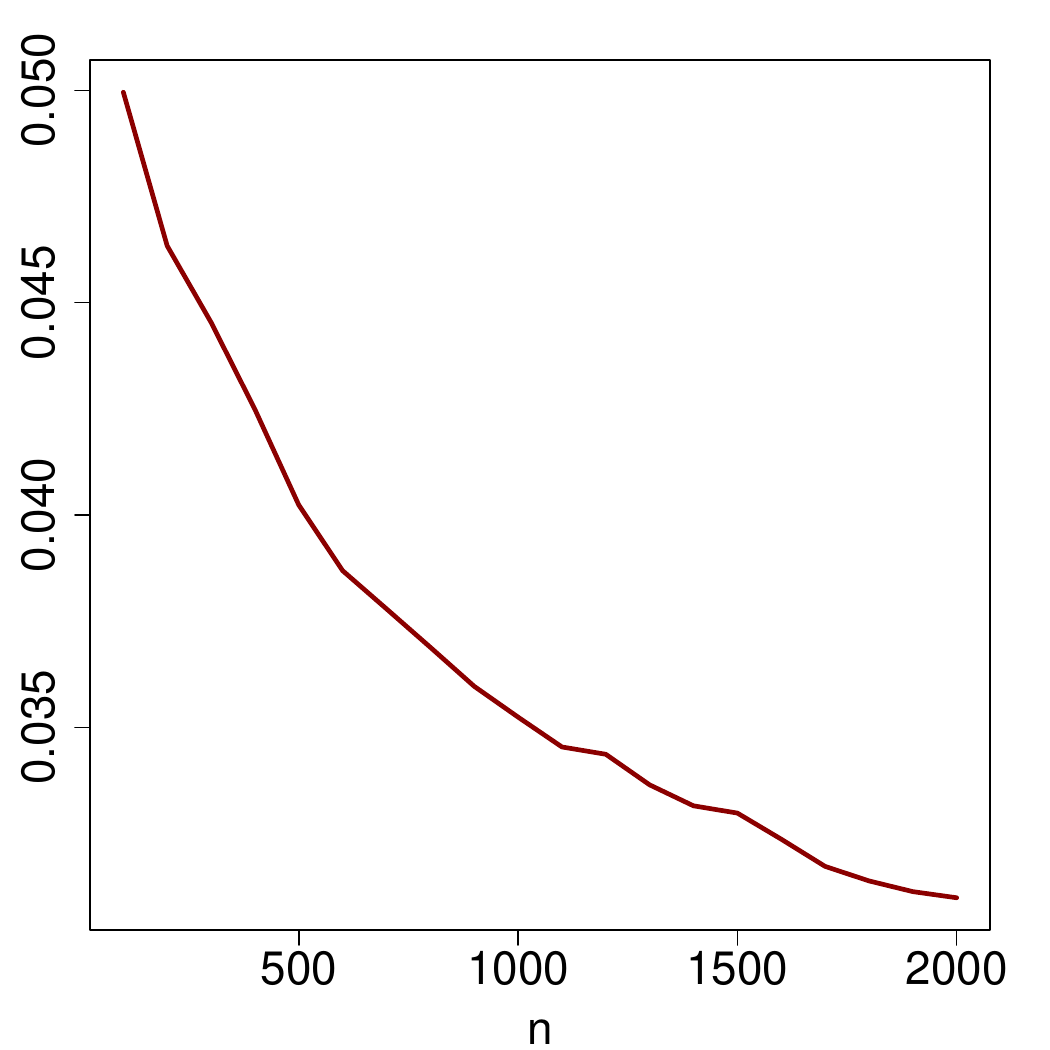}
\end{tabular}
\vskip-0.1in
\caption{\small \label{fig:EMAE-n} Behaviour of  \textsc{EMAE} as a function of $n$ for different values of $c>0$ when the penalty parameter equals $\lambda_n=cn^{-0.4}$. In all cases, $\umbral_0=0.5$, $\delta=-1$ and $\sigma=0.01$.} 
\end{center}
\end{figure}

To evaluate the behaviour of the threshold estimate as the sample size increases, Figure \ref{fig:EMAE-n} presents the plots of the empirical mean absolute  error as $n$ increases for $c=0, 0.005, 0.01$ and $0.05$.  In the four situations represented,  the error standard deviation equals $\sigma=0.01$ and the regression function corresponds to $r_{\umbral_0,\delta}(x)$ with $\umbral_0=0.5$ and $\delta=-1$.   Figure \ref{fig:EMAE-n} confirms the consistency result stated in Theorem \ref{teo:tiro_del_final} when penalizing the objective function. Effectively, for any $c>0$, the \textsc{EMAE} decreases with $n$ showing that $\widehat{\umbral}_n\to \umbral_0$ when $n\to\infty$. In contrast, when $c=0$ the empirical mean absolute  error  has an erratic behaviour as the sample size increases, showing that the penalizing term is necessary for the consistency of the estimator.

\begin{figure}[ht!]
\begin{center}
\renewcommand{\arraystretch}{0.1}
\newcolumntype{G}{>{\centering\arraybackslash}m{\dimexpr.33\linewidth-1\tabcolsep}}
\begin{tabular}{cc}
\multicolumn{2}{c}{$r_{\umbral_0,\delta}(x)$, $\umbral_0=0.5$,  $\delta\,=\,-\,1$ }\\
\multicolumn{1}{c}{a)} & \multicolumn{1}{c}{b)}\\
\includegraphics[scale=0.33]{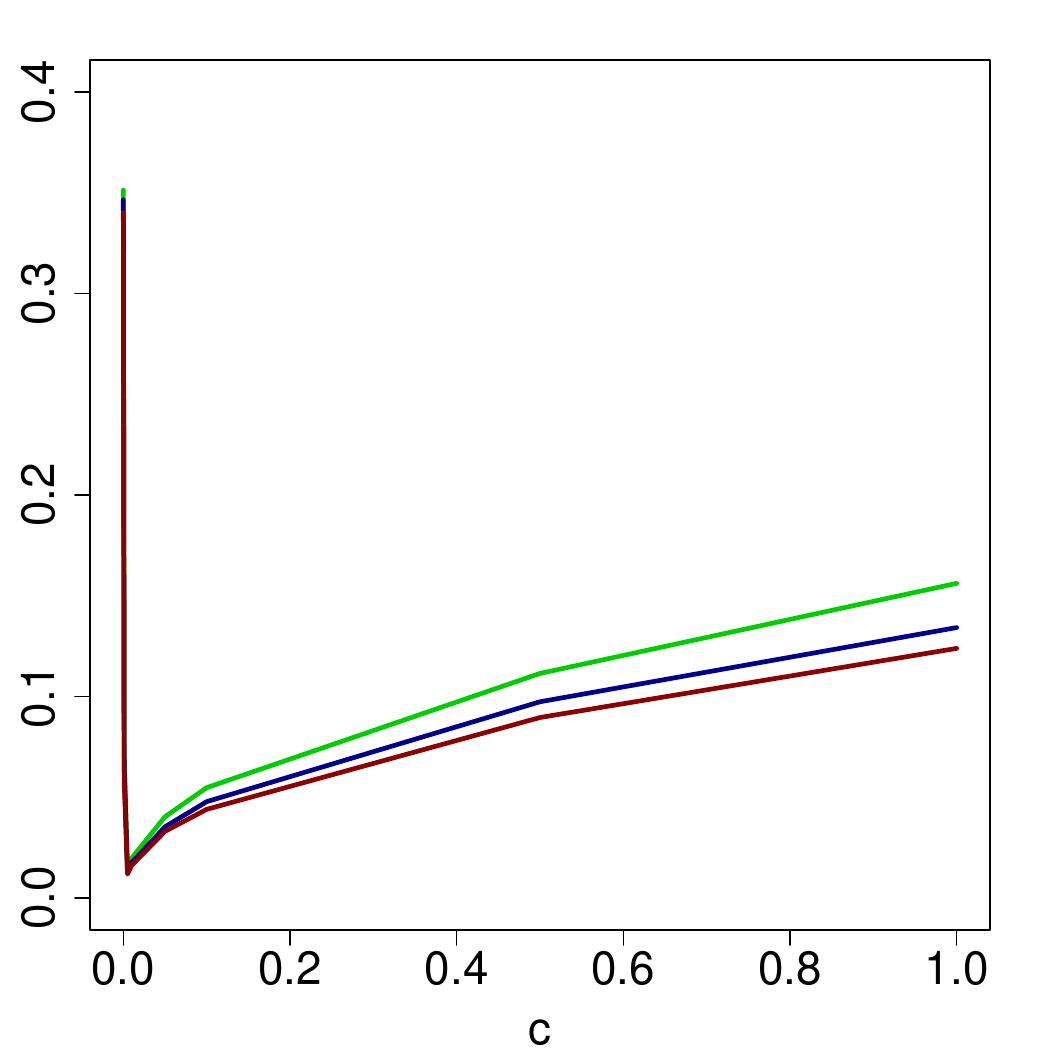} &
\includegraphics[scale=0.33]{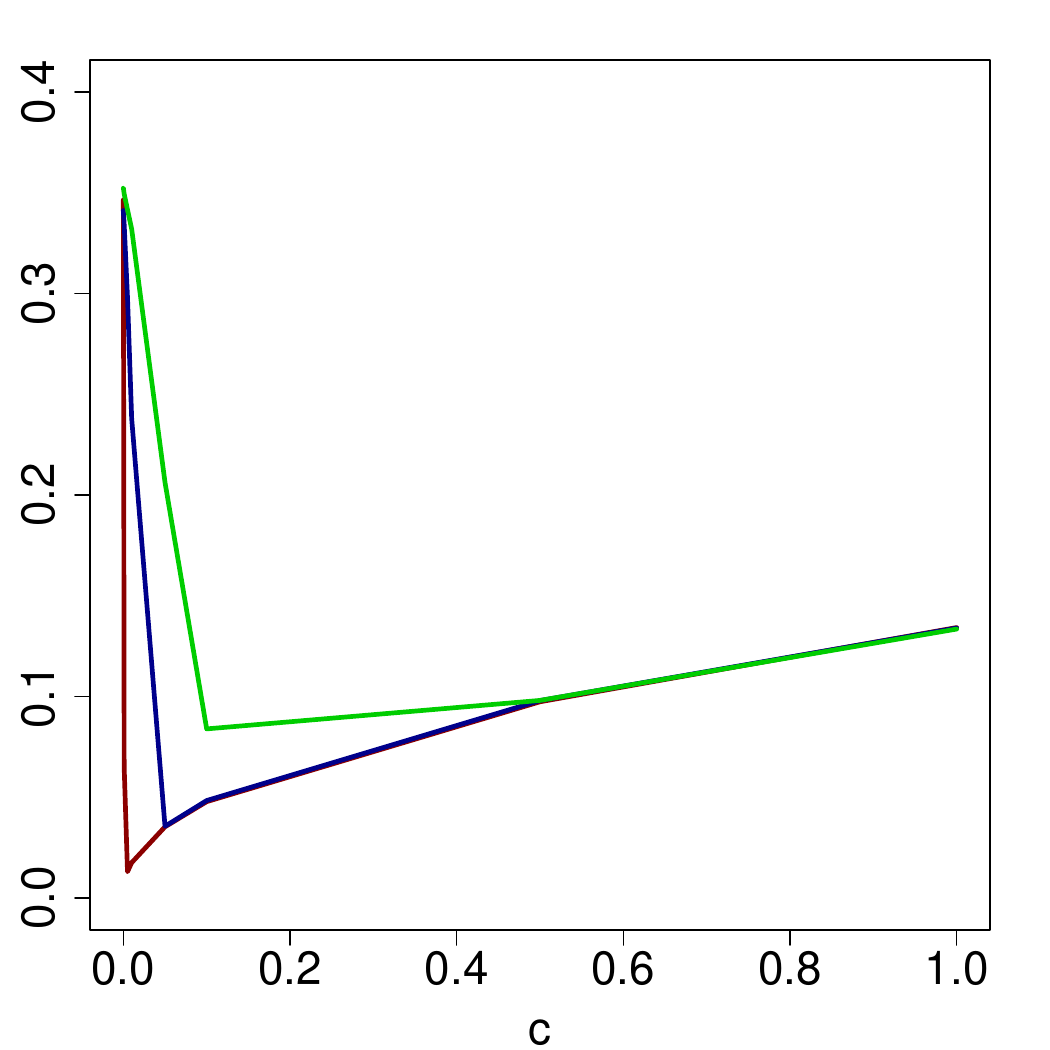}\\

\multicolumn{2}{c}{$r_{\umbral_0,\delta}(x)$, $\umbral_0=0.75$,  $\delta\,=\,-\,1$ }\\

\multicolumn{1}{c}{c)} & \multicolumn{1}{c}{d)}\\
\includegraphics[scale=0.33]{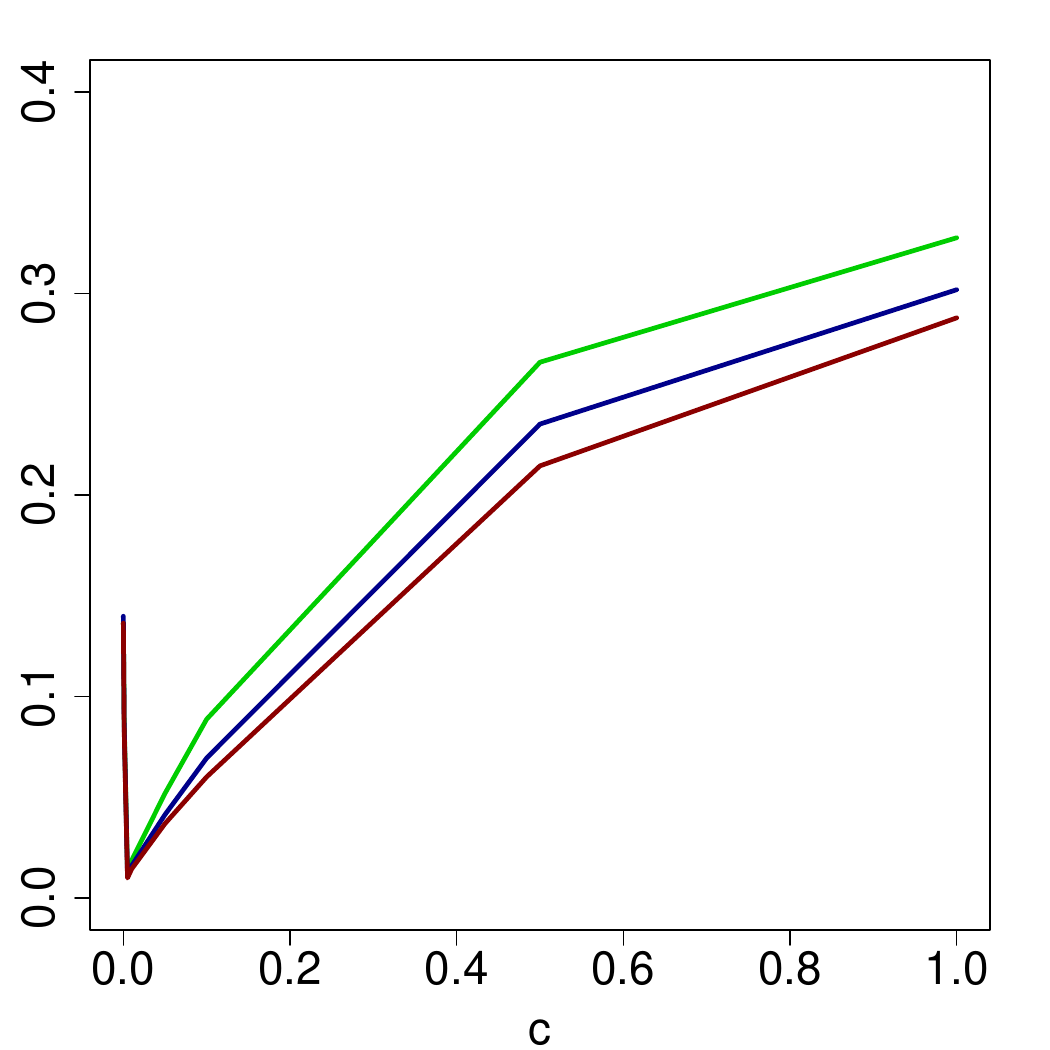} &
\includegraphics[scale=0.33]{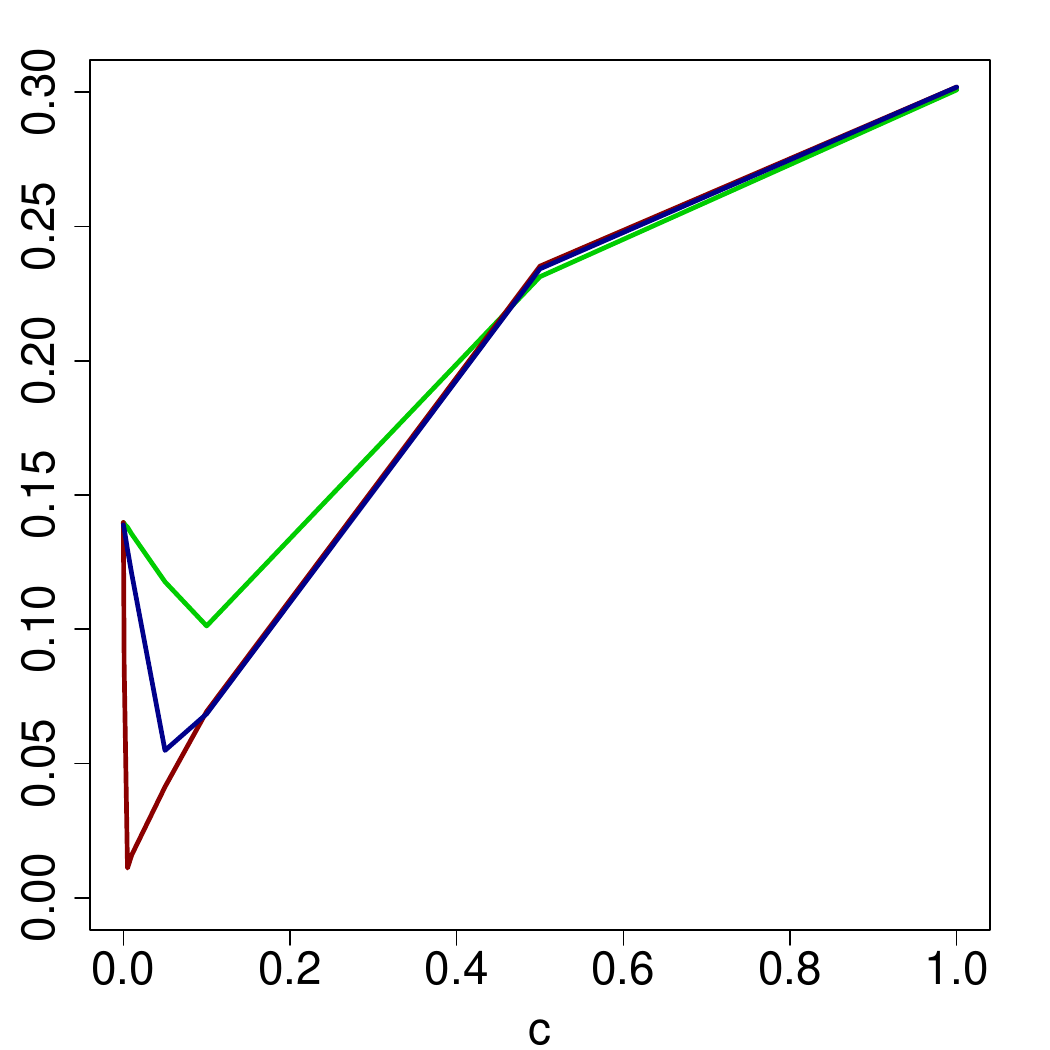} 
\end{tabular}
\vskip-0.1in
\caption{\small \label{fig:EMAE-c} Plot of  \textsc{EMAE} as a function of the penalizing constant $c$ for two regression functions. The a) and c) panels correspond to $\sigma=0.01$  and the green, blue and red lines represent the  \textsc{EMAE} for $n=500, 1000 $ and $1500$. The b) and d) panels correspond  to $n=1000$ and the solid  green,  blue and red lines correspond to $\sigma=0.10, 0.05$ and $0.01$, respectively.} 
\end{center}
\end{figure}

Taking into account that the estimators performance may depend on the penalizing constant $c$, we also explore the behaviour of \textsc{EMAE}  for different sample sizes and values of the errors standard deviation as a function of $c$, when  $c$ varies between $0$ and $1$ taking the values 0, $10^{-k}$ and $0.5\times 10^{-k}$, for $k=0, \ldots, 10$.  Figure \ref{fig:EMAE-c}  displays the   empirical mean absolute  error for $c\in [0,1]$. The left panels correspond to $\sigma=0.01$  with the green, blue and red lines representing the  \textsc{EMAE} for three sample sizes $n=500, 1000 $ and $1500$. The plots on the right panels represent  \textsc{EMAE} when the sample size equals $n=1000$ and three values of the standard error. In this case, the solid  green, blue and red lines correspond to $\sigma=0.10, 0.05$ and $0.01$, respectively. 
Figure \ref{fig:EMAE-c} clearly shows that, as expected, a discontinuity in the empirical mean absolute  error  occurs at $c=0$, since for that value of the penalty parameter the threshold estimator does not converge to the true threshold. Besides, even when the estimators are consistent for any $c>0$, the  \textsc{EMAE} increases with $c$ for the sample sizes and standard errors considered, since in this case, the penalty tends to dominate over the objective function $\widehat{\ell}(\umbral,{\mathcal{Z}}_n)$. The plots reveal that there is an optimal value for the constant $c$, that is, a value  $c_{\mbox{\sc\footnotesize opt}}$ of $c$ minimizing  the empirical  mean absolute  error,  for the different situations considered. Moreover, the   value of $c_{\mbox{\sc\footnotesize opt}}$ seems to coincide for the different values of $n$ when $\sigma$ is fixed (see  Figure  \ref{fig:EMAE-c}.a and \ref{fig:EMAE-c}.c). In contrast, as revealed in  Figure  \ref{fig:EMAE-c}.b and \ref{fig:EMAE-c}.d, when the  sample size equals $n=1000$ and the errors scale vary  the optimal $c$ increases with $\sigma$, suggesting that any data--driven rule to select $c$ should depend on the errors standard deviation estimator. 

%
%

\begin{figure}[ht!]
\begin{center}
\renewcommand{\arraystretch}{0.1}
\newcolumntype{G}{>{\centering\arraybackslash}m{\dimexpr.33\linewidth-1\tabcolsep}}
\begin{tabular}{cc}
\multicolumn{2}{c}{$r_{\umbral_0,\delta}$, $\umbral_0=0.5$}\\
\multicolumn{1}{c}{a)} & \multicolumn{1}{c}{b)}\\
\includegraphics[scale=0.33]{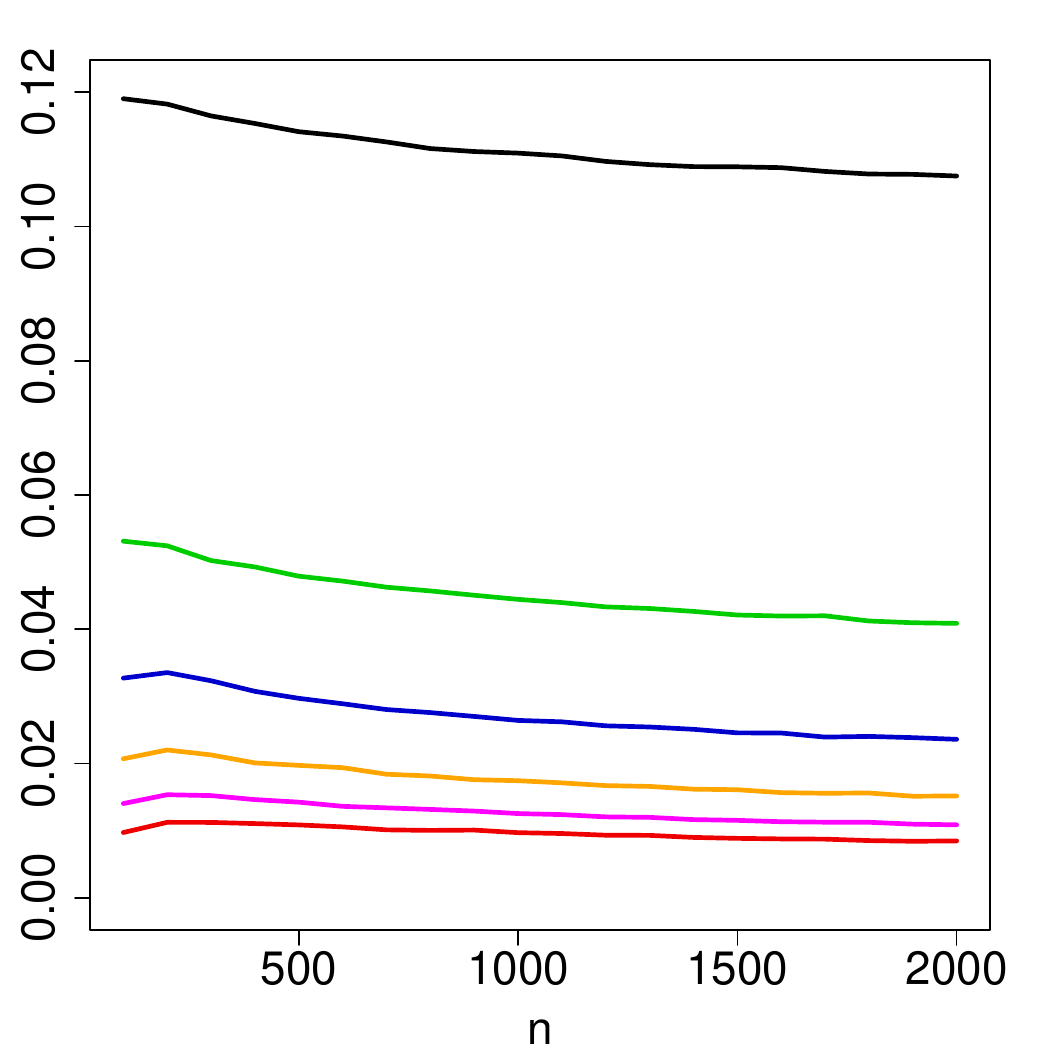} &
\includegraphics[scale=0.33]{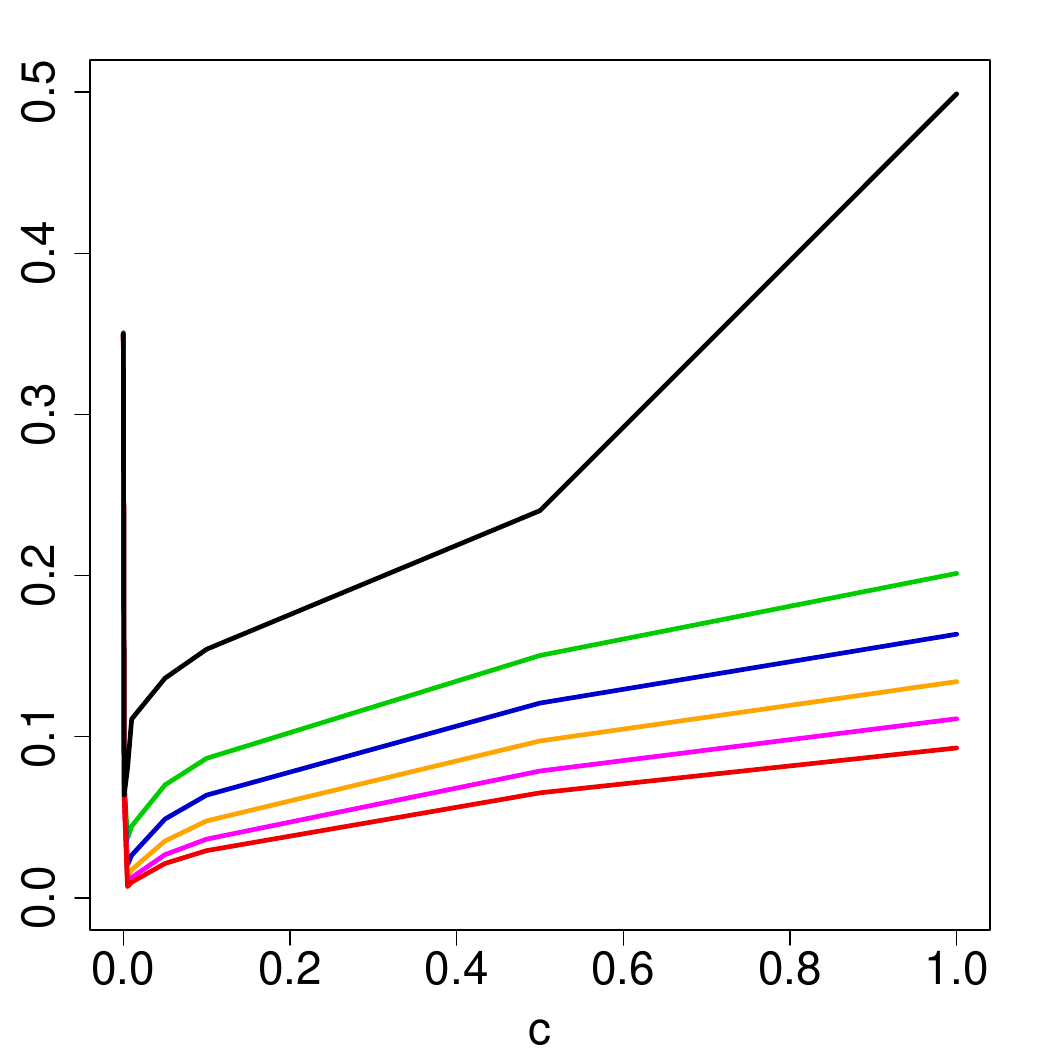}\\

\multicolumn{2}{c}{$r_{\umbral_0,\delta}$, $\umbral_0=0.75$}\\
\multicolumn{1}{c}{c)} & \multicolumn{1}{c}{d)}\\
\includegraphics[scale=0.33]{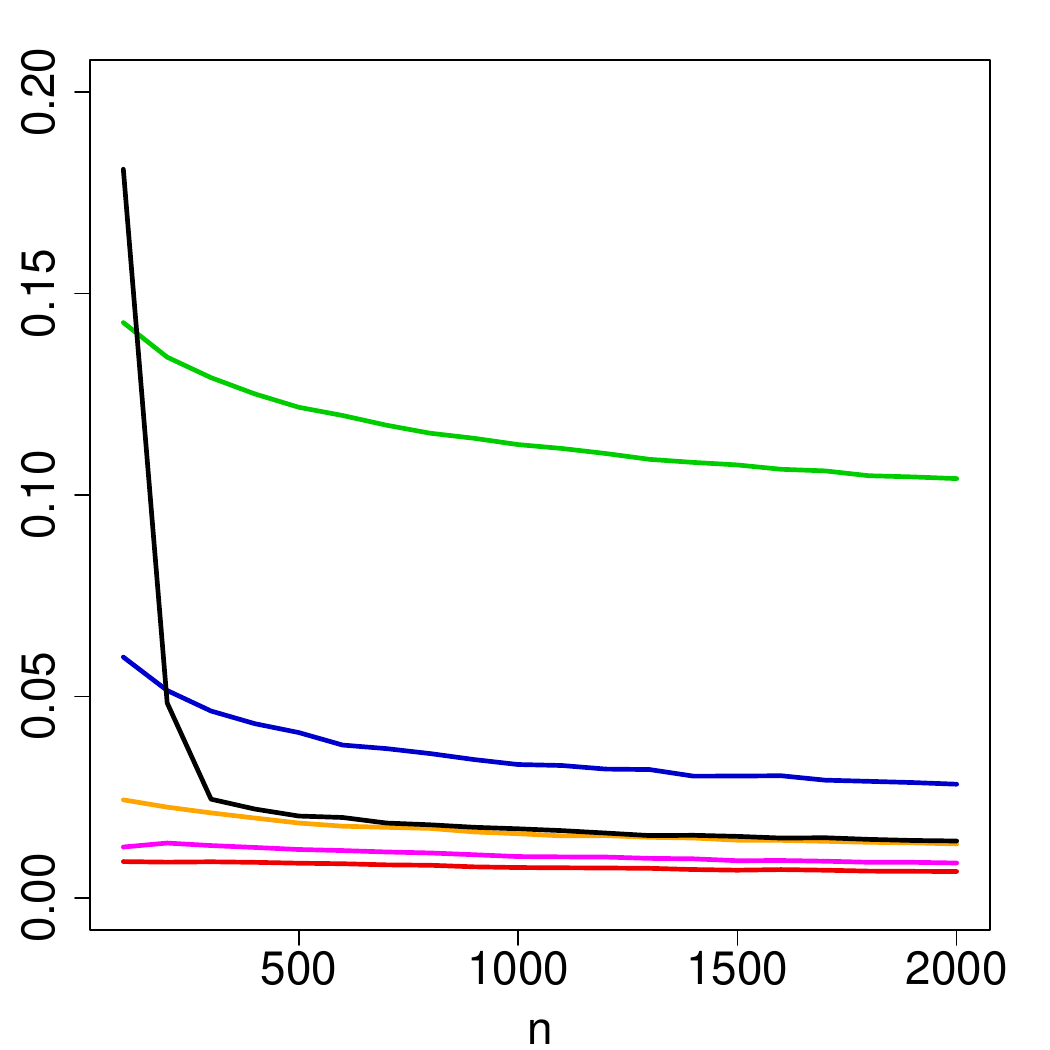} &
\includegraphics[scale=0.33]{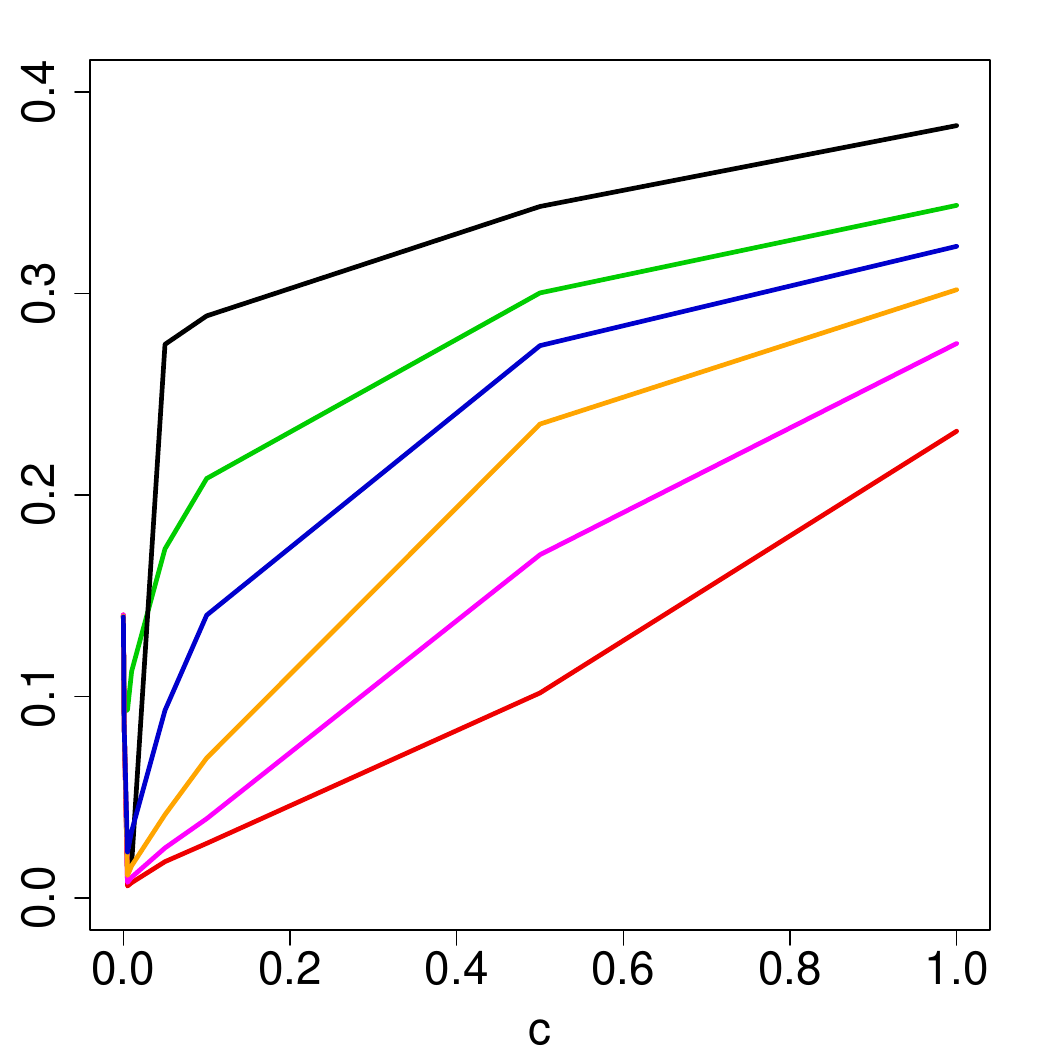}
\end{tabular}
\vskip-0.1in
\caption{\small \label{fig:EMAE-delta}	 Plot of  \textsc{EMAE}, for two choices of the threshold $\umbral_0$. Panels a) and c) display the EMAE as a function  of the sample size $n$, when $c=0.01$, while panels b) and d) present the EMAE  as a function of the penalizing constant $c$   when $n=1000$  (b) and d). In all cases, the errors standard deviation equals   $\sigma=0.01$.   
The black, green, blue, orange, magenta and red lines represent the  \textsc{EMAE} for $\delta=1, \, 0, \,-\, 0.5,\,-\,1,\,-\,1.5$  and $\,-\,2$.} 
\end{center}
\end{figure}
To analyse the behaviour of the \textsc{EMAE} according to the smoothness of the regression function $r_{\umbral_0,\delta}$ given by the parameter $\delta$, panels a) and c) in Figure \ref{fig:EMAE-delta}  present  a plot of  \textsc{EMAE} as a function  of $n$ for four values of $\delta$ when $c=0.01$, while  panels b) and d)    display the  \textsc{EMAE} as a function  of the penalizing constant $c$ also for four choices of $\delta$ and $n=1000$.  The standard deviation $\sigma$ was set equal to $0.01$ in all cases. When  $\umbral_0=0.5$, the green, violet, orange and magenta lines represent the  \textsc{EMAE} for $\delta=0, \,-\, 0.5,\,-\,1$ and $\,-\,1.5$, while for   $\umbral_0=0.75$, the blue, green,   orange and grey lines correspond to  $\delta=1, 0, ,\,-\,1$ and $\,-\,2$. Note that the case $\delta= \,-\, 1$ was already considered in the left panels of Figure \ref{fig:EMAE-c}, labelled a) and c), for $\sigma=0.01$ for three sample sizes. 
The obtained plots illustrate that the   empirical mean absolute  error   is smaller as the absolute value  of $\delta$ increases, confirming that the linear relationship is easier to  detect  as   the jump of the derivative of the regression function  at the threshold becomes larger.

\begin{figure}[ht!]
\begin{center}
\renewcommand{\arraystretch}{0.1}
\newcolumntype{G}{>{\centering\arraybackslash}m{\dimexpr.33\linewidth-1\tabcolsep}}
\begin{tabular}{GGG}
\multicolumn{3}{c}{$r_{\umbral_0,\delta}$, $\umbral_0=0.5$}\\[2ex]
\multicolumn{1}{c}{$\delta= \,-1$} & \multicolumn{1}{c}{$\delta= \,0$} & \multicolumn{1}{c}{$\delta= \,1$}\\[-2ex]
\includegraphics[scale=0.33]{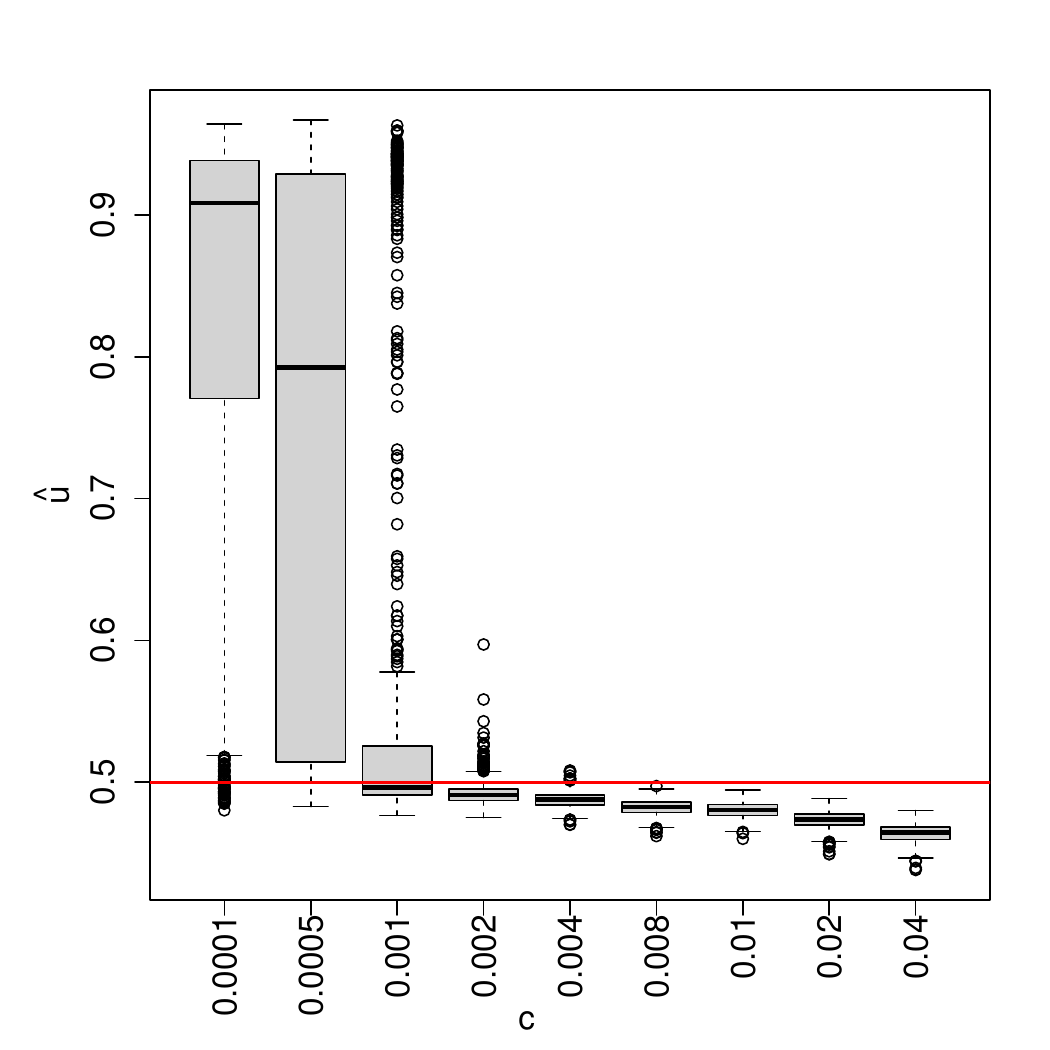} &
\includegraphics[scale=0.33]{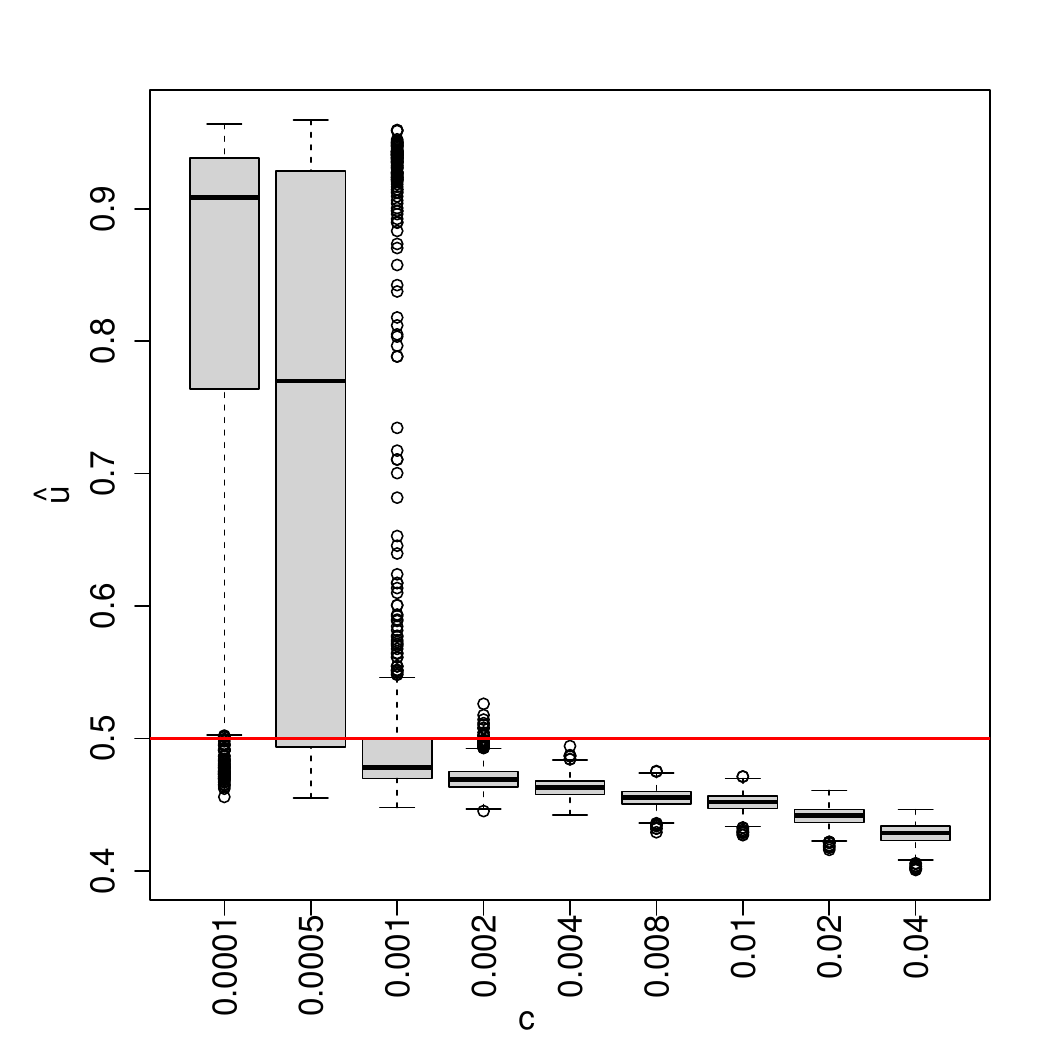} 
&
\includegraphics[scale=0.33]{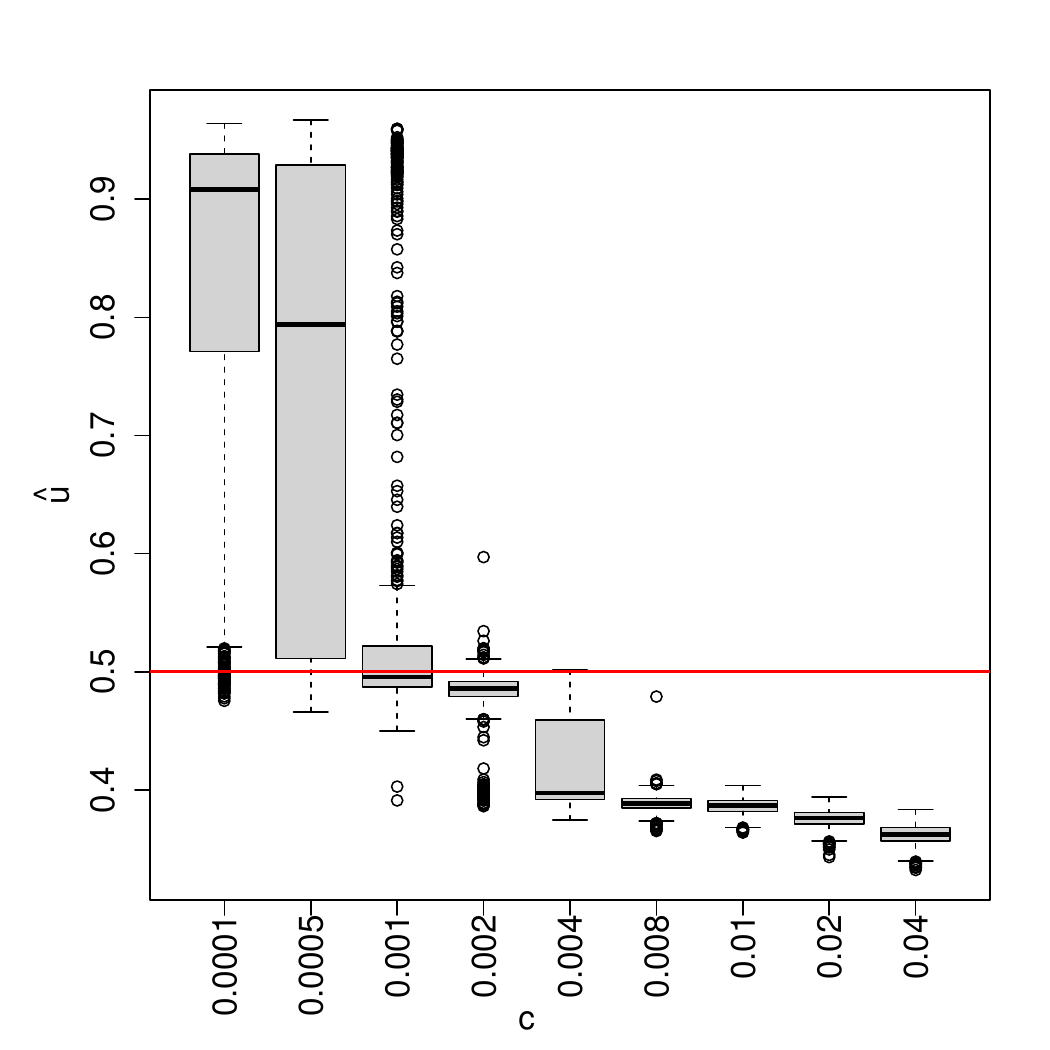}
\\

\multicolumn{3}{c}{$r_{\umbral_0,\delta}$, $\umbral_0=0.75$}\\[2ex]
\multicolumn{1}{c}{$\delta= \,-1$} & \multicolumn{1}{c}{$\delta= \,0$} & \multicolumn{1}{c}{$\delta= \,1$}\\[-2ex]
\includegraphics[scale=0.33]{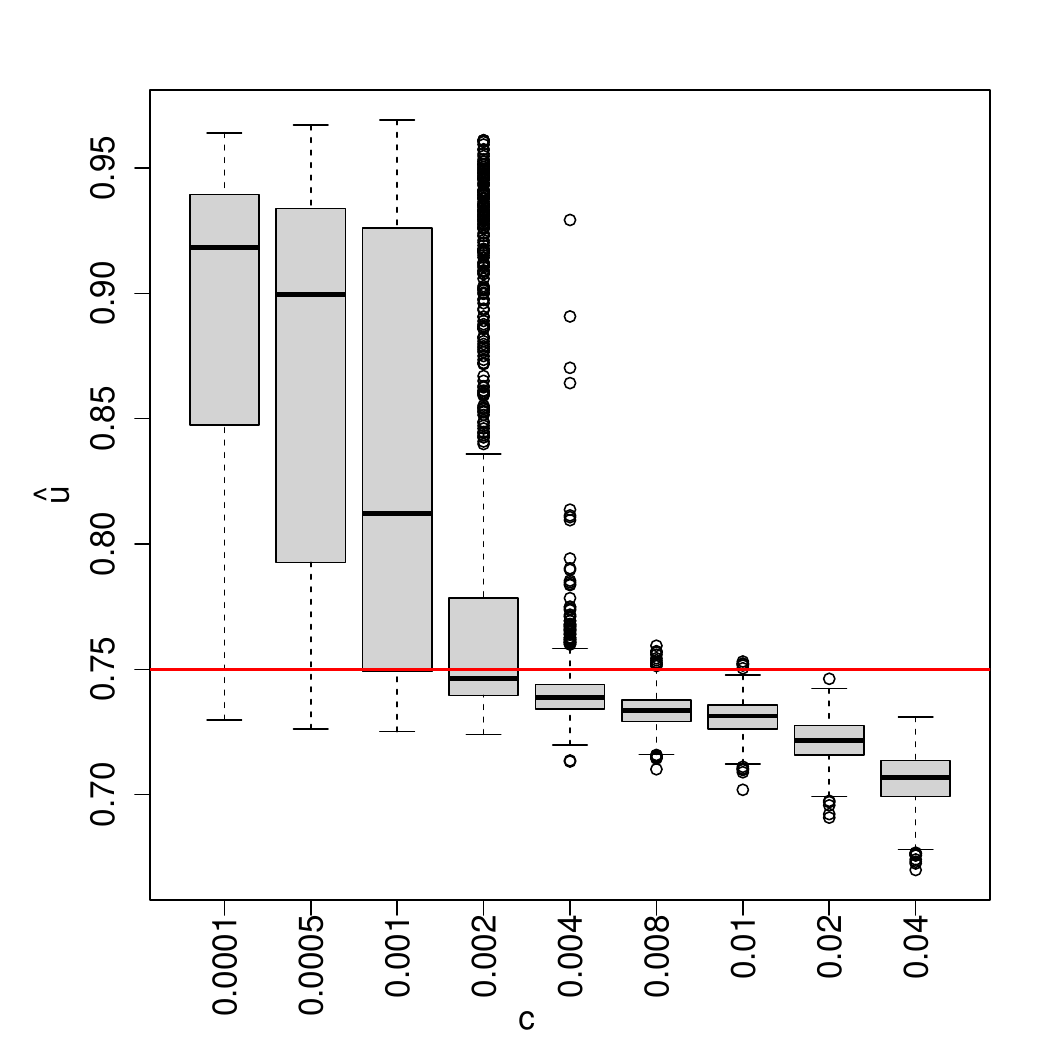} &
\includegraphics[scale=0.33]{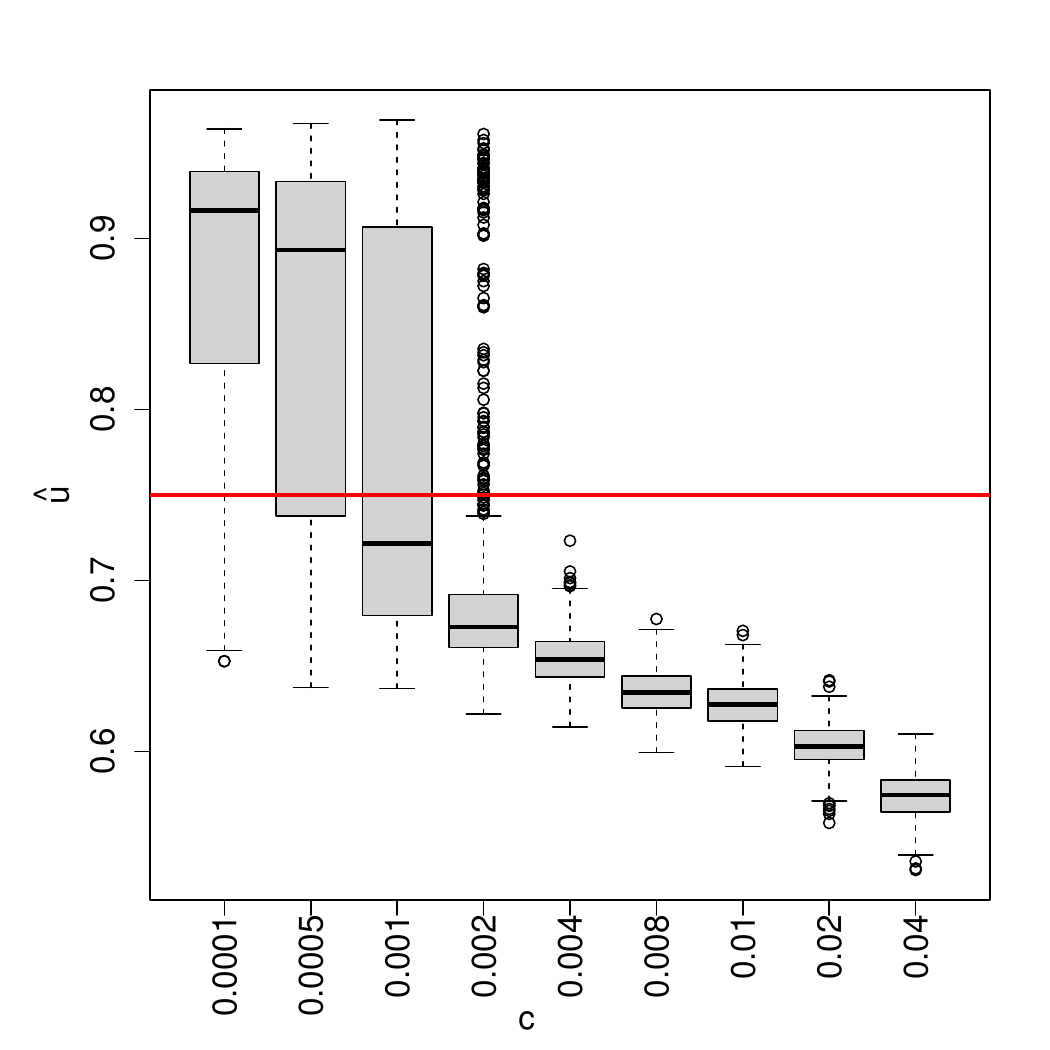}
&
\includegraphics[scale=0.33]{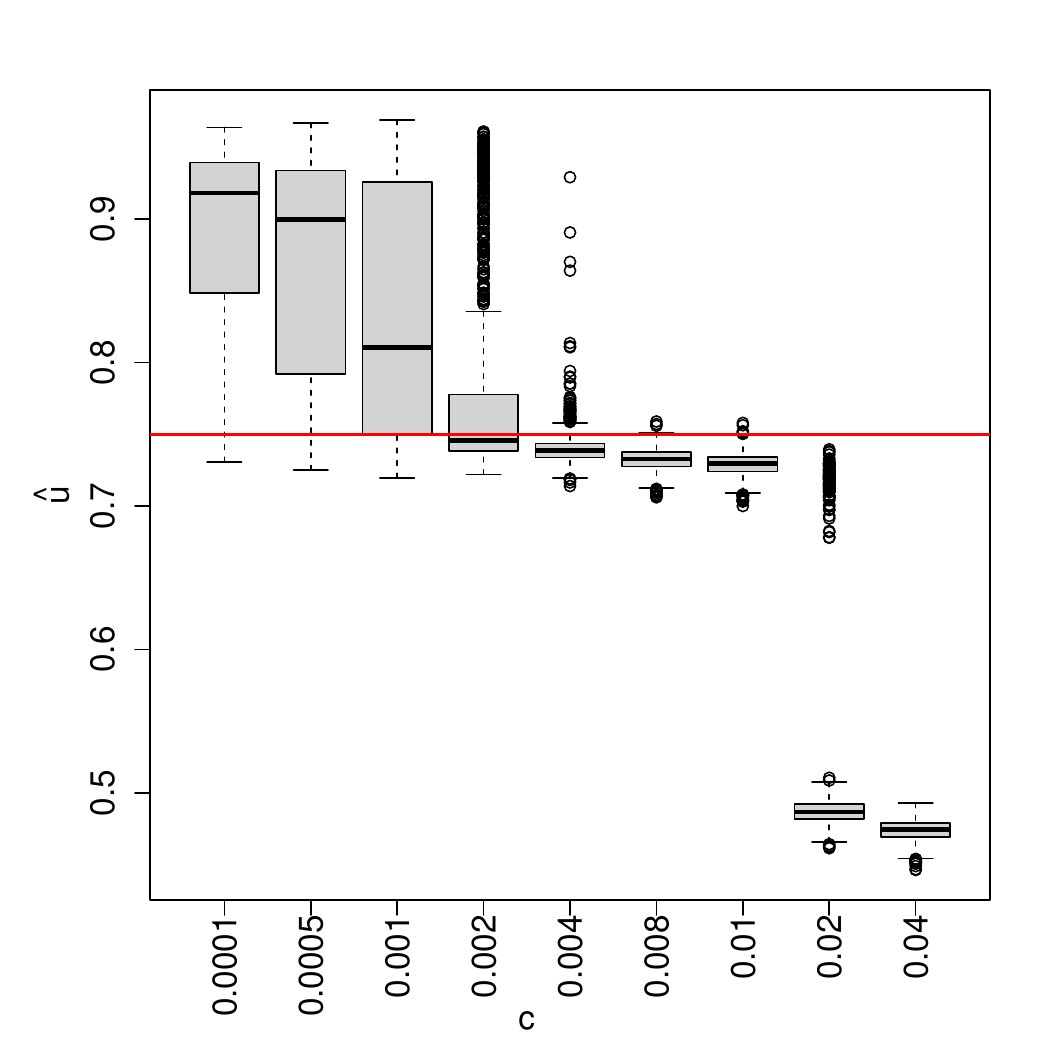}
\end{tabular}
\vskip-0.1in
\caption{\small \label{fig:bxp-uhat}	 Boxplots of the estimators $\widehat{\umbral}$ for different choices of the  penalizing constant $c$,   when $n=500$  and $\sigma=0.01$. The horizontal red line corresponds to the true value $\umbral_0$.} 
\end{center}
\end{figure}

 As an illustration of the finite sample estimators behaviour, Figure \ref{fig:bxp-uhat} displays parallel boxplots of the estimators of the threshold ${\umbral}_0$ for different choices of the penalizing constant $c$, when $u_0=0.5$ and $0.75$, $\delta=-1,0$ and $1$, $\sigma=0.01$ and the sample size equals 500. Similar plots are obtained for other values of $n$. The true parameter is indicated with the horizontal solid  red line. For both choices of $u_0$ and the three values of $\delta$, the constant $c$  takes values in   $\{ 0.0001, 0.0005,  0.001,0.002,0.004,$ $0.008, 0.01, 0.02 , 0.04\}$.
 Figure \ref{fig:bxp-uhat}  shows that, as expected, for small values of $c$ values, the threshold  estimates are larger than the target. In contrast, large values of $c$ lead to  boxplots of   $\widehat{\umbral}$ that lie below the true threshold.    In all the considered situations, for one of the selected values of $c$, the boxplot of the estimates is centered at  ${\umbral}_0$.   These plots also explain the \textsc{EMAE} results displayed in the left panel of Figure \ref{fig:EMAE-c}. Figures \ref{fig:bxp-uhat-5} and \ref{fig:bxp-uhat-10} display the boxplots when $\sigma=0.05$ and $0.10$, for proper grid values of $c$. Similar conclusions arise for these values of the standard deviation, but it should be noticed that, as expected, as expected, larger values of the constant $c$ are desirable as $\sigma$ increases. This behaviour was already observed in Figure  \ref{fig:EMAE-c}.

\begin{figure}[ht!]
\begin{center}
\renewcommand{\arraystretch}{0.1}
\newcolumntype{G}{>{\centering\arraybackslash}m{\dimexpr.33\linewidth-1\tabcolsep}}
\begin{tabular}{GGG}
\multicolumn{3}{c}{$r_{\umbral_0,\delta}$, $\umbral_0=0.5$}\\[2ex]
\multicolumn{1}{c}{$\delta= \,-1$} & \multicolumn{1}{c}{$\delta= \,0$} & \multicolumn{1}{c}{$\delta= \,1$}\\[-2ex]
\includegraphics[scale=0.33]{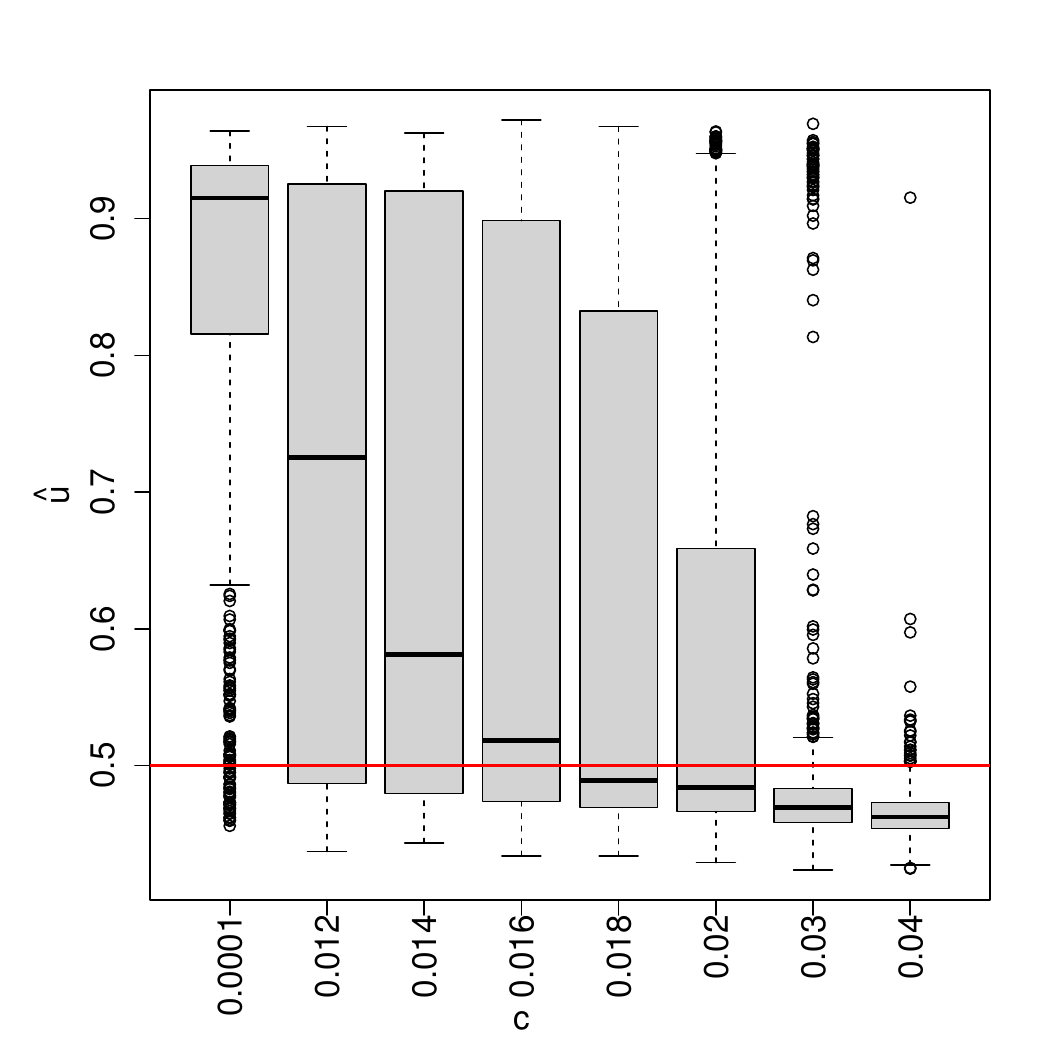} &
\includegraphics[scale=0.33]{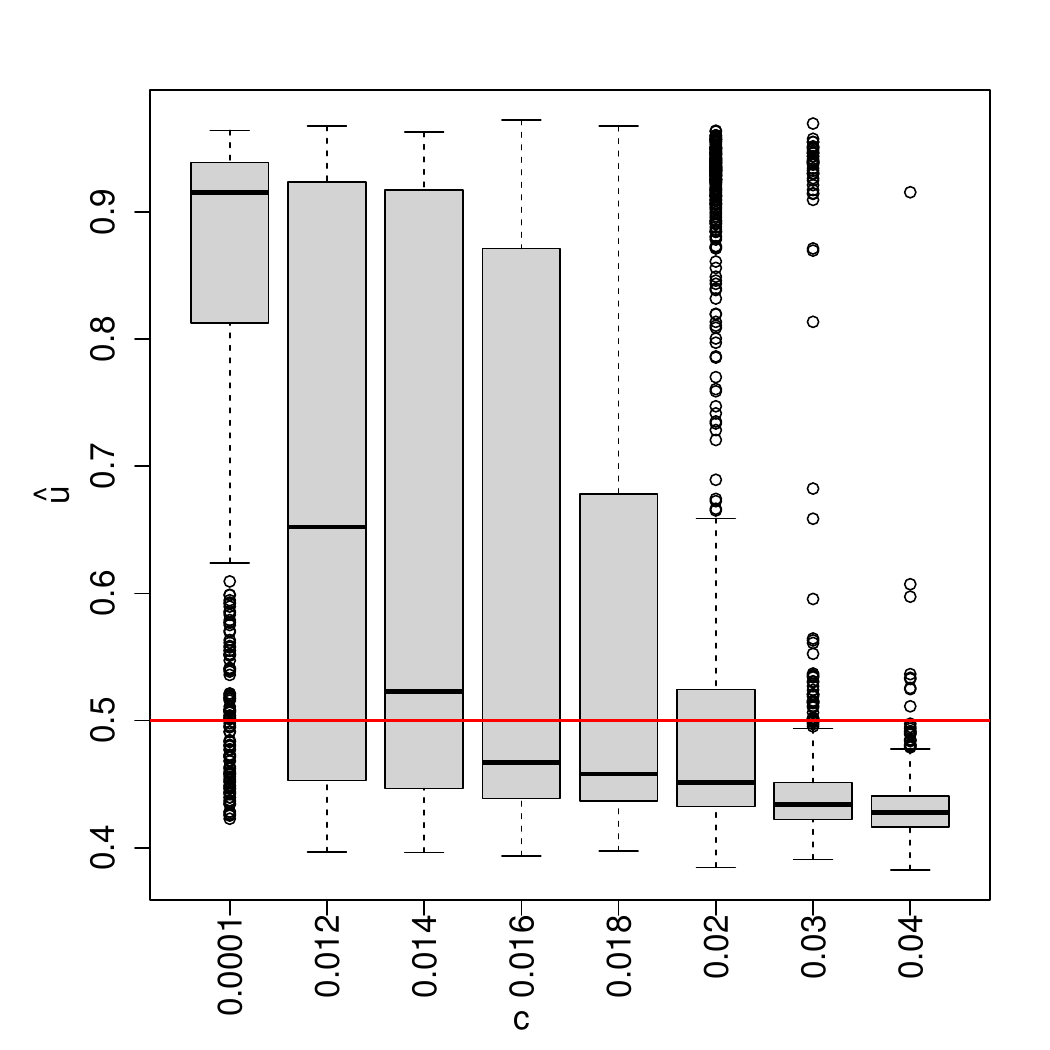} 
&
\includegraphics[scale=0.33]{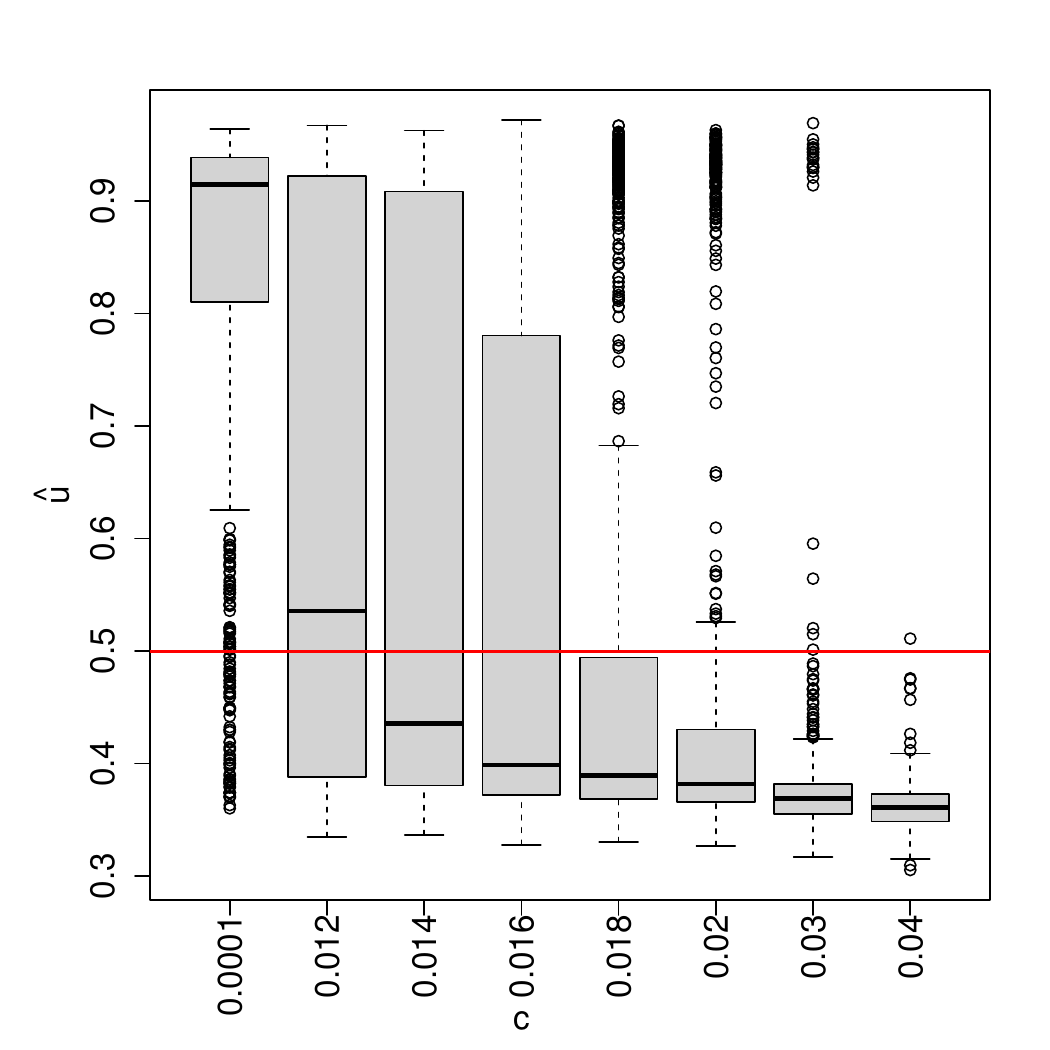}
\\

\multicolumn{3}{c}{$r_{\umbral_0,\delta}$, $\umbral_0=0.75$}\\[2ex]
\multicolumn{1}{c}{$\delta= \,-1$} & \multicolumn{1}{c}{$\delta= \,0$} & \multicolumn{1}{c}{$\delta= \,1$}\\[-2ex]
\includegraphics[scale=0.33]{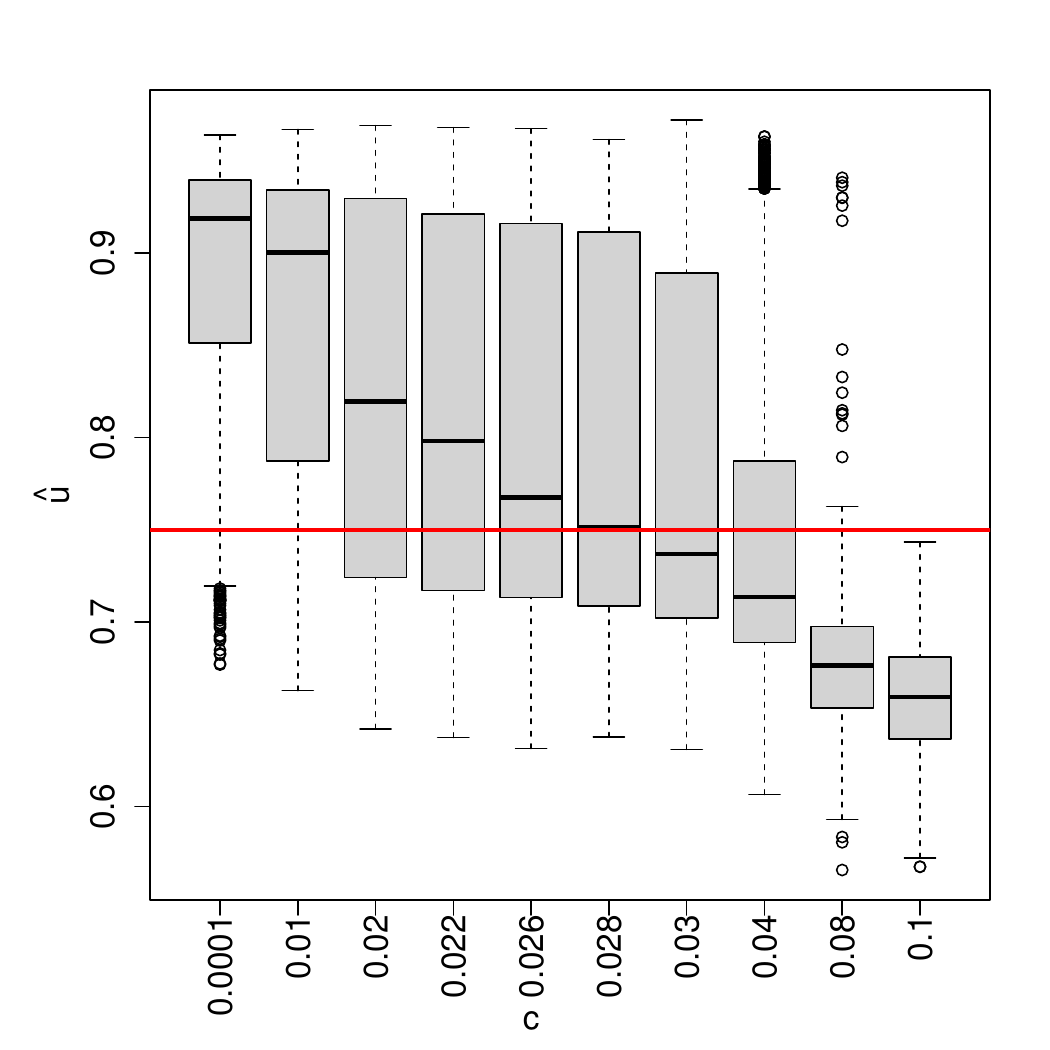} &
\includegraphics[scale=0.33]{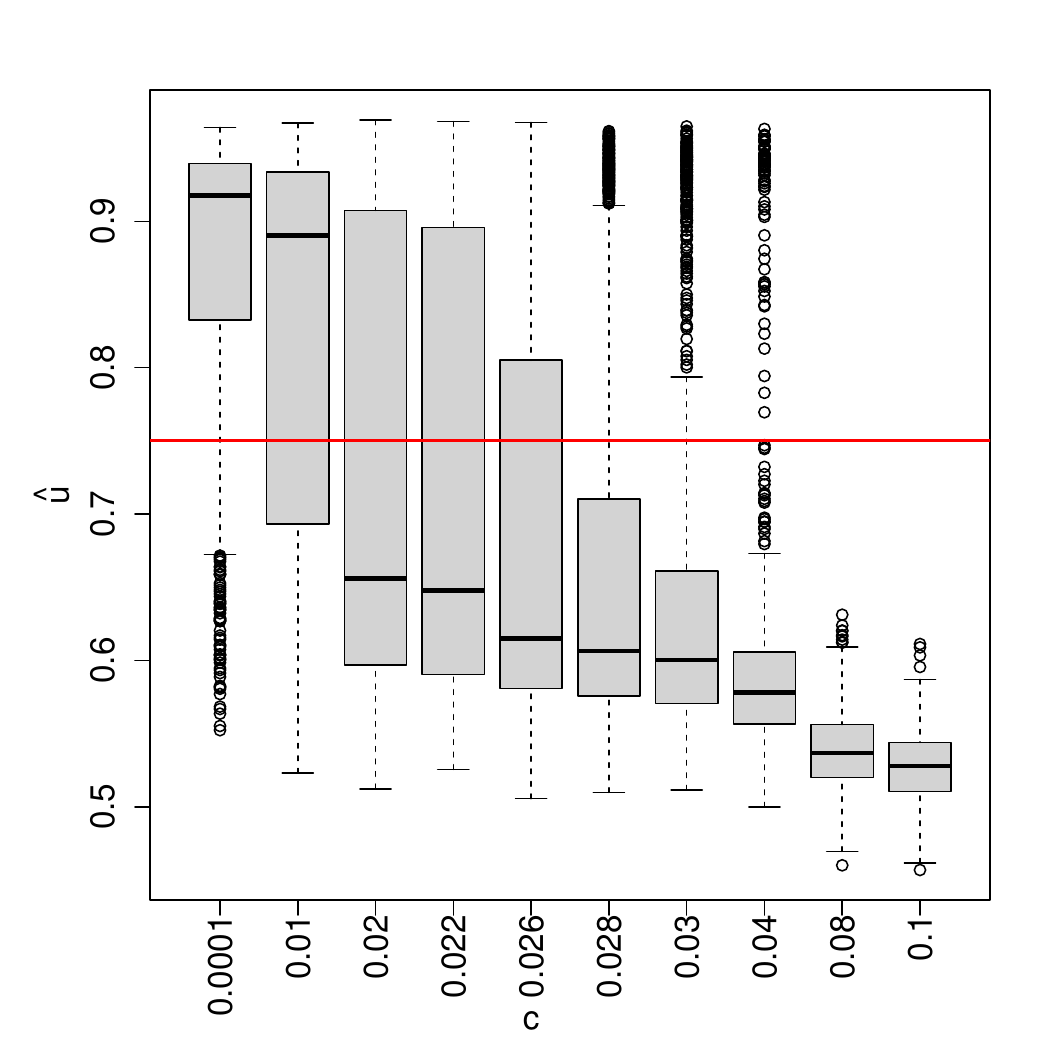}
&
\includegraphics[scale=0.33]{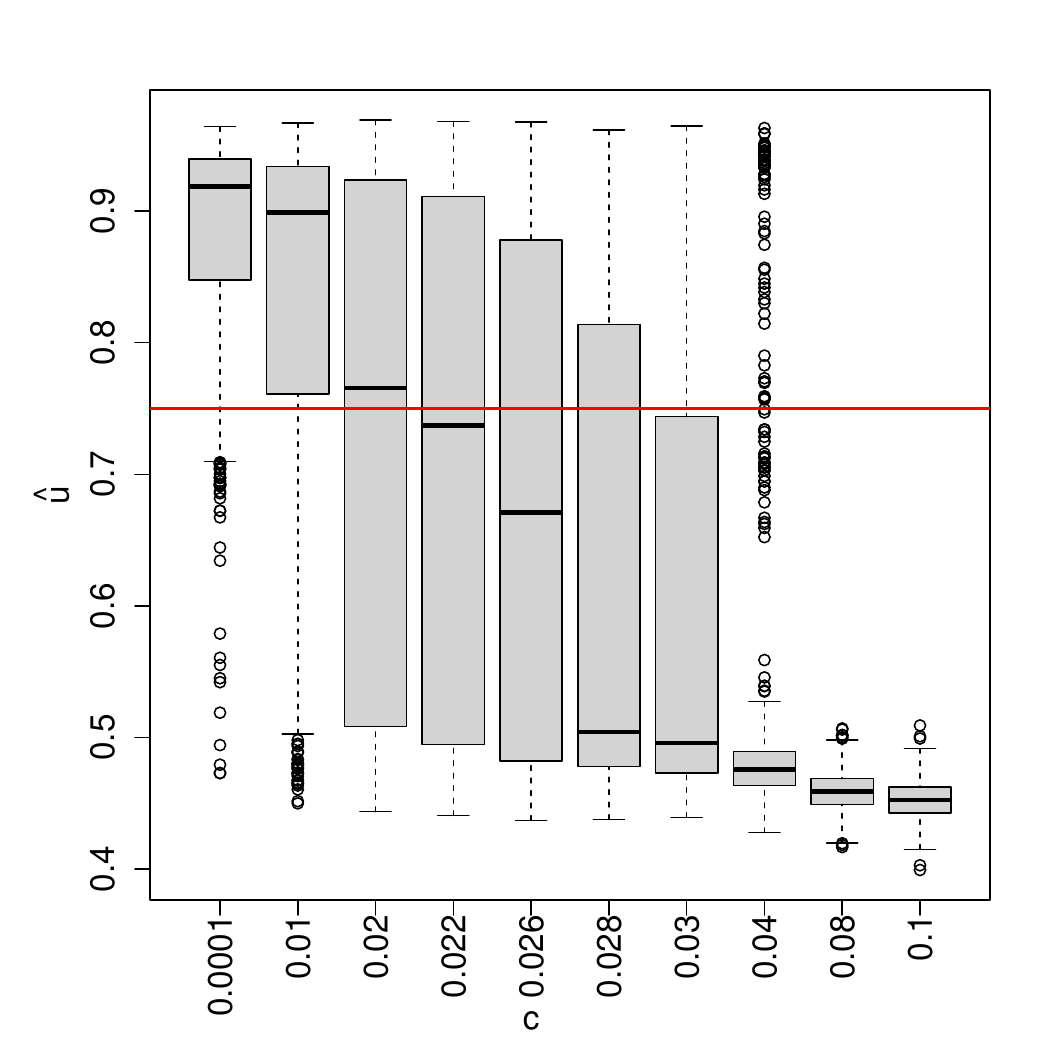}
\end{tabular}
\vskip-0.1in
\caption{\small \label{fig:bxp-uhat-5}	 Boxplots of the estimators $\widehat{\umbral}$ for different choices of the  penalizing constant $c$,   when $n=500$  and $\sigma=0.05$. The horizontal red line corresponds to the true value $\umbral_0$.} 
\end{center}
\end{figure}

\begin{figure}[ht!]
\begin{center}
\renewcommand{\arraystretch}{0.1}
\newcolumntype{G}{>{\centering\arraybackslash}m{\dimexpr.33\linewidth-1\tabcolsep}}
\begin{tabular}{GGG}
\multicolumn{3}{c}{$r_{\umbral_0,\delta}$, $\umbral_0=0.5$}\\[2ex]
\multicolumn{1}{c}{$\delta= \,-1$} & \multicolumn{1}{c}{$\delta= \,0$} & \multicolumn{1}{c}{$\delta= \,1$}\\[-2ex]
\includegraphics[scale=0.33]{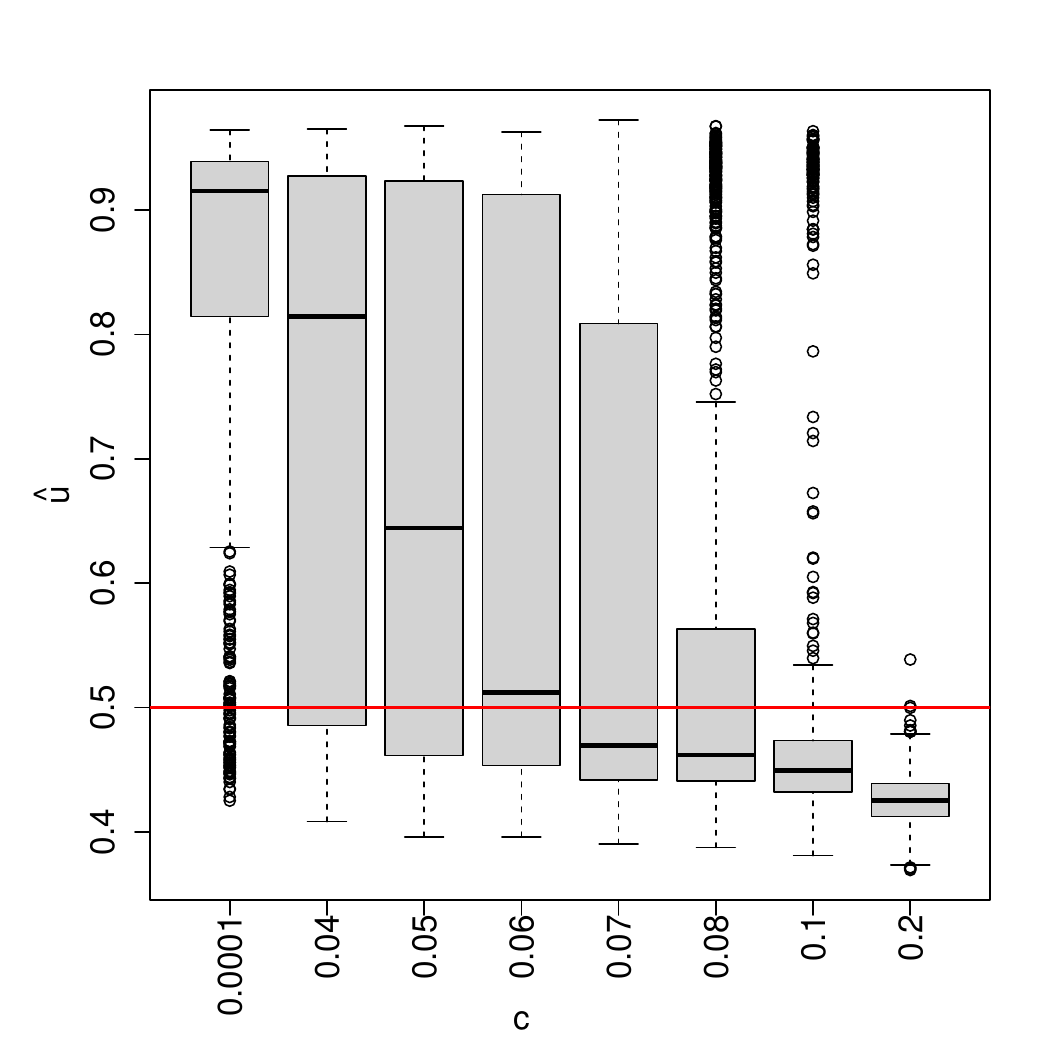} &
\includegraphics[scale=0.33]{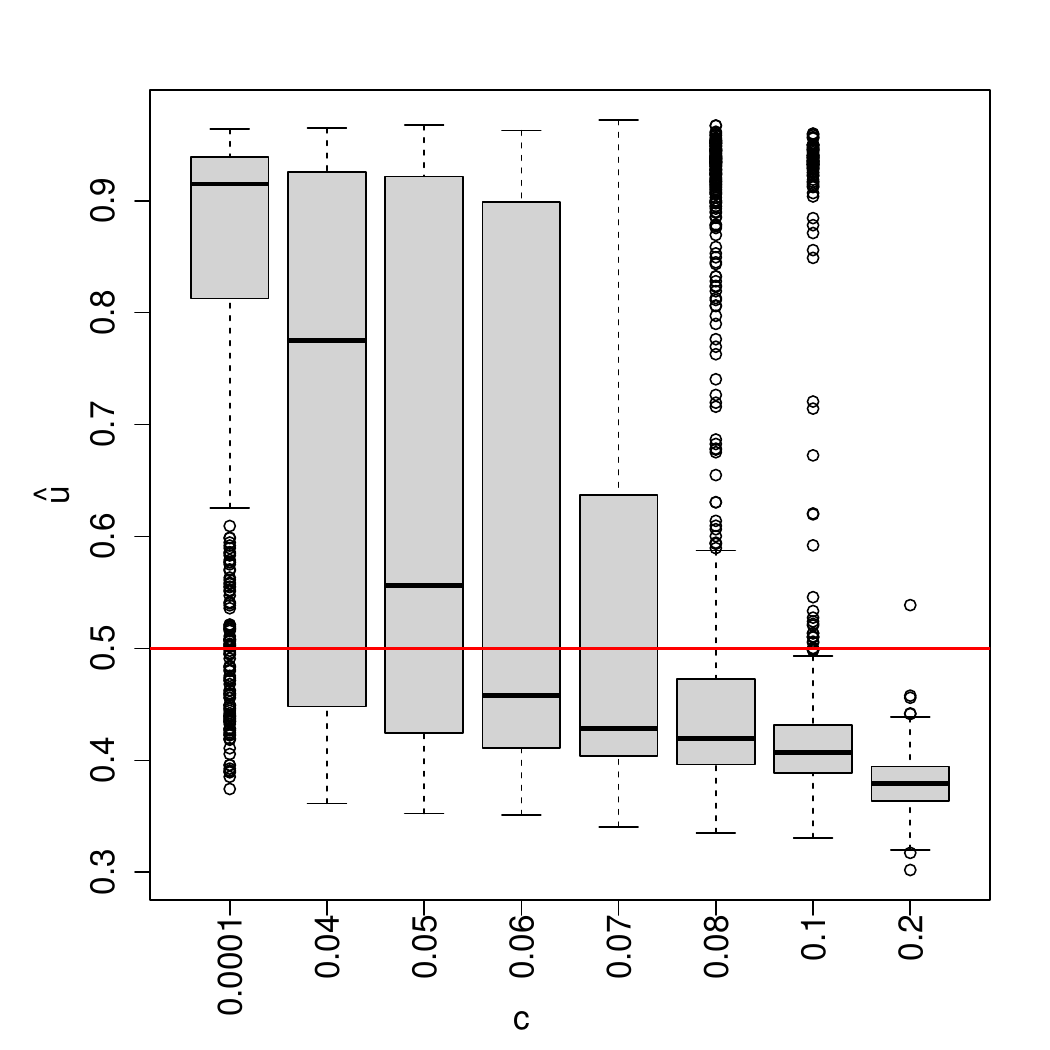} 
&
\includegraphics[scale=0.33]{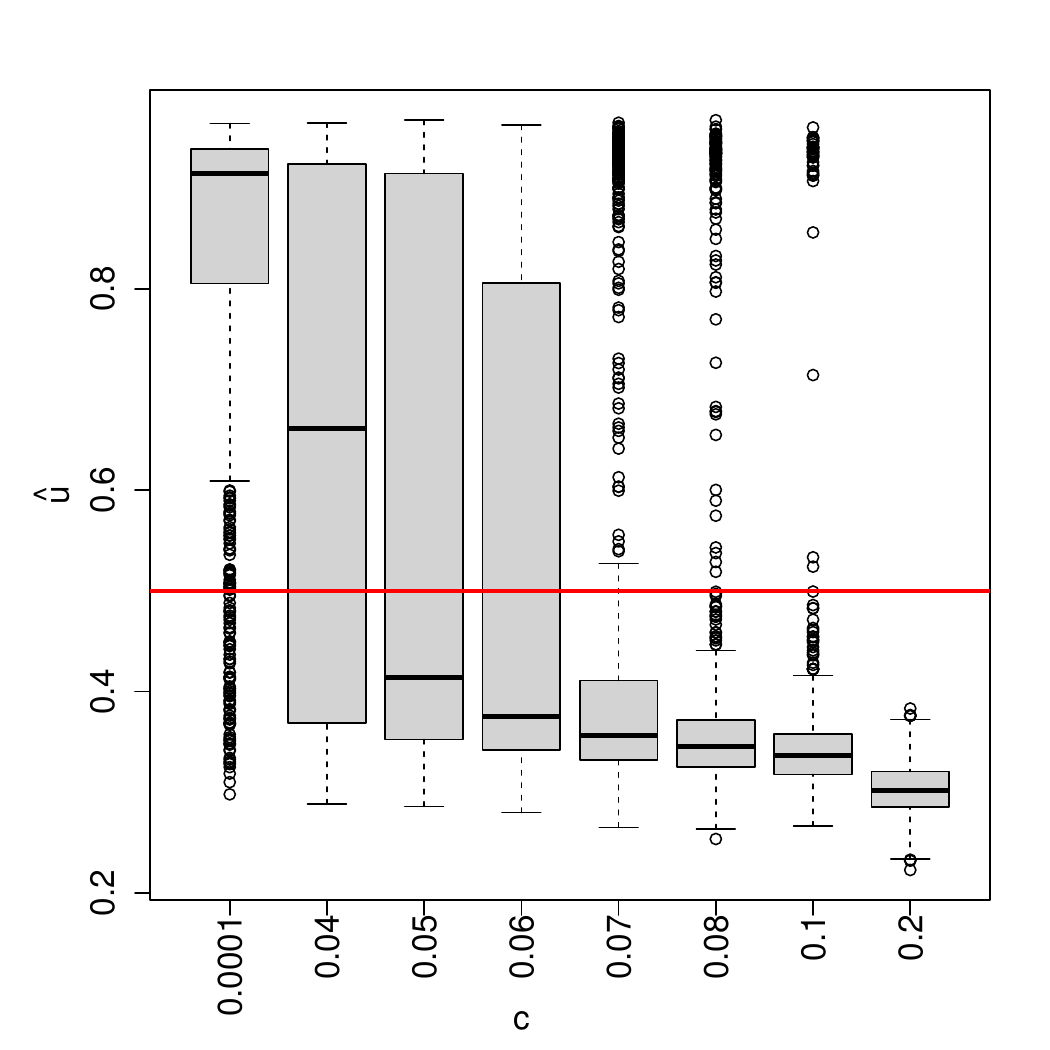}
\\

\multicolumn{3}{c}{$r_{\umbral_0,\delta}$, $\umbral_0=0.75$}\\[2ex]
\multicolumn{1}{c}{$\delta= \,-1$} & \multicolumn{1}{c}{$\delta= \,0$} & \multicolumn{1}{c}{$\delta= \,1$}\\[-2ex]
\includegraphics[scale=0.33]{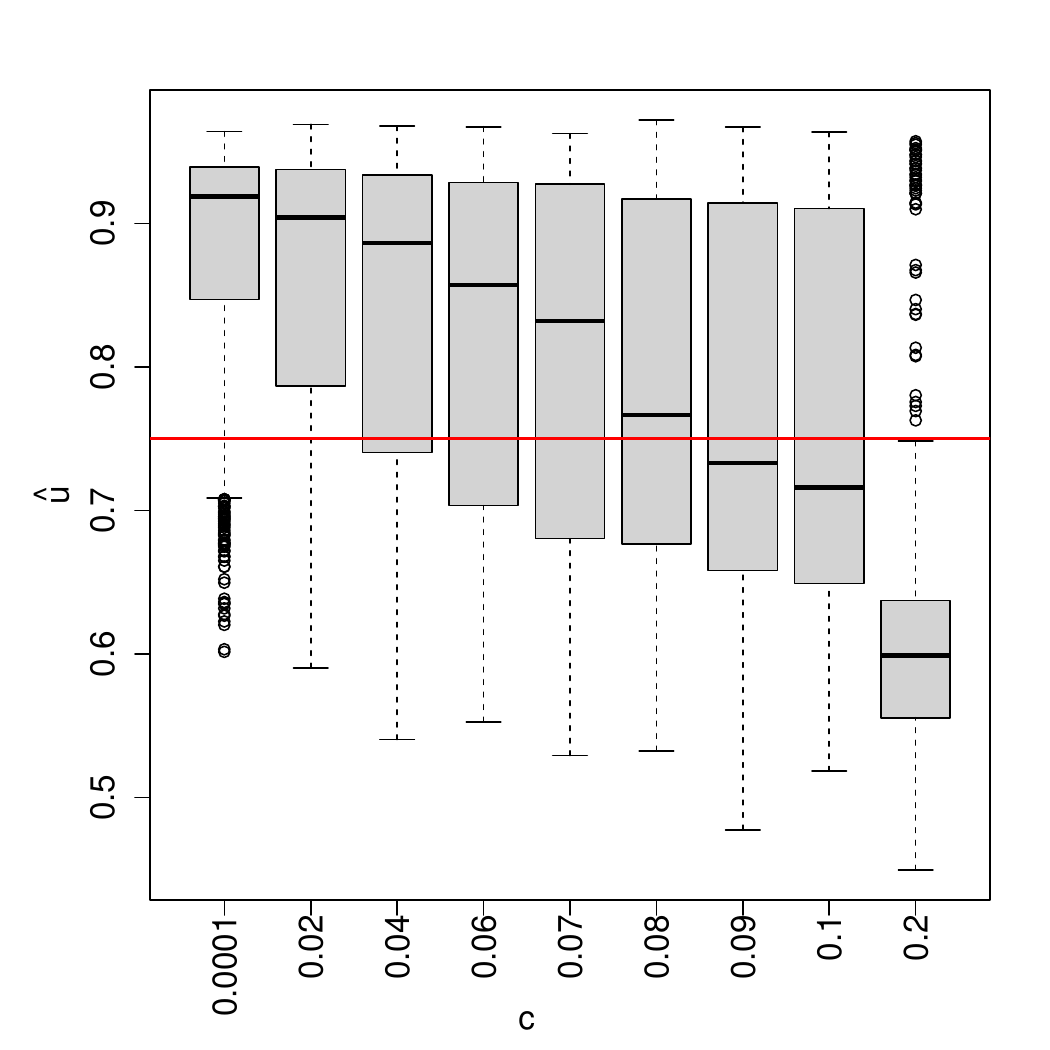} &
\includegraphics[scale=0.33]{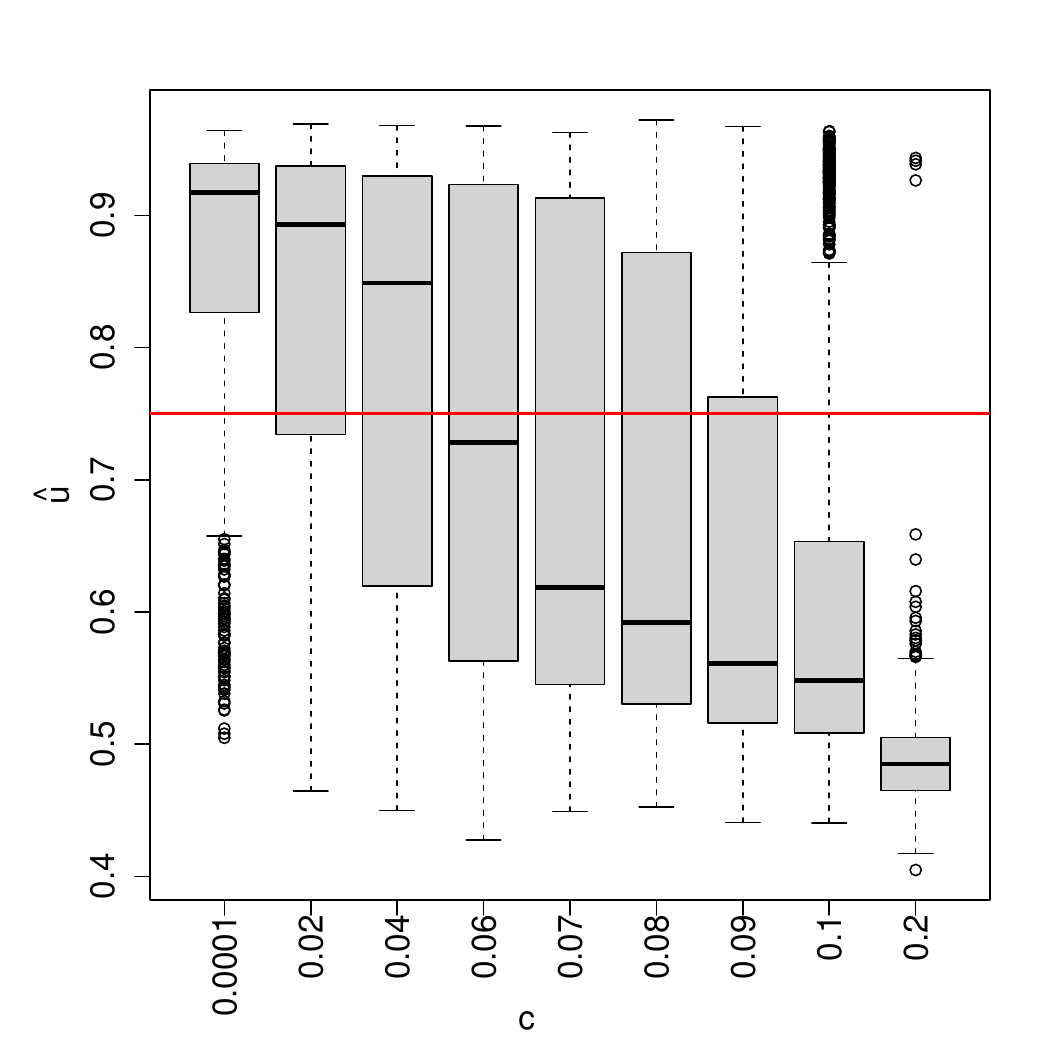}
&
\includegraphics[scale=0.33]{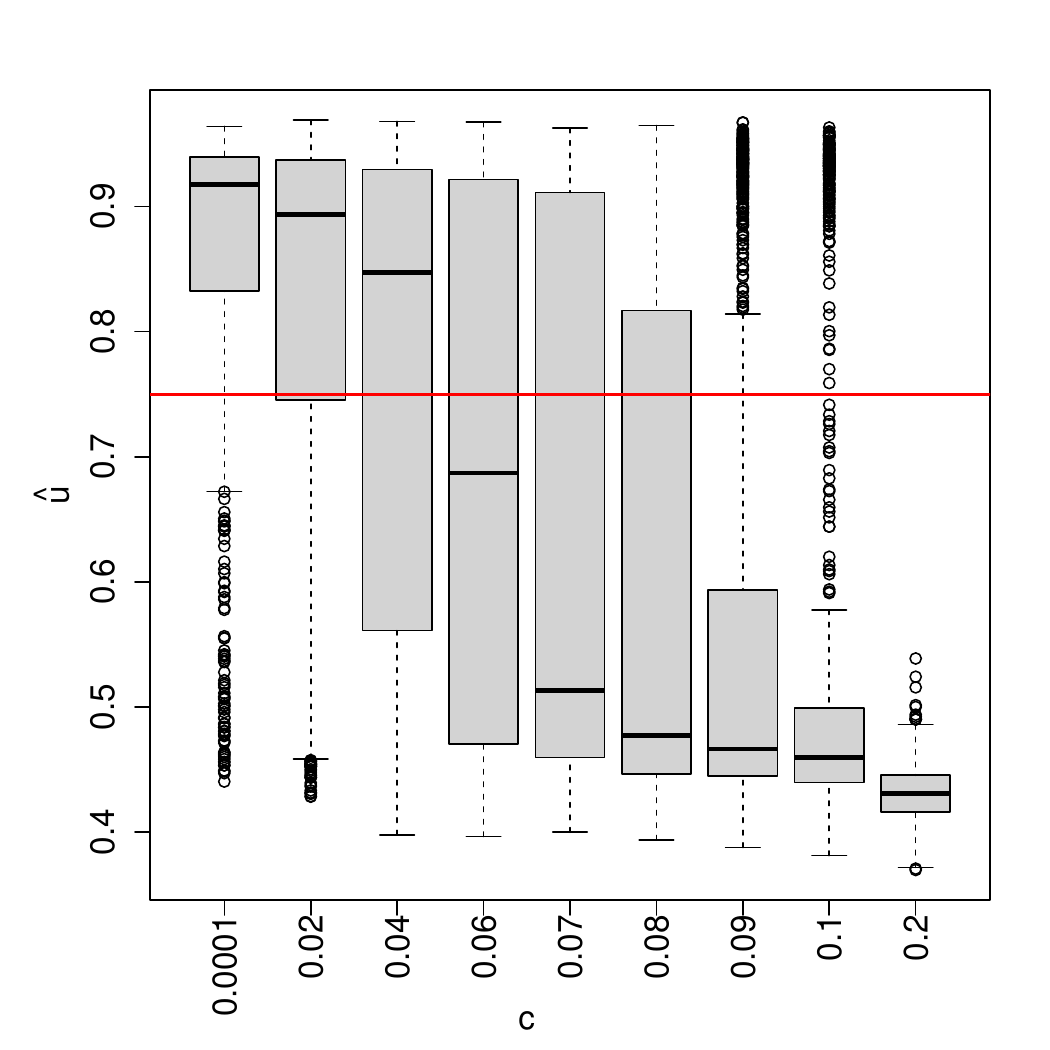}
\end{tabular}
\vskip-0.1in
\caption{\small \label{fig:bxp-uhat-10}	 Boxplots of the estimators $\widehat{\umbral}$ for different choices of the  penalizing constant $c$,   when $n=500$  and $\sigma=0.10$. The horizontal red line corresponds to the true value $\umbral_0$.} 
\end{center}
\end{figure}

\begin{figure}[ht!]
\begin{center}
\renewcommand{\arraystretch}{0.1}
\newcolumntype{G}{>{\centering\arraybackslash}m{\dimexpr.33\linewidth-1\tabcolsep}}
\begin{tabular}{GGG}
\multicolumn{3}{c}{$r_{\umbral_0,\delta}$, $\umbral_0=0.5$}\\[2ex]
\multicolumn{1}{c}{$\delta= \,-1$} & \multicolumn{1}{c}{$\delta= \,0$} & \multicolumn{1}{c}{$\delta= \,1$}\\[-2ex]
\includegraphics[scale=0.33]{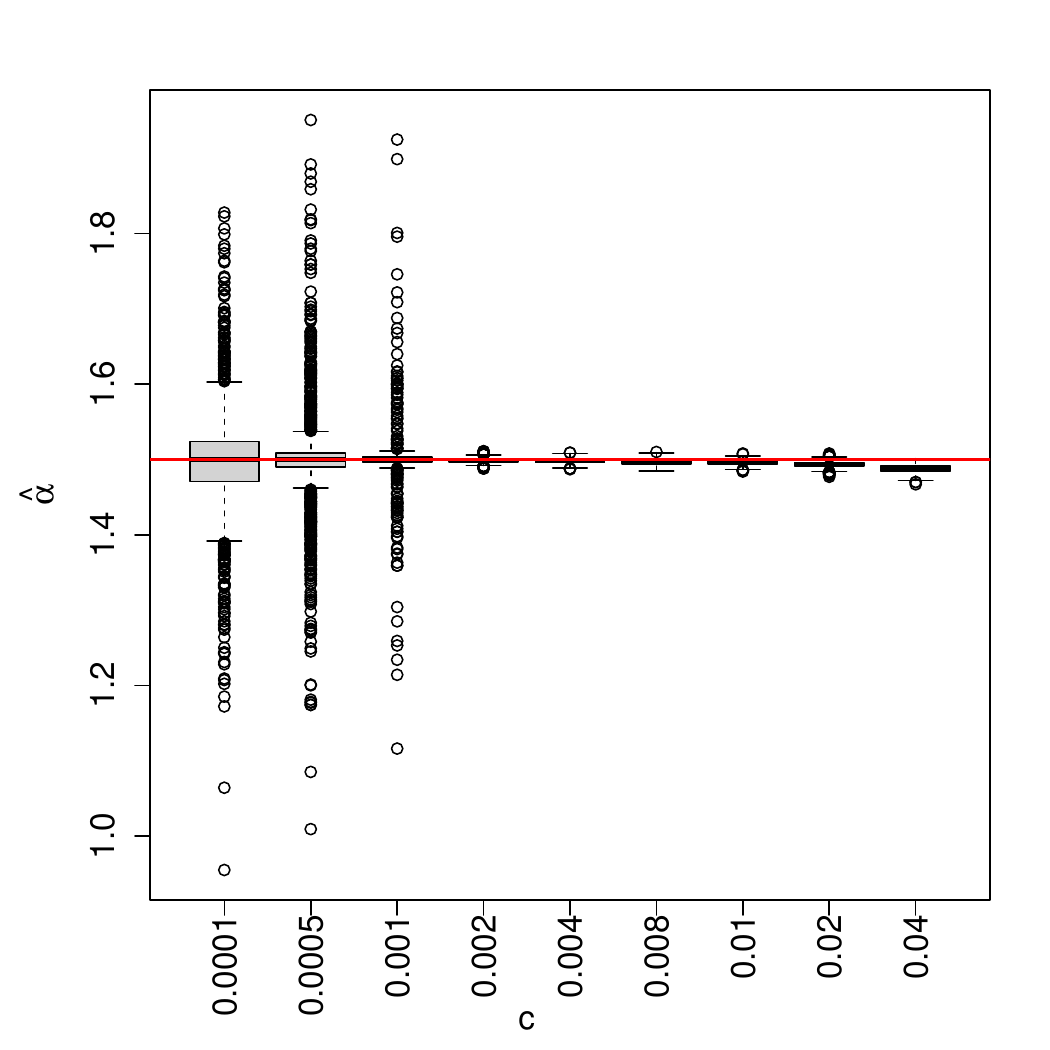} &
\includegraphics[scale=0.33]{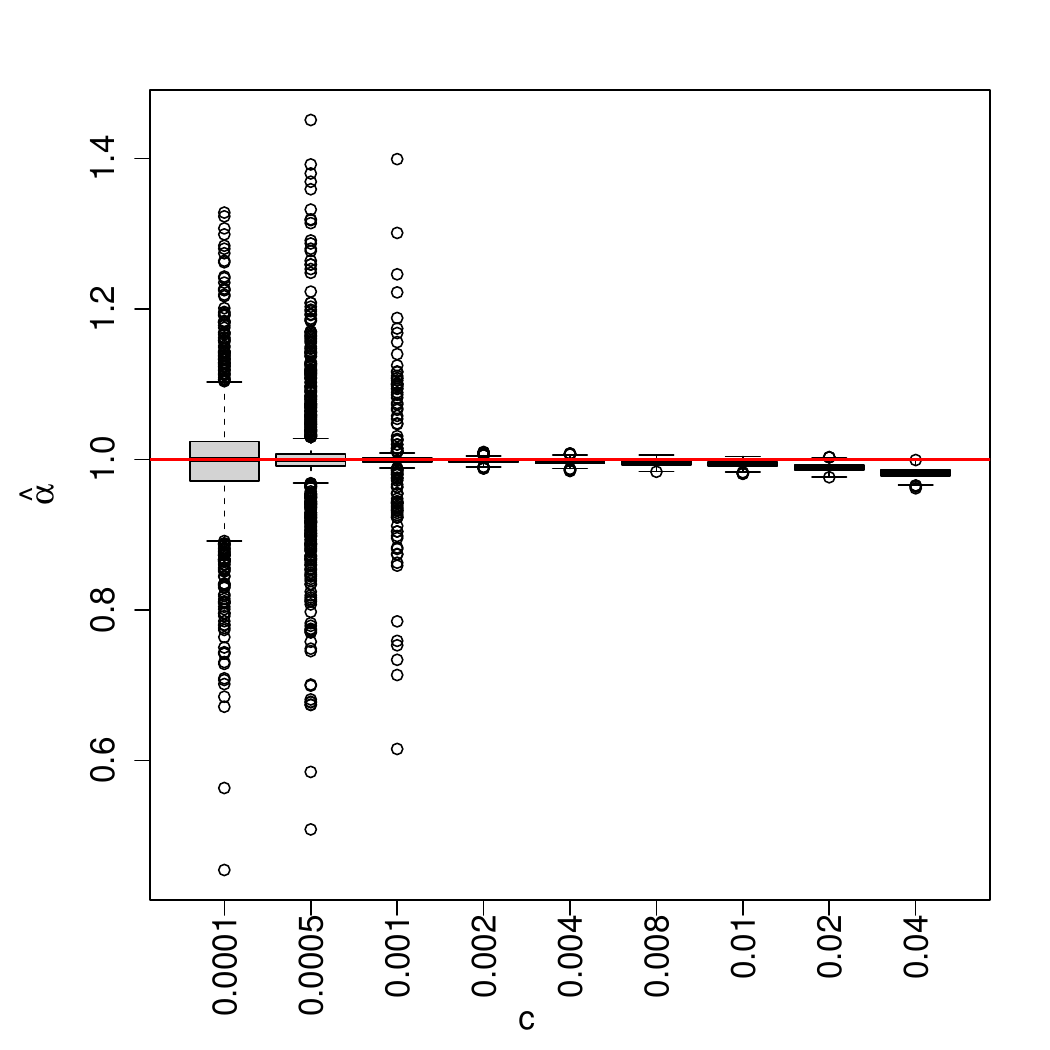}
&
\includegraphics[scale=0.33]{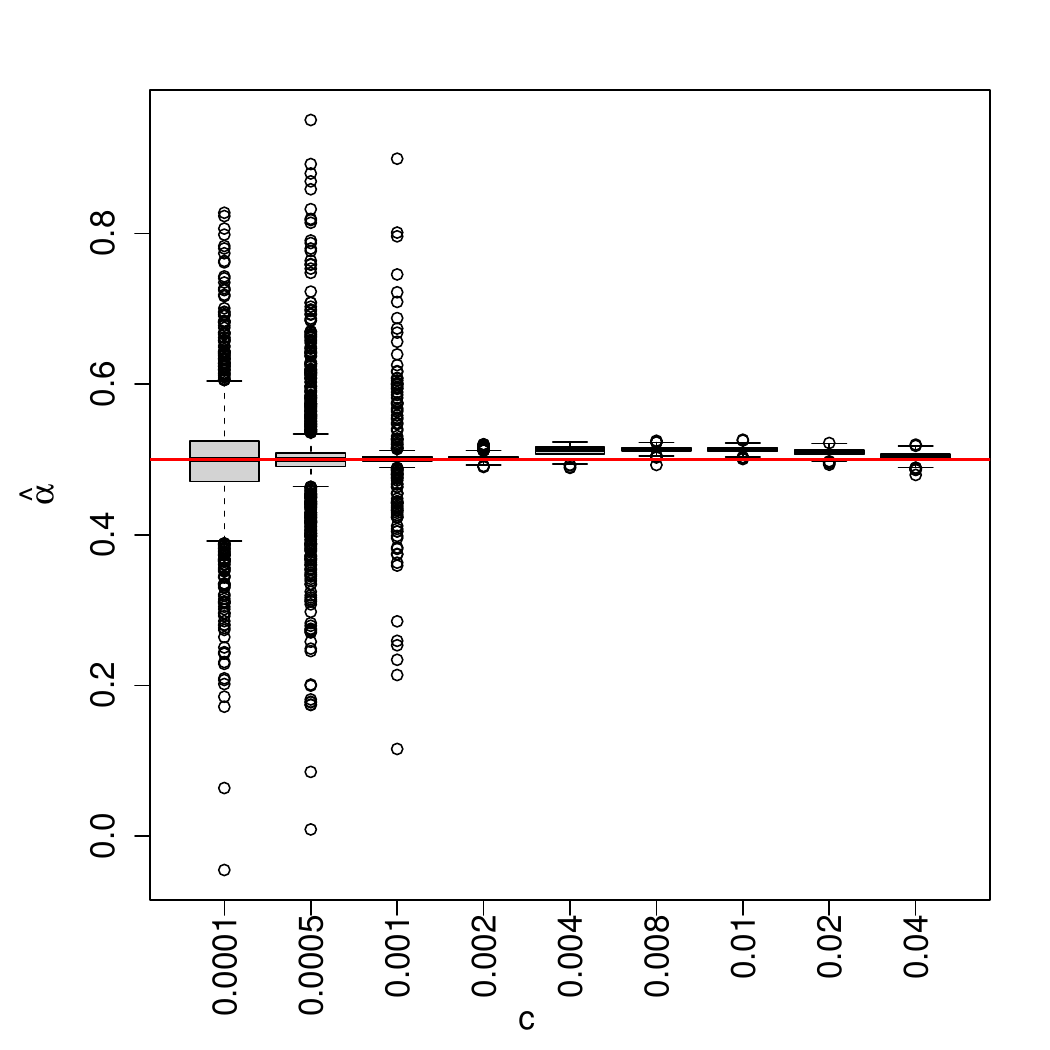}
\\

\multicolumn{3}{c}{$r_{\umbral_0,\delta}$, $\umbral_0=0.75$}\\[2ex]
\multicolumn{1}{c}{$\delta= \,-1$} & \multicolumn{1}{c}{$\delta= \,0$} & \multicolumn{1}{c}{$\delta= \,1$}\\[-2ex]
\includegraphics[scale=0.33]{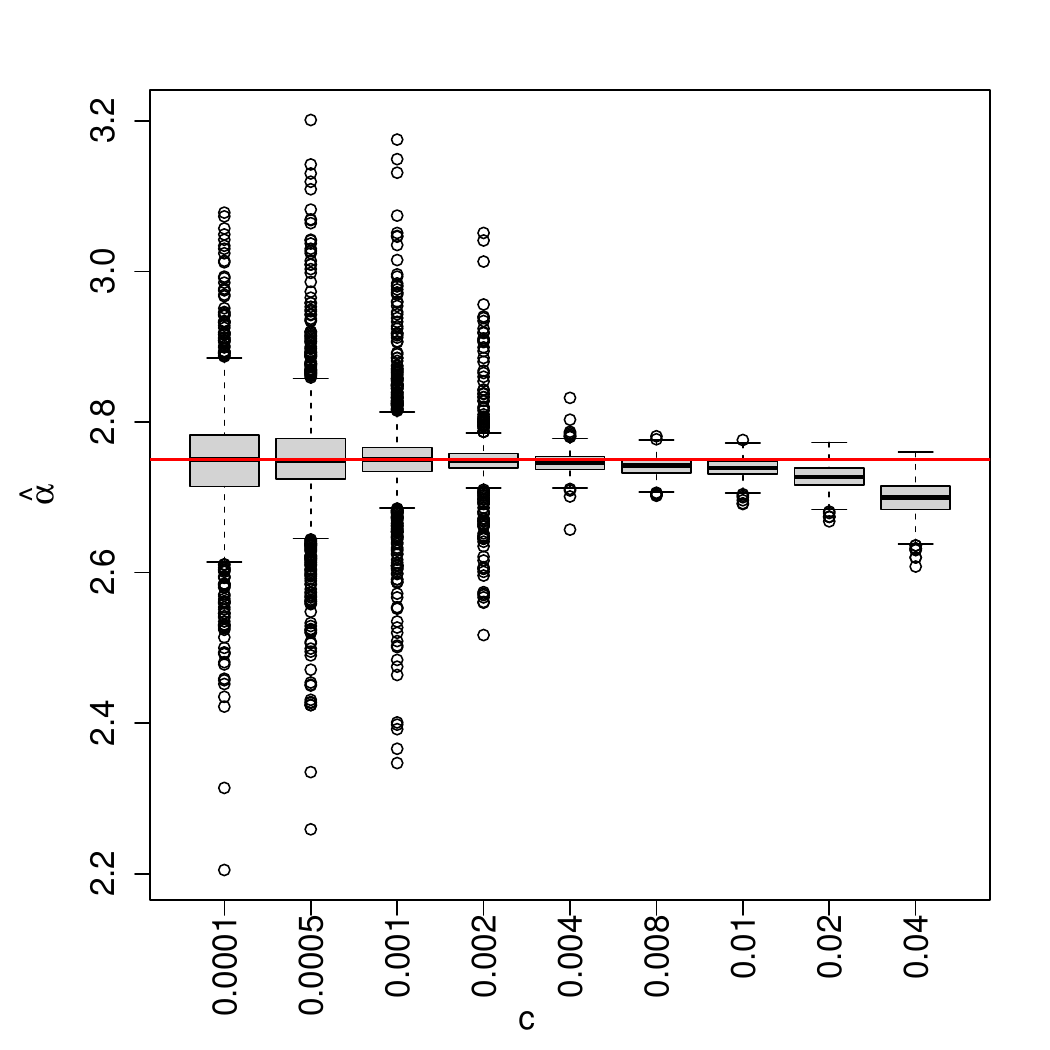} &
\includegraphics[scale=0.33]{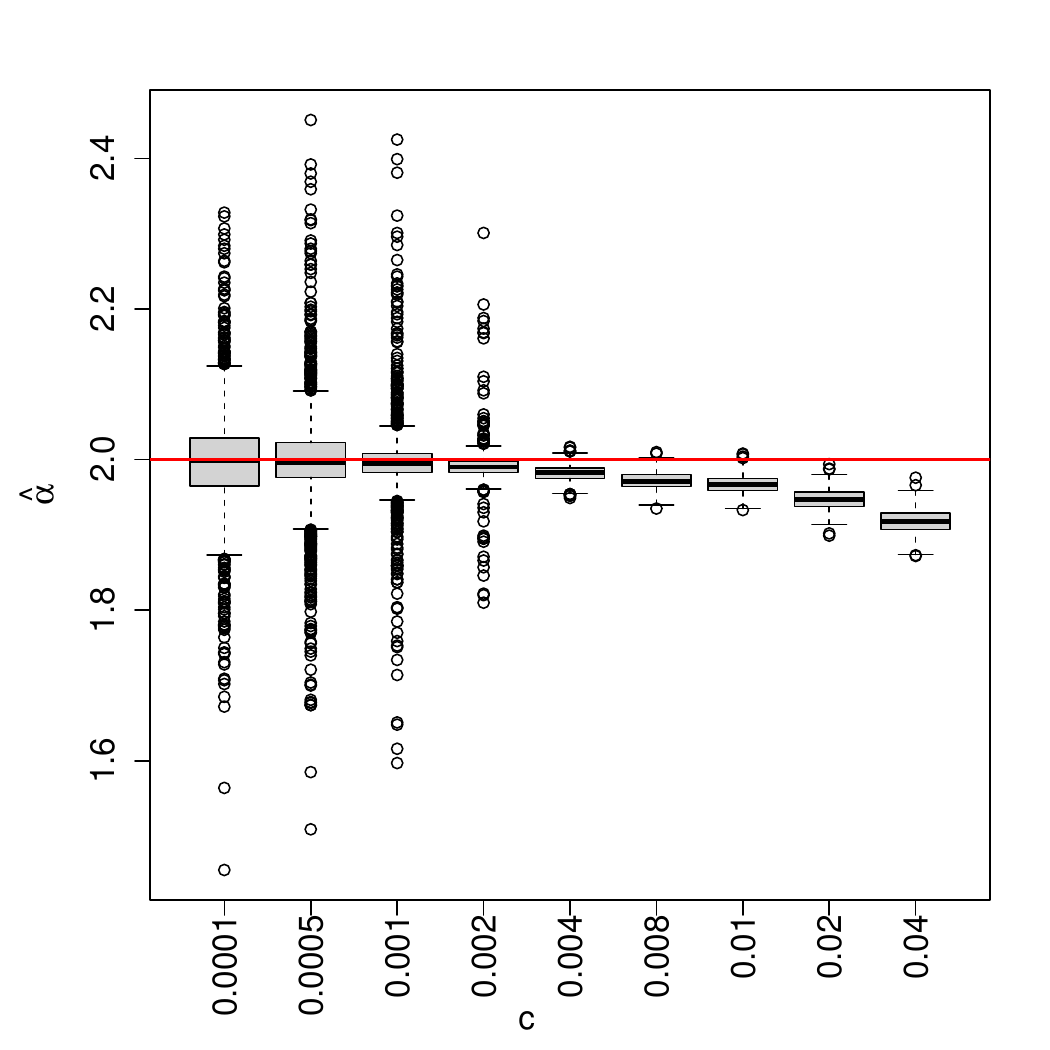} &
\includegraphics[scale=0.33]{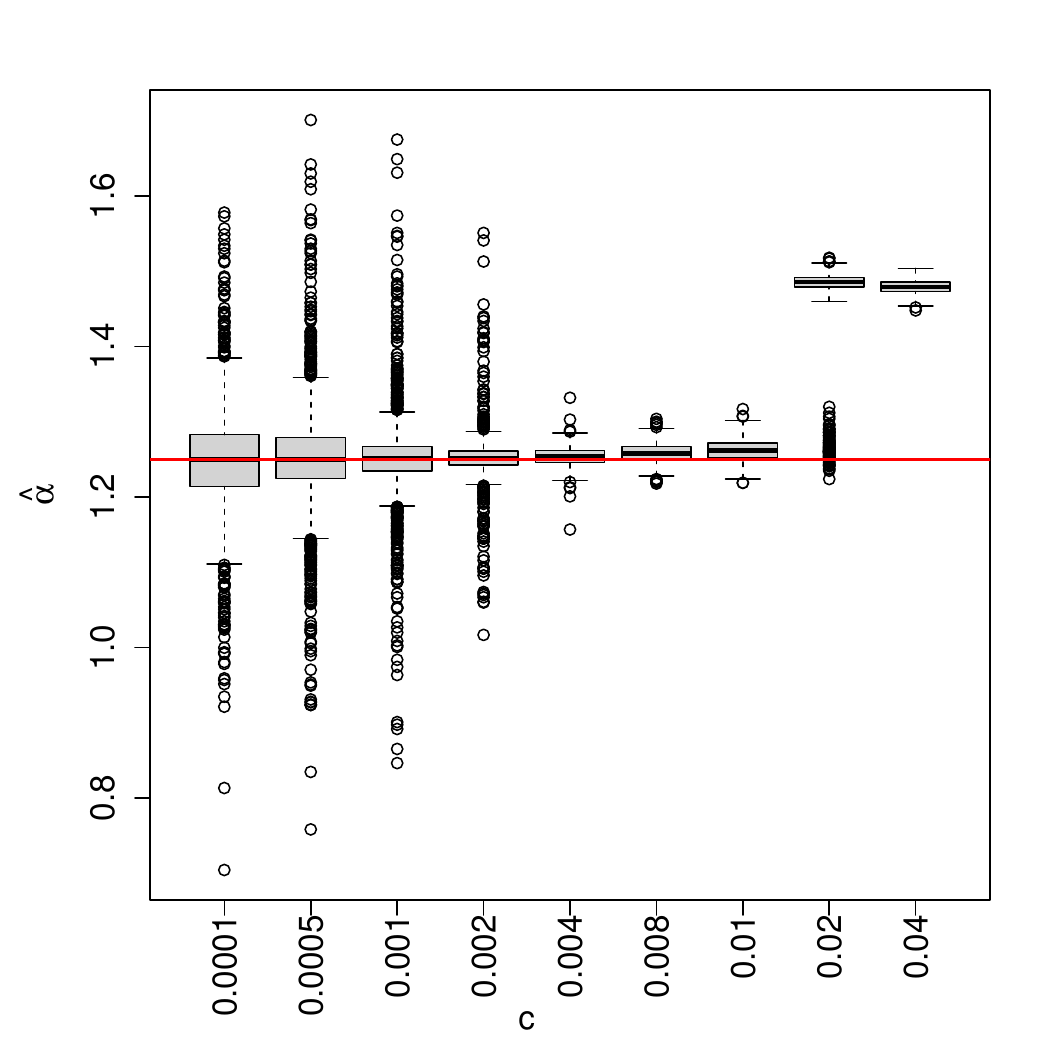}
\end{tabular}
\vskip-0.1in
\caption{\small \label{fig:bxp-alfa}	 Boxplots of the estimators $\widehat{\alpha}$ for different choices of the  penalizing constant $c$,   when $n=500$  and $\sigma=0.01$. The horizontal red line corresponds to the  value $\alpha_0$.} 
\end{center}
\end{figure}

Figures \ref{fig:bxp-alfa} and \ref{fig:bxp-beta} present the boxplots of the estimators $\widehat{\alpha}_{\widehat{\umbral}}$ and  $\widehat{\beta}_{\widehat{\umbral}}$, respectively when $\sigma=0.01$. The boxplots when $\sigma=0.05$ and $0.10$ are displayed in Figures \ref{fig:bxp-alfa-5} to \ref{fig:bxp-beta-10}.   In each Figure, the horizontal solid red line corresponds to the true parameters $\alpha_0$ and $\beta_0$, respectively. Taking into account that for values of $c$ smaller than $0.001$, the threshold  estimates are mainly  larger or equal  than the target,  the estimates of the slope and intercept accomplish the goal since their boxplots are centered around the true values. On the contrary,  for large values of $c$ some of the observations used to estimate the slope and intercept correspond to data generated using the function $g(x)$, that is, with  the nonparametric model component and for that reason, they  fail in their purpose.

\begin{figure}[ht!]
\begin{center}
\renewcommand{\arraystretch}{0.1}
\newcolumntype{G}{>{\centering\arraybackslash}m{\dimexpr.33\linewidth-1\tabcolsep}}
\begin{tabular}{GGG}
\multicolumn{3}{c}{$r_{\umbral_0,\delta}$, $\umbral_0=0.5$}\\[2ex]
\multicolumn{1}{c}{$\delta= \,-1$} & \multicolumn{1}{c}{$\delta= \,0$} & \multicolumn{1}{c}{$\delta= \,1$}\\[-2ex]
\includegraphics[scale=0.33]{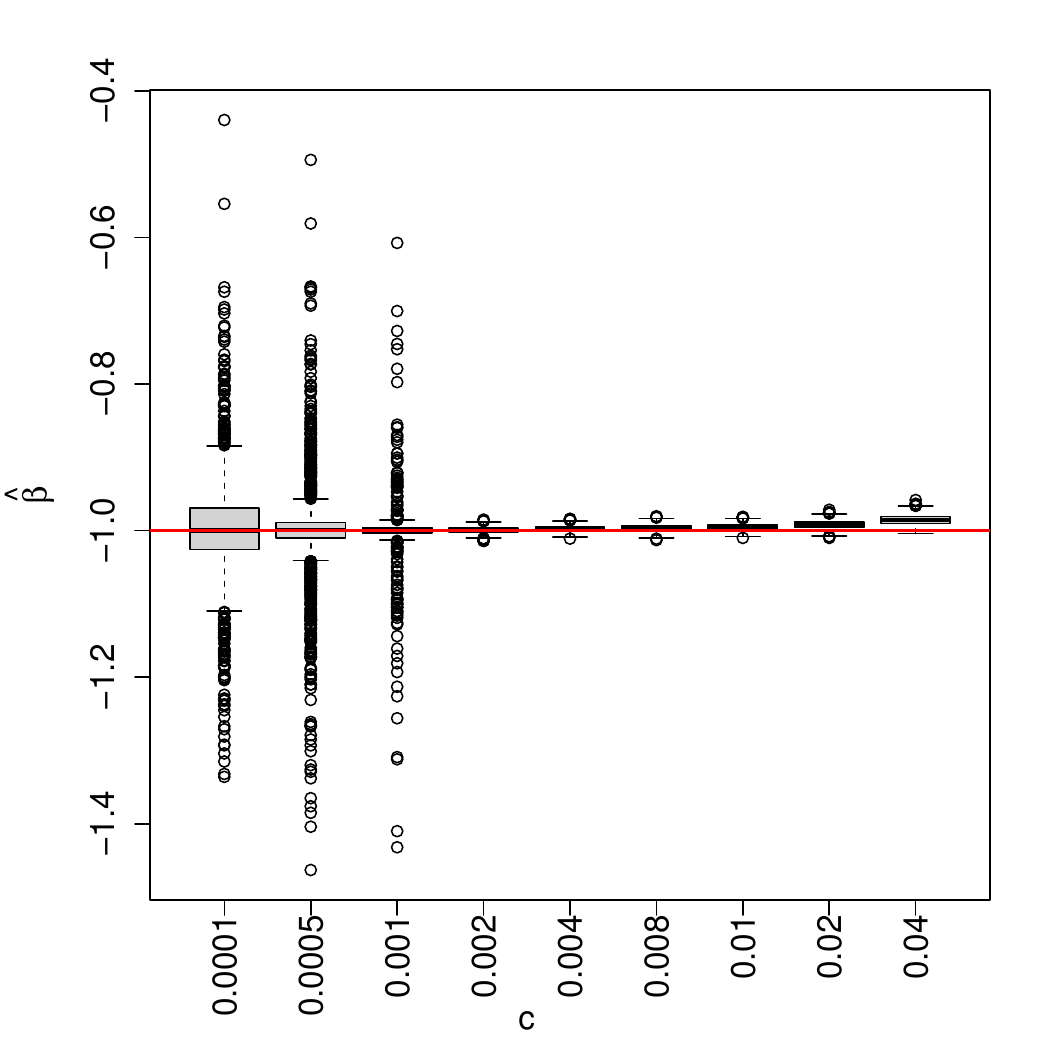} &
\includegraphics[scale=0.33]{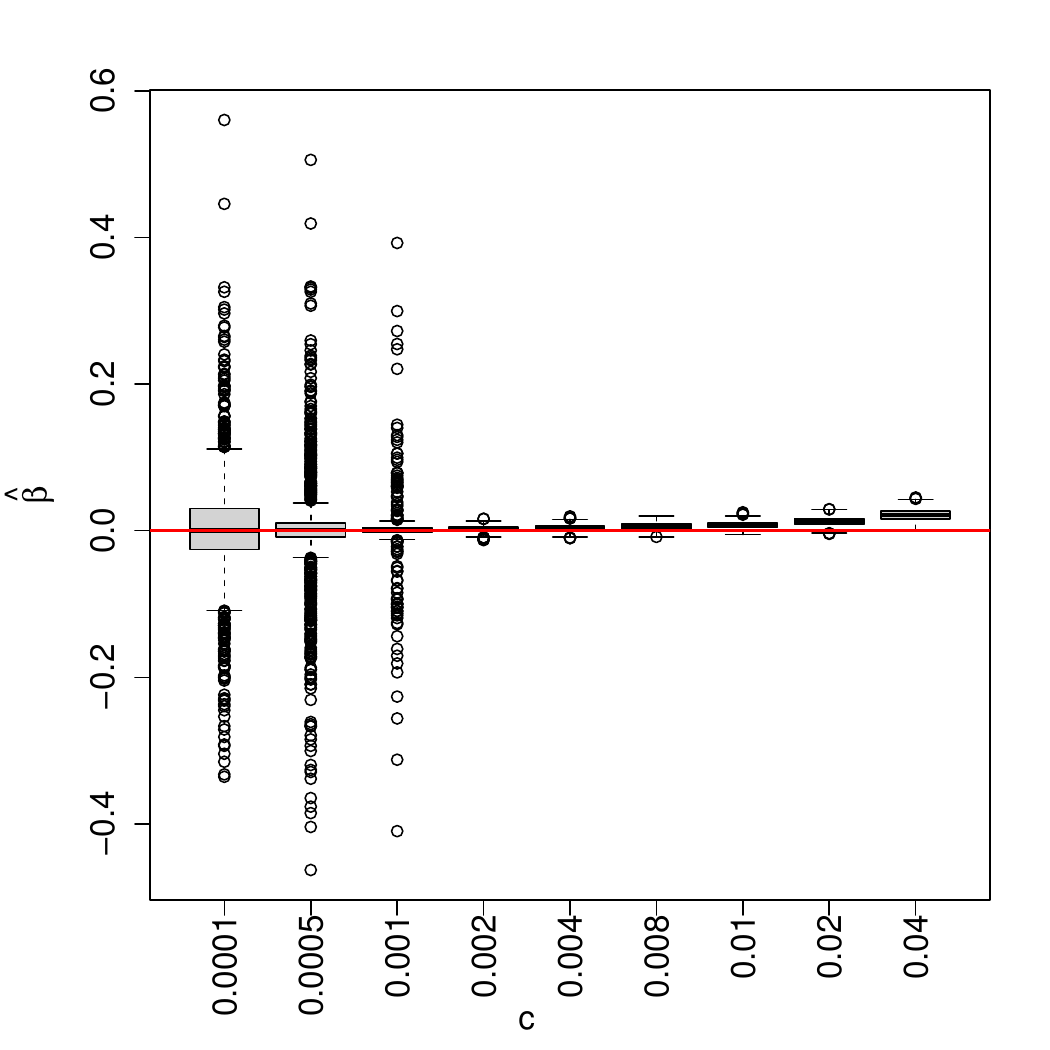}
&
\includegraphics[scale=0.33]{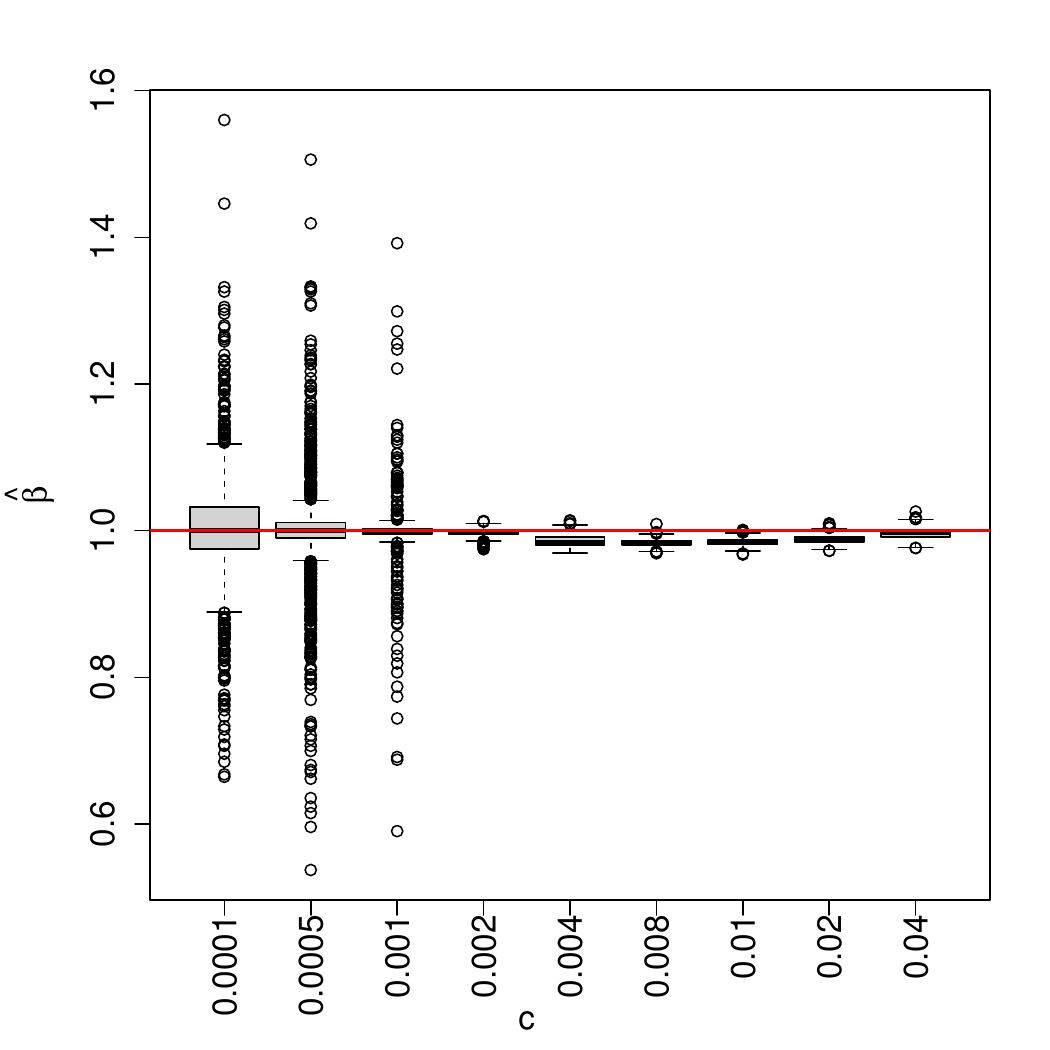}
\\

\multicolumn{3}{c}{$r_{\umbral_0,\delta}$, $\umbral_0=0.75$}\\[2ex]
\multicolumn{1}{c}{$\delta= \,-1$} & \multicolumn{1}{c}{$\delta= \,0$} & \multicolumn{1}{c}{$\delta= \,1$}\\[-2ex]
\includegraphics[scale=0.33]{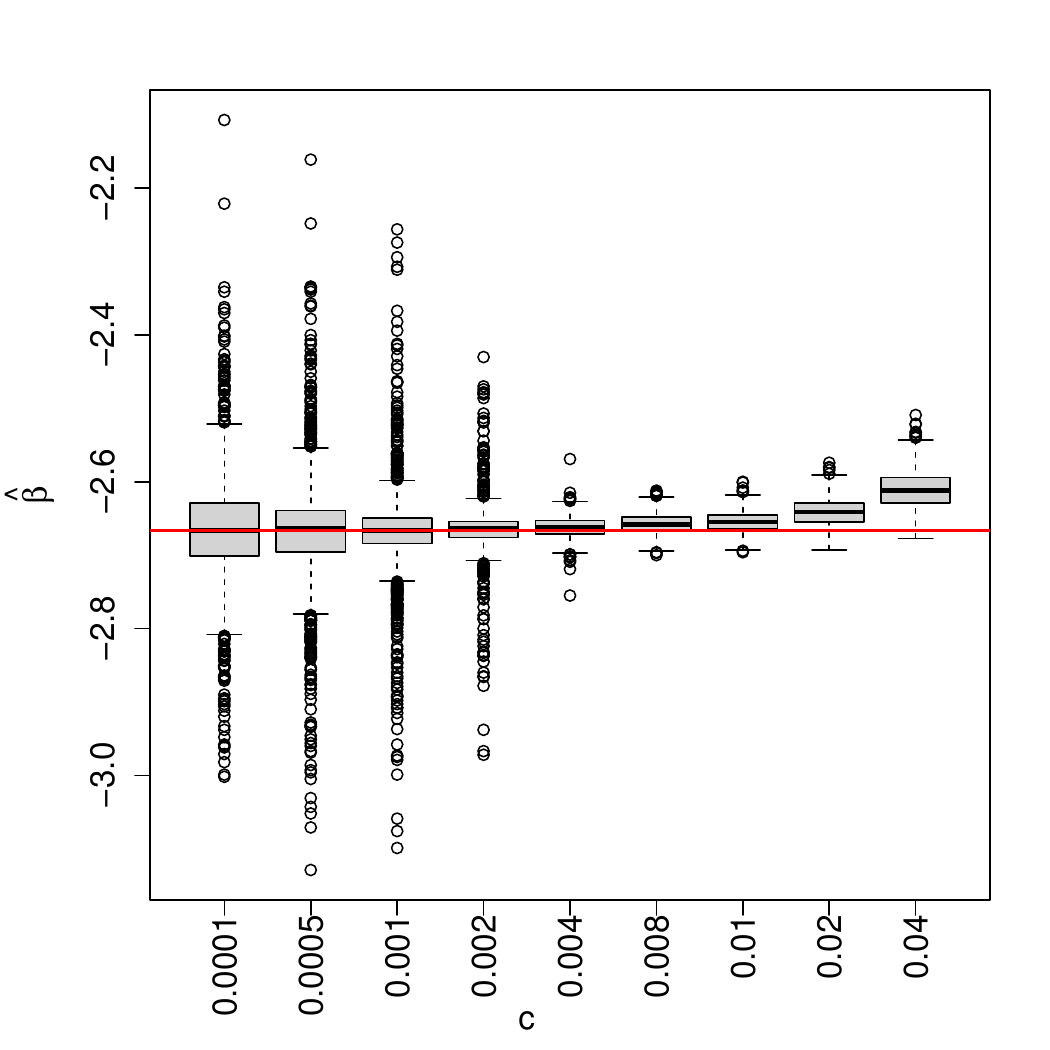} &
\includegraphics[scale=0.33]{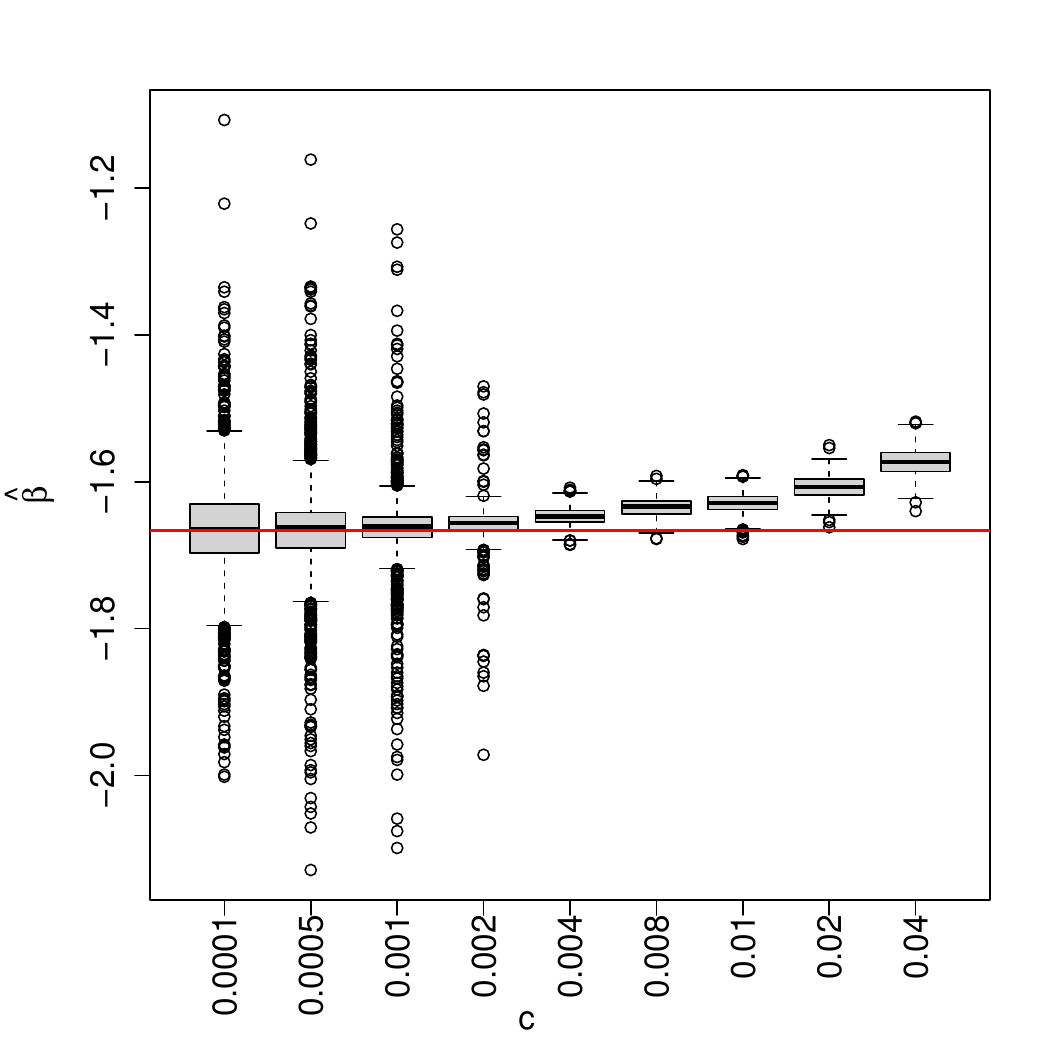}
&
\includegraphics[scale=0.33]{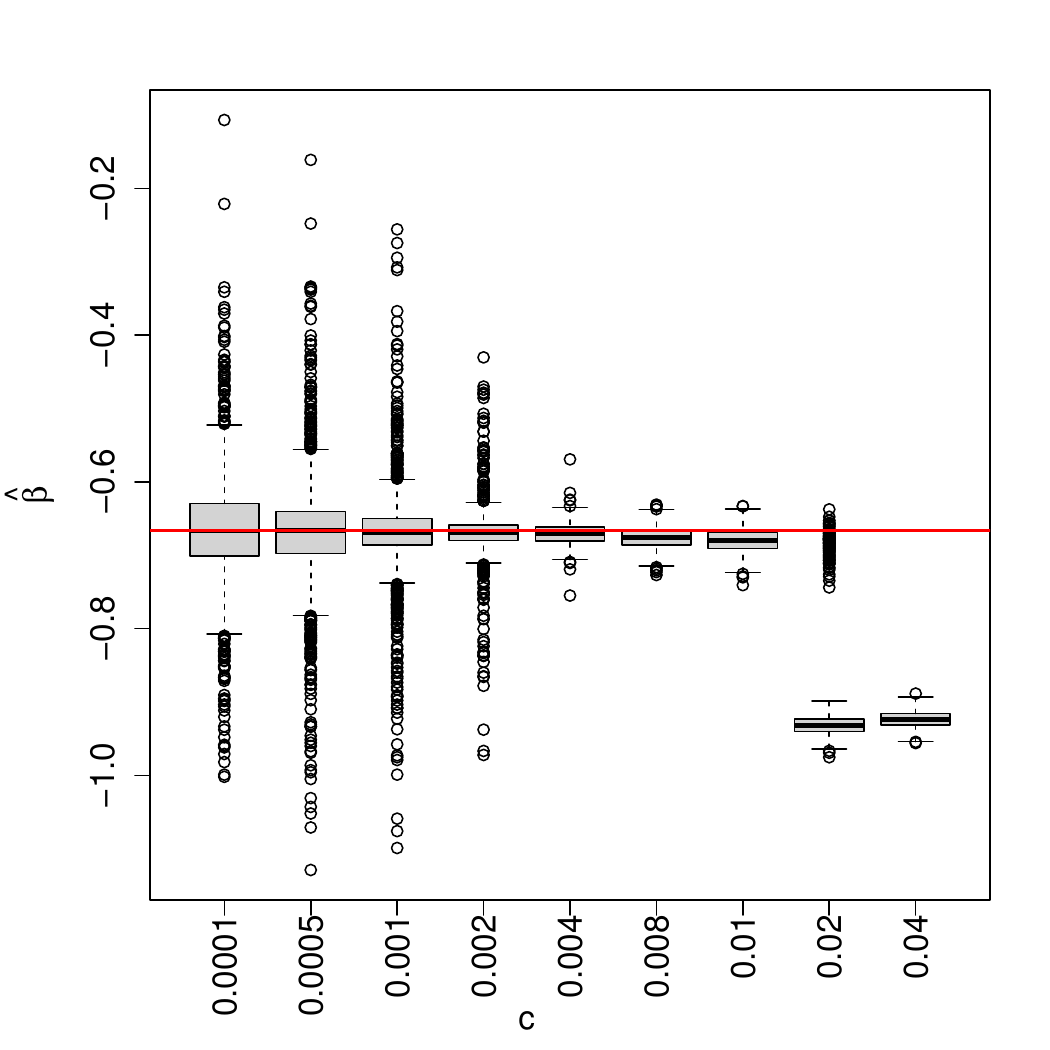}
\end{tabular}
\vskip-0.1in
\caption{\small \label{fig:bxp-beta}	 Boxplots of the estimators $\widehat{\beta}$ for different choices of the  penalizing constant $c$,   when $n=500$  and $\sigma=0.01$. The horizontal red line corresponds to the  value $\beta_0$.} 
\end{center}
\end{figure}

\begin{figure}[ht!]
\begin{center}
\renewcommand{\arraystretch}{0.1}
\newcolumntype{G}{>{\centering\arraybackslash}m{\dimexpr.33\linewidth-1\tabcolsep}}
\begin{tabular}{GGG}
\multicolumn{3}{c}{$r_{\umbral_0,\delta}$, $\umbral_0=0.5$}\\[2ex]
\multicolumn{1}{c}{$\delta= \,-1$} & \multicolumn{1}{c}{$\delta= \,0$} & \multicolumn{1}{c}{$\delta= \,1$}\\[-2ex]
\includegraphics[scale=0.33]{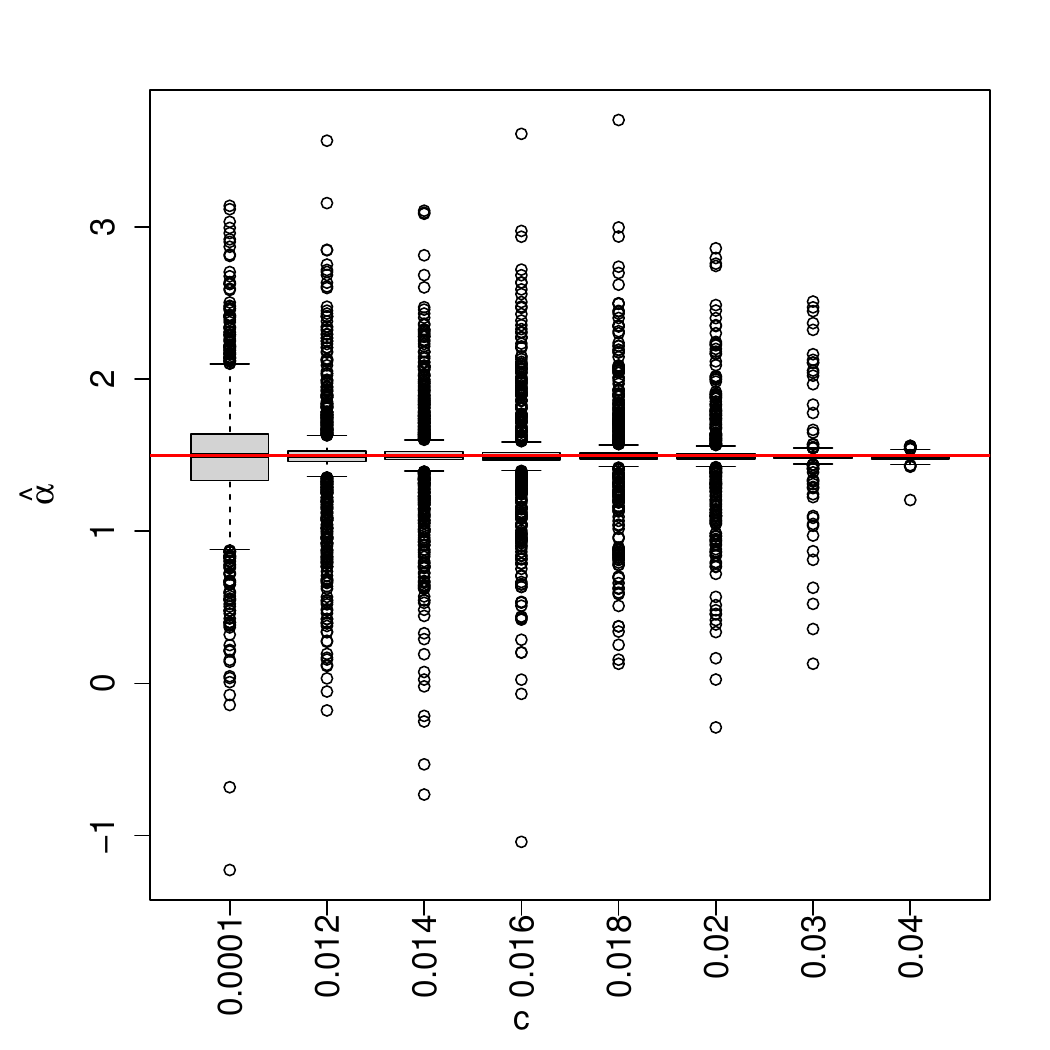} &
\includegraphics[scale=0.33]{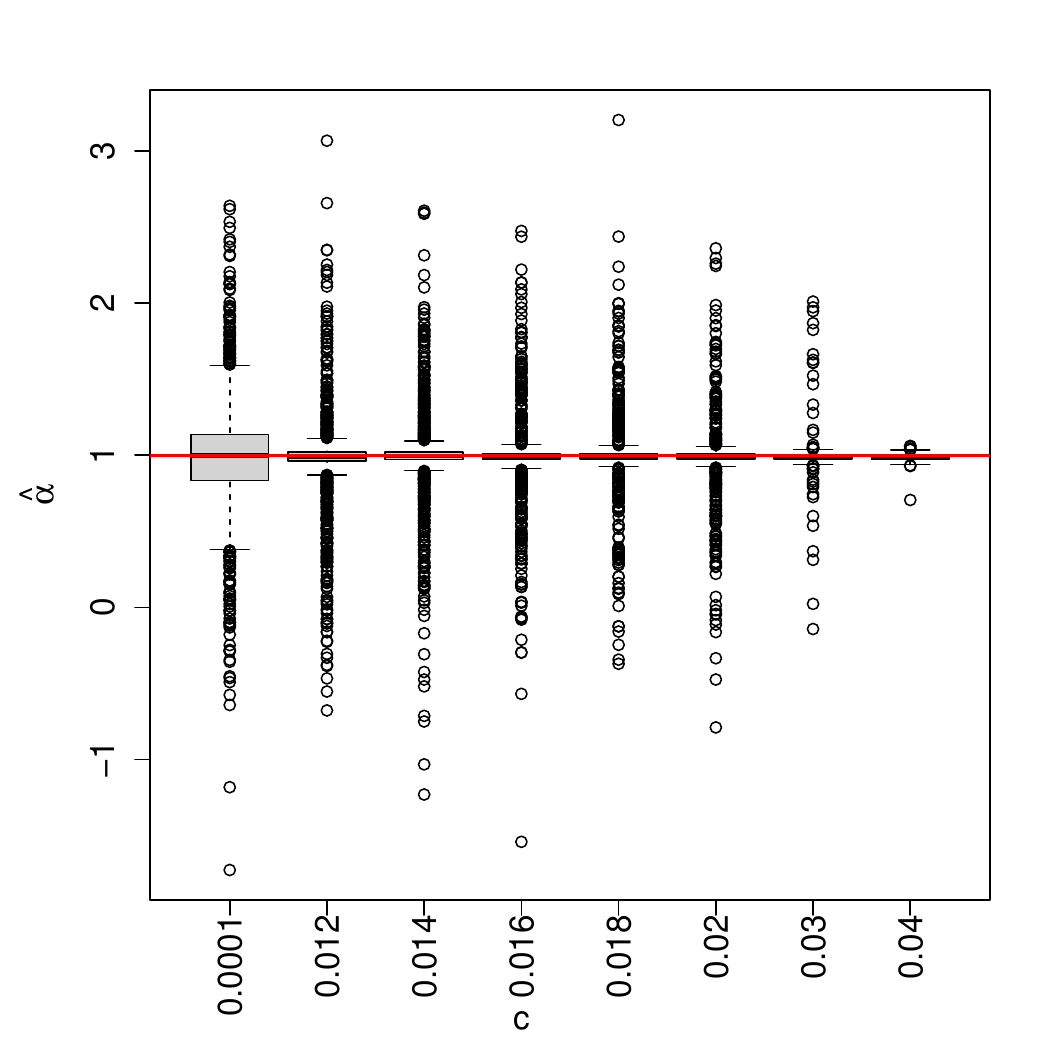}
&
\includegraphics[scale=0.33]{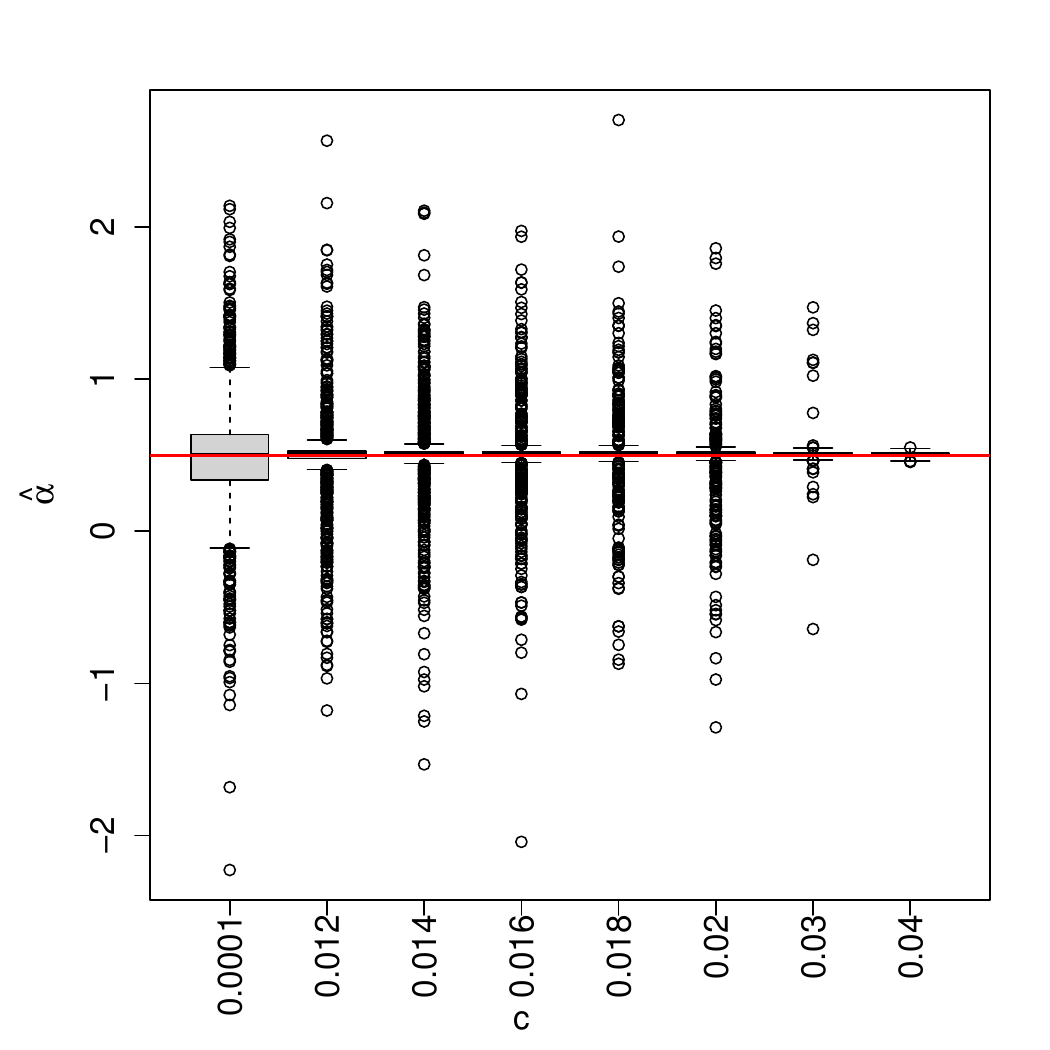}
\\

\multicolumn{3}{c}{$r_{\umbral_0,\delta}$, $\umbral_0=0.75$}\\[2ex]
\multicolumn{1}{c}{$\delta= \,-1$} & \multicolumn{1}{c}{$\delta= \,0$} & \multicolumn{1}{c}{$\delta= \,1$}\\[-2ex]
\includegraphics[scale=0.33]{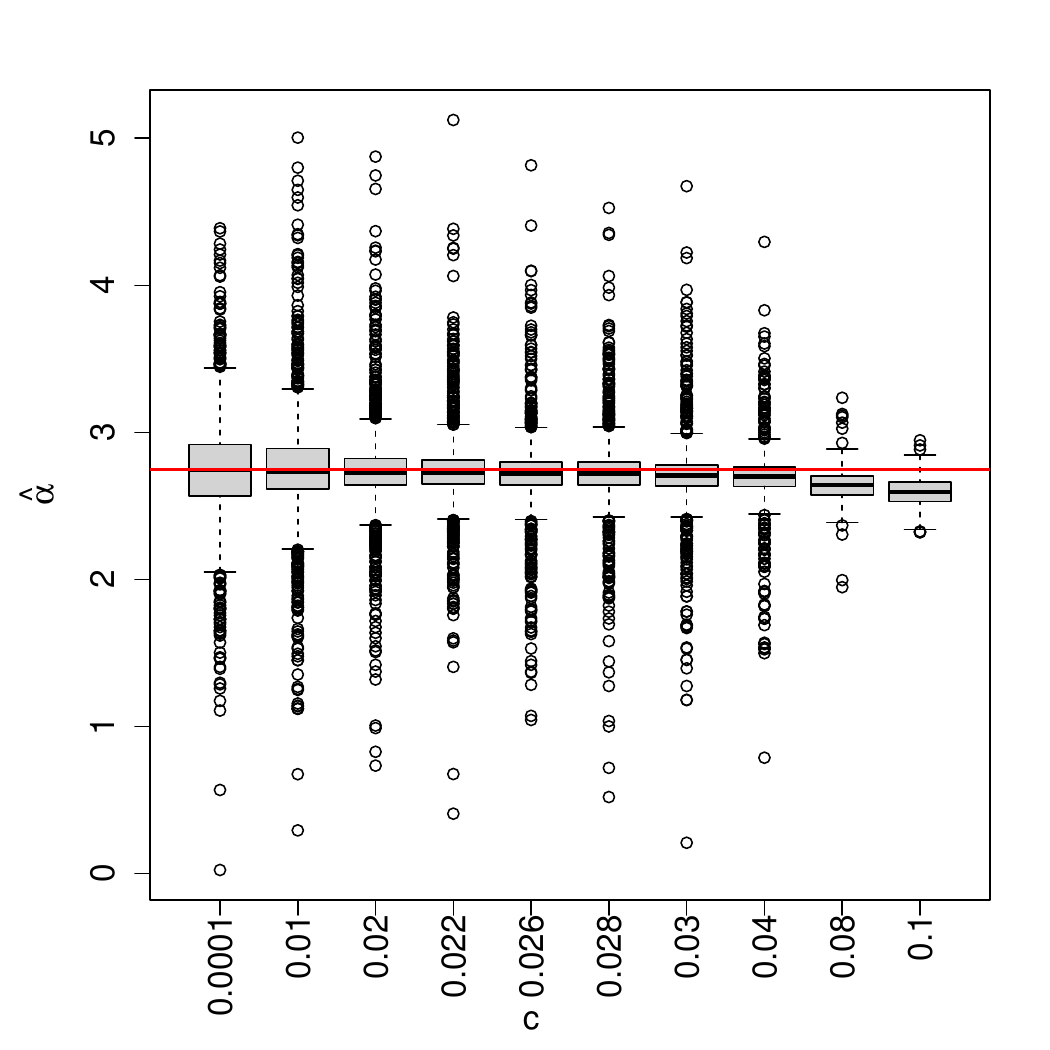} &
\includegraphics[scale=0.33]{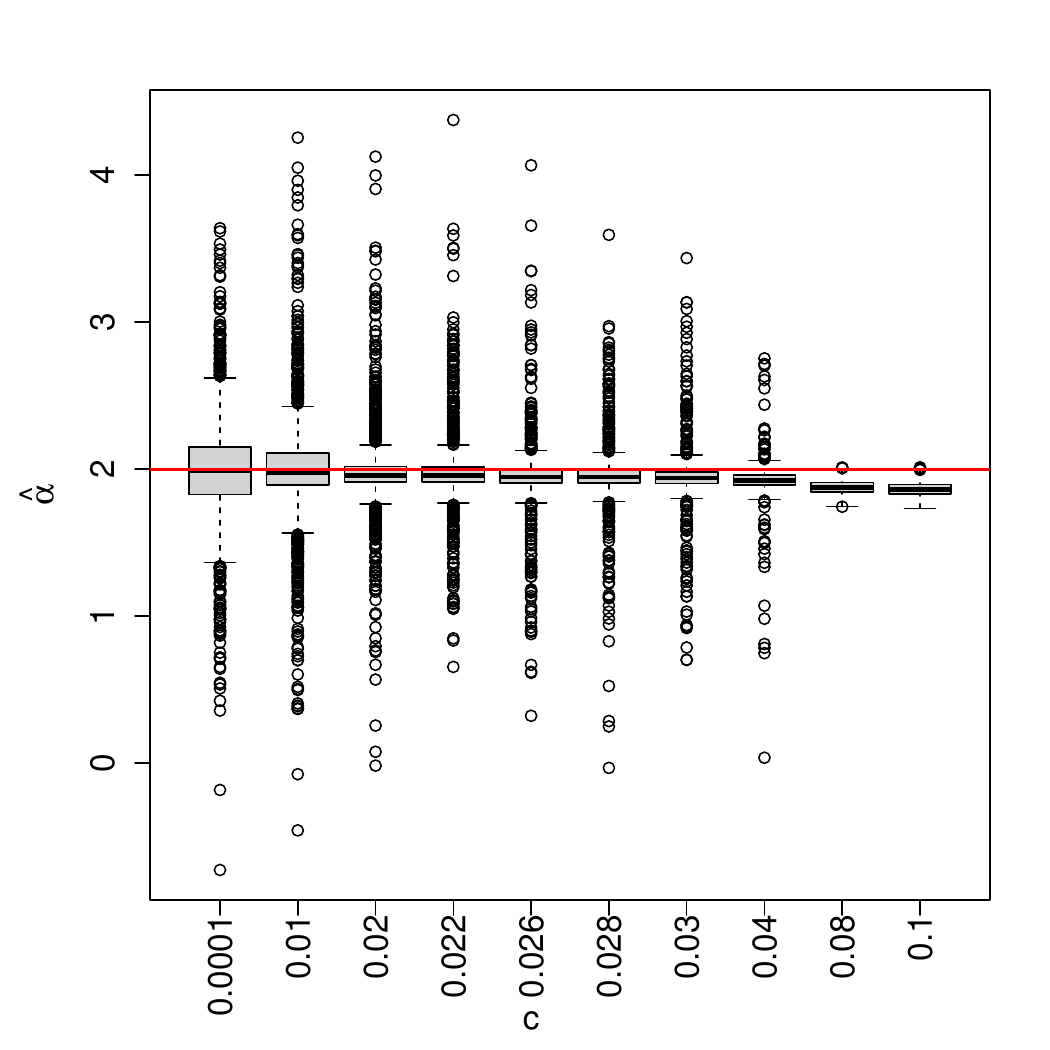} &
\includegraphics[scale=0.33]{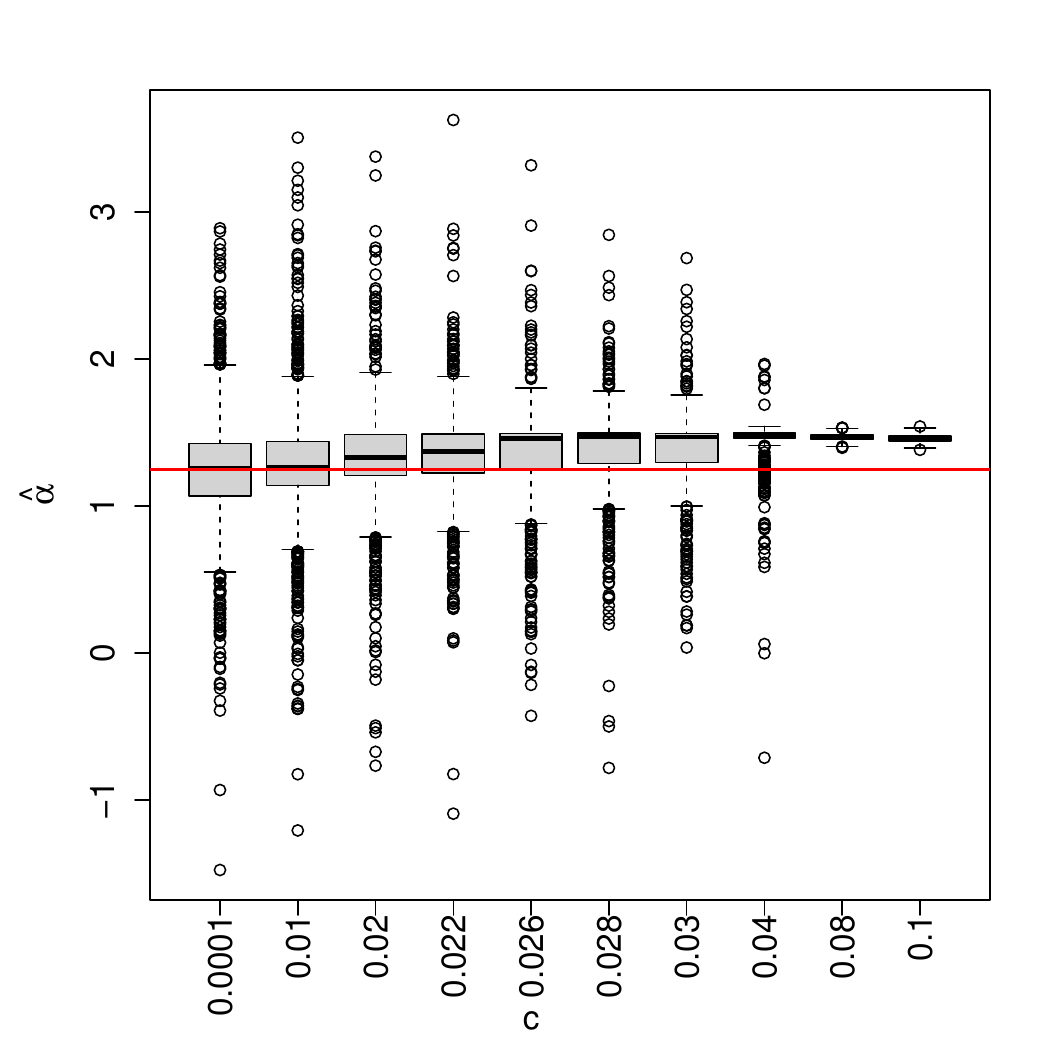}
\end{tabular}
\vskip-0.1in
\caption{\small \label{fig:bxp-alfa-5}	 Boxplots of the estimators $\widehat{\alpha}$ for different choices of the  penalizing constant $c$,   when $n=500$  and $\sigma=0.05$. The horizontal red line corresponds to the  value $\alpha_0$.} 
\end{center}
\end{figure}

\begin{figure}[ht!]
\begin{center}
\renewcommand{\arraystretch}{0.1}
\newcolumntype{G}{>{\centering\arraybackslash}m{\dimexpr.33\linewidth-1\tabcolsep}}
\begin{tabular}{GGG}
\multicolumn{3}{c}{$r_{\umbral_0,\delta}$, $\umbral_0=0.5$}\\[2ex]
\multicolumn{1}{c}{$\delta= \,-1$} & \multicolumn{1}{c}{$\delta= \,0$} & \multicolumn{1}{c}{$\delta= \,1$}\\[-2ex]
\includegraphics[scale=0.33]{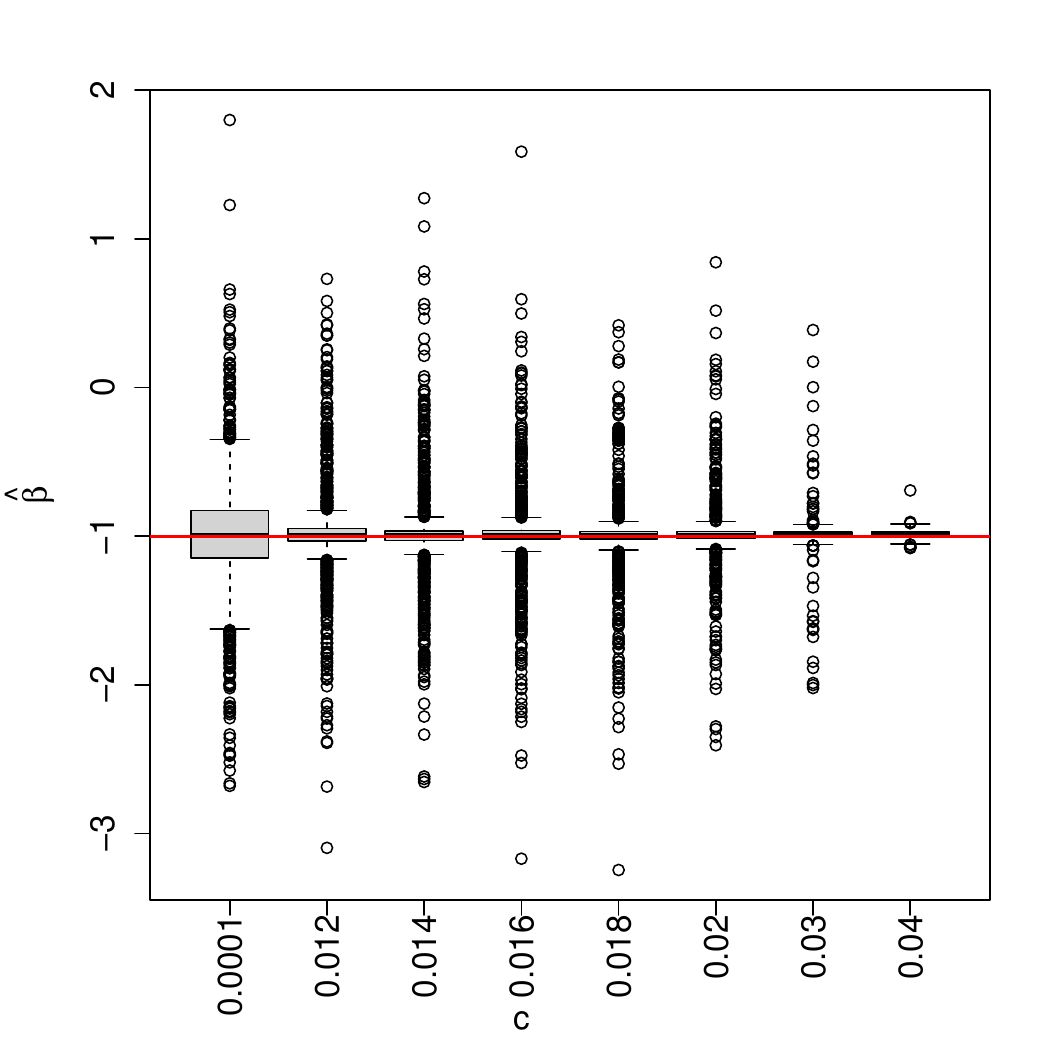} &
\includegraphics[scale=0.33]{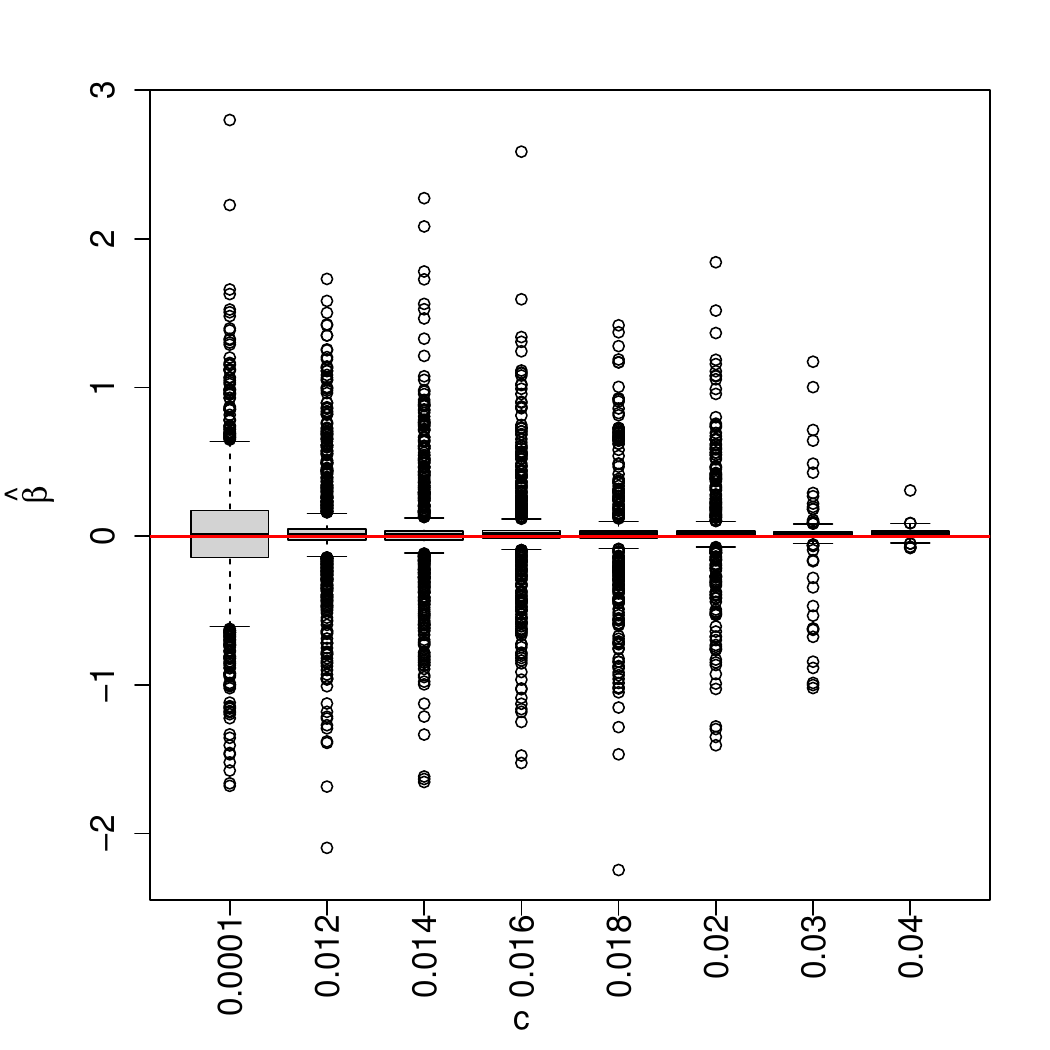}
&
\includegraphics[scale=0.33]{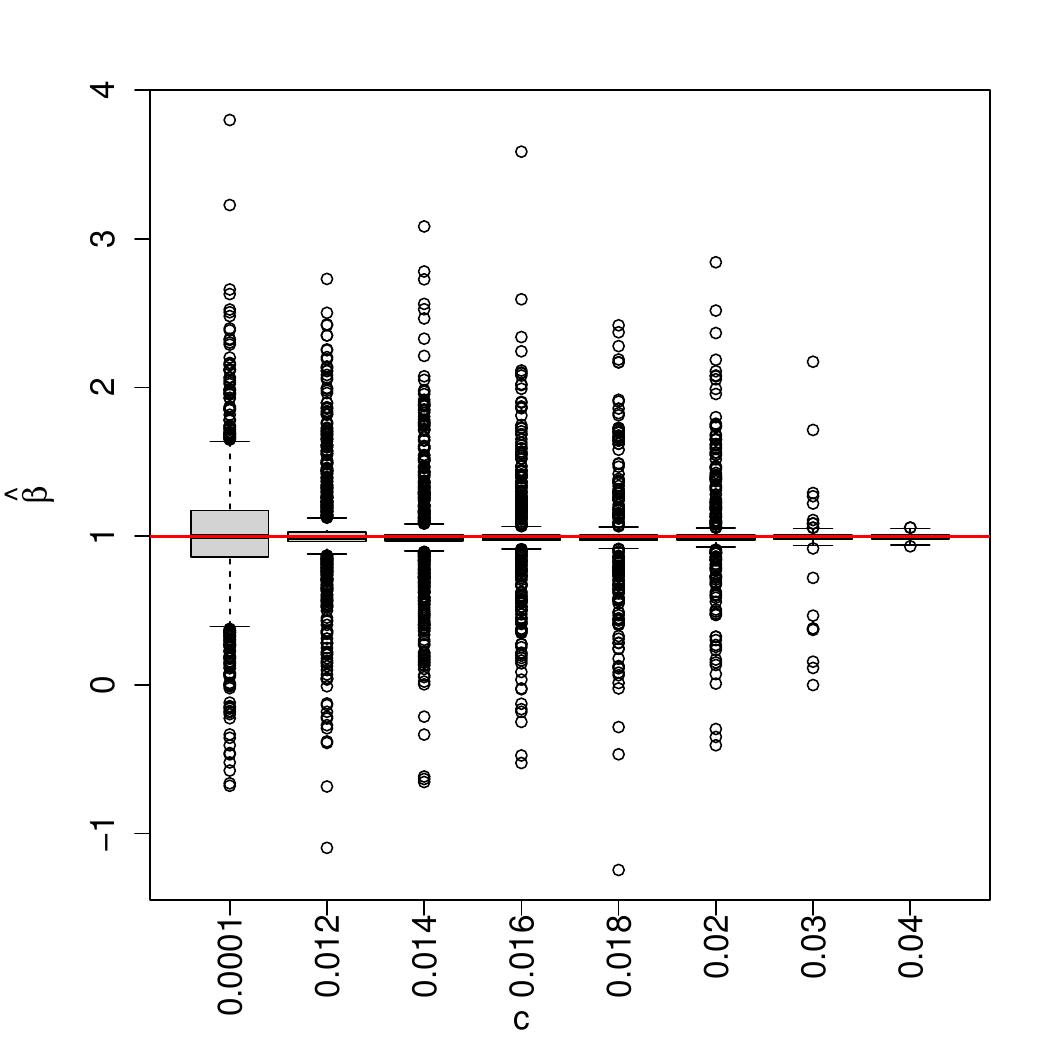}
\\

\multicolumn{3}{c}{$r_{\umbral_0,\delta}$, $\umbral_0=0.75$}\\[2ex]
\multicolumn{1}{c}{$\delta= \,-1$} & \multicolumn{1}{c}{$\delta= \,0$} & \multicolumn{1}{c}{$\delta= \,1$}\\[-2ex]
\includegraphics[scale=0.33]{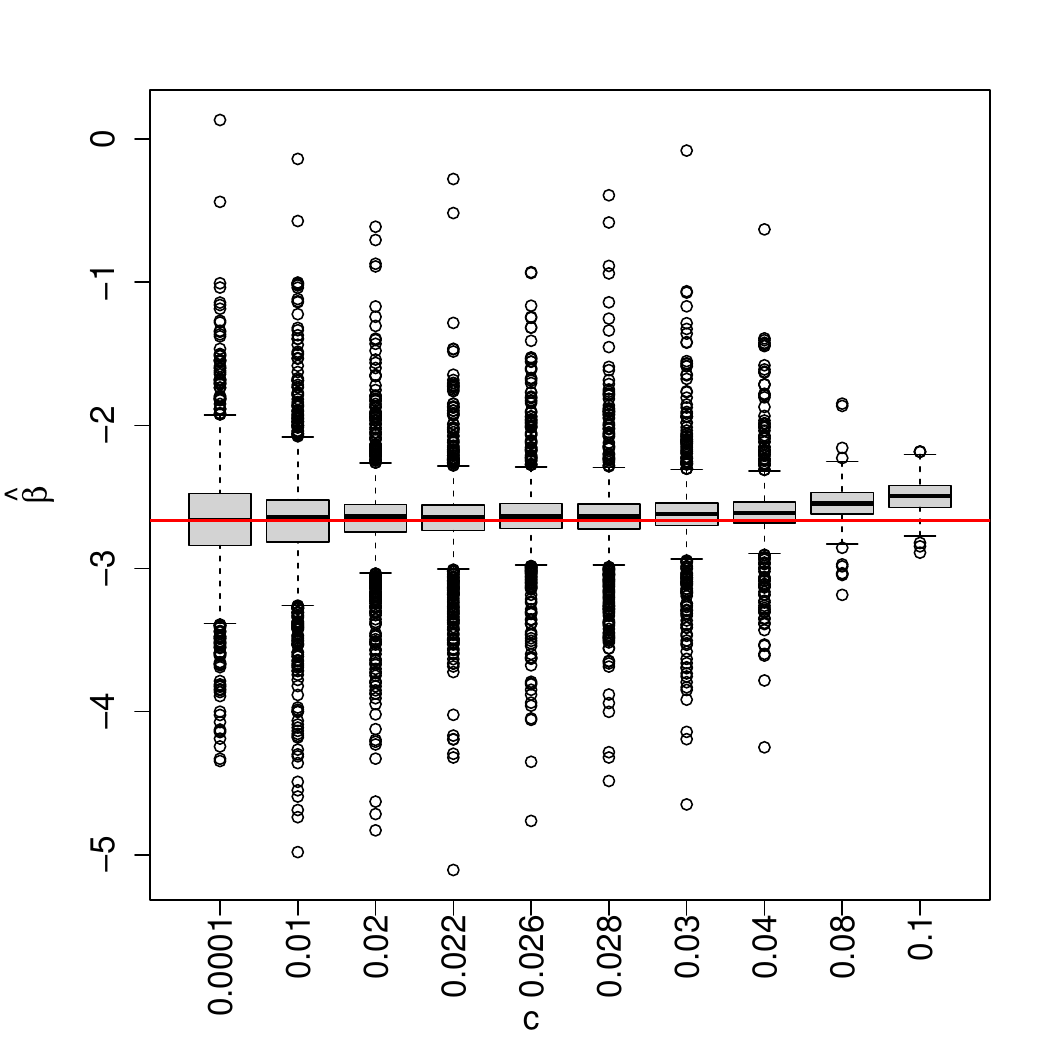} &
\includegraphics[scale=0.33]{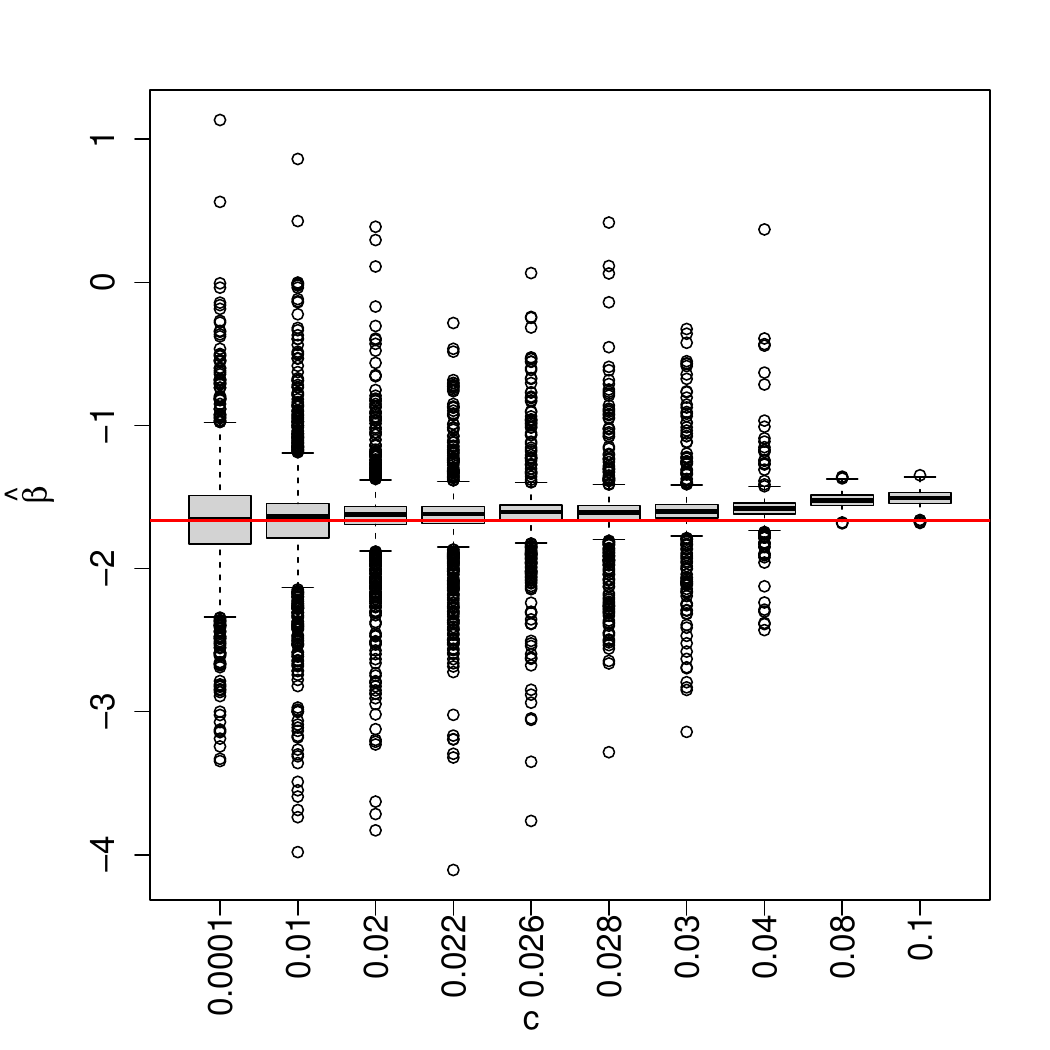}
&
\includegraphics[scale=0.33]{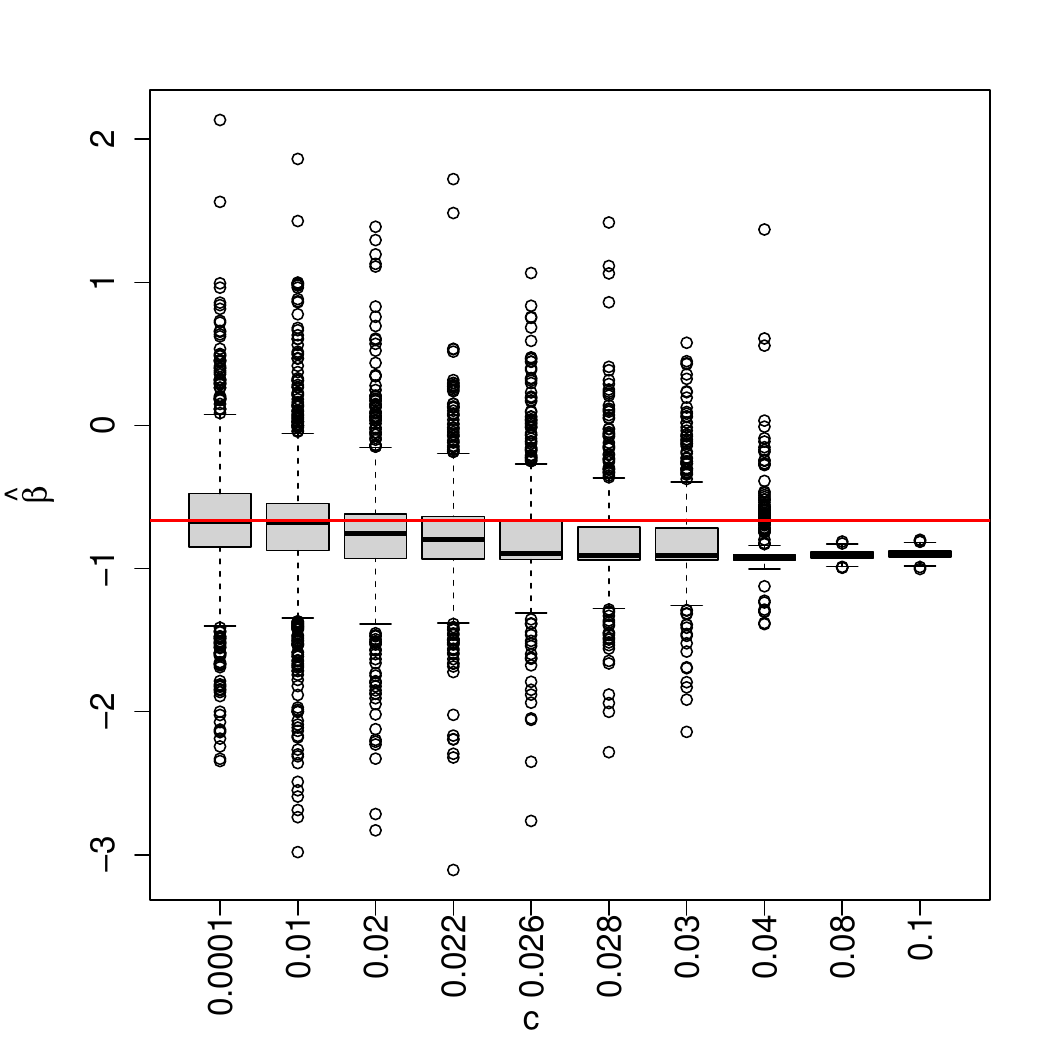}
\end{tabular}
\vskip-0.1in
\caption{\small \label{fig:bxp-beta-5}	 Boxplots of the estimators $\widehat{\beta}$ for different choices of the  penalizing constant $c$,   when $n=500$  and $\sigma=0.05$. The horizontal red line corresponds to the  value $\beta_0$.} 
\end{center}
\end{figure}

\begin{figure}[ht!]
\begin{center}
\renewcommand{\arraystretch}{0.1}
\newcolumntype{G}{>{\centering\arraybackslash}m{\dimexpr.33\linewidth-1\tabcolsep}}
\begin{tabular}{GGG}
\multicolumn{3}{c}{$r_{\umbral_0,\delta}$, $\umbral_0=0.5$}\\[2ex]
\multicolumn{1}{c}{$\delta= \,-1$} & \multicolumn{1}{c}{$\delta= \,0$} & \multicolumn{1}{c}{$\delta= \,1$}\\[-2ex]
\includegraphics[scale=0.33]{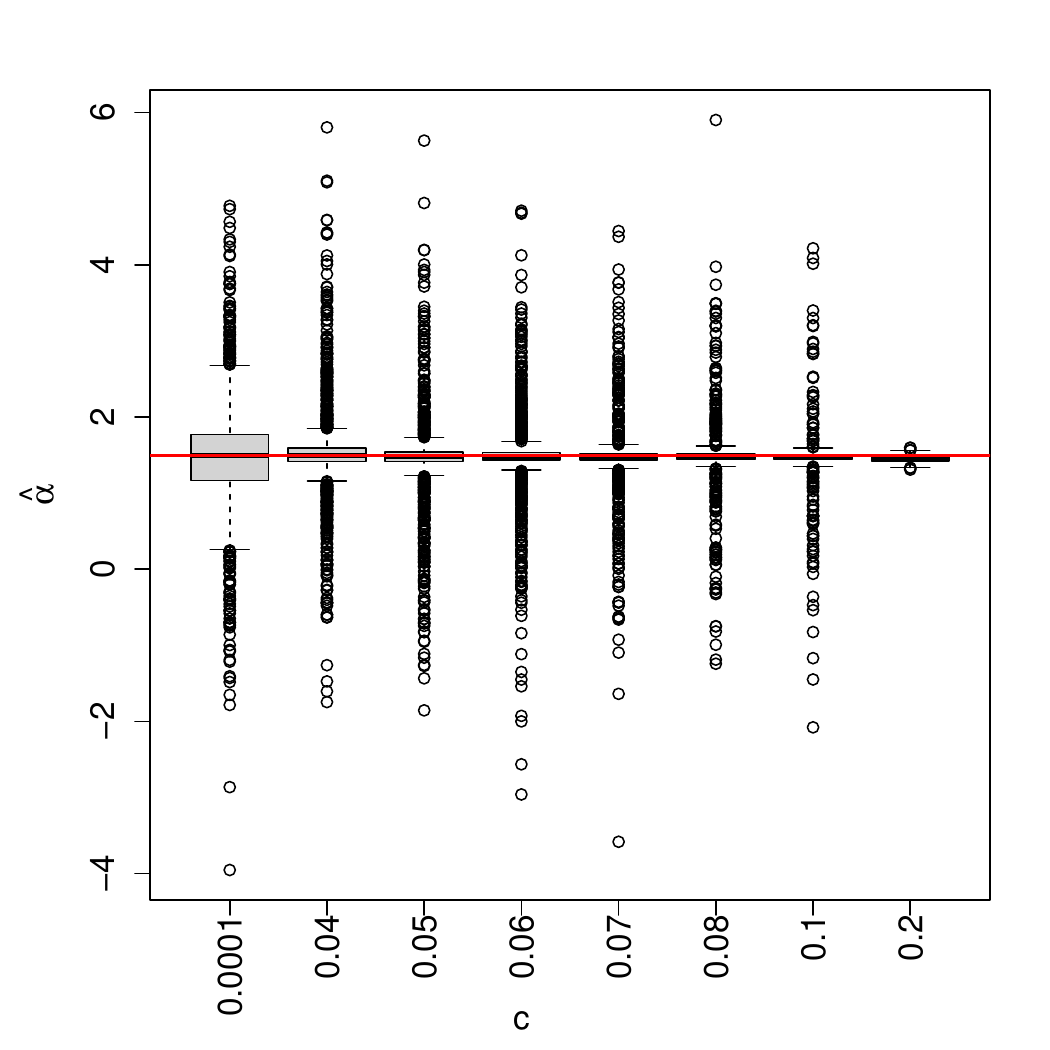} &
\includegraphics[scale=0.33]{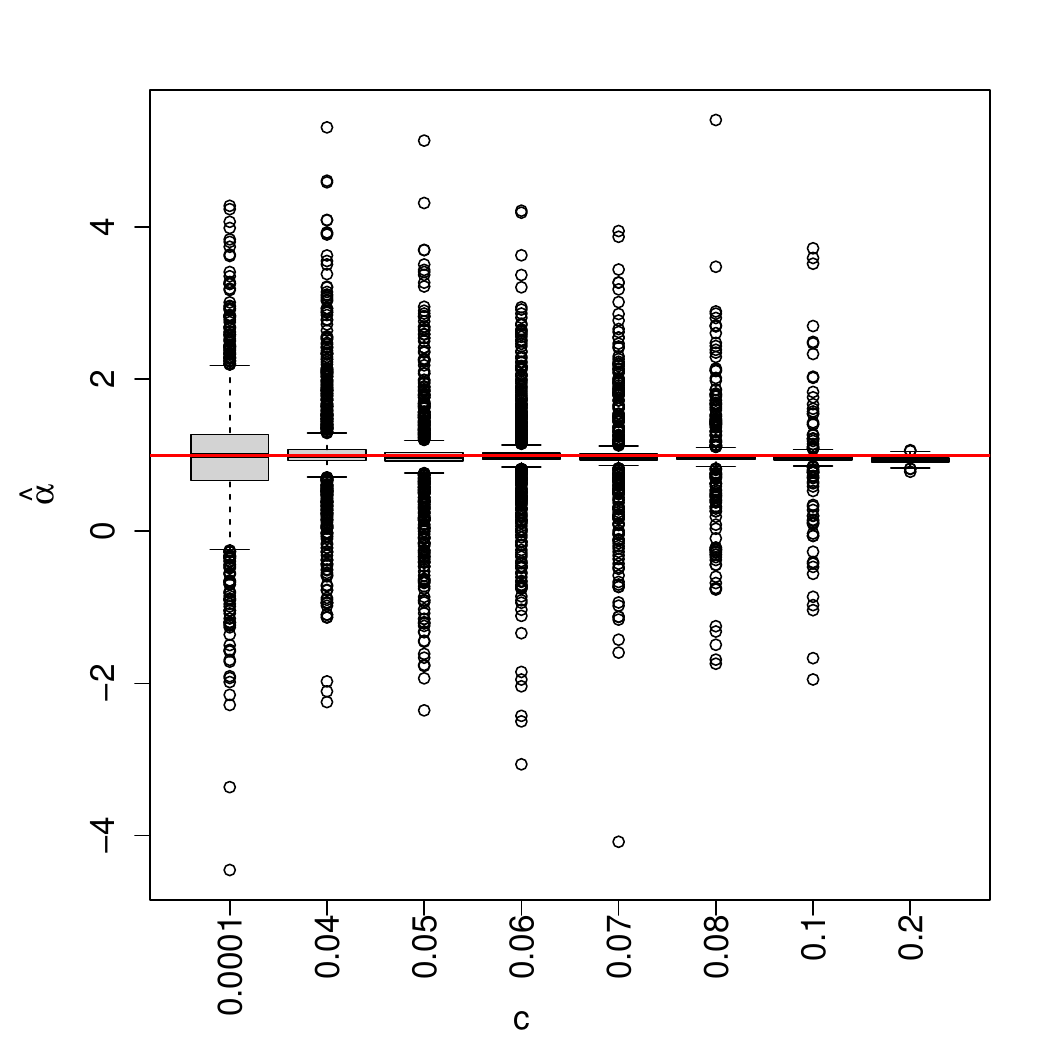}
&
\includegraphics[scale=0.33]{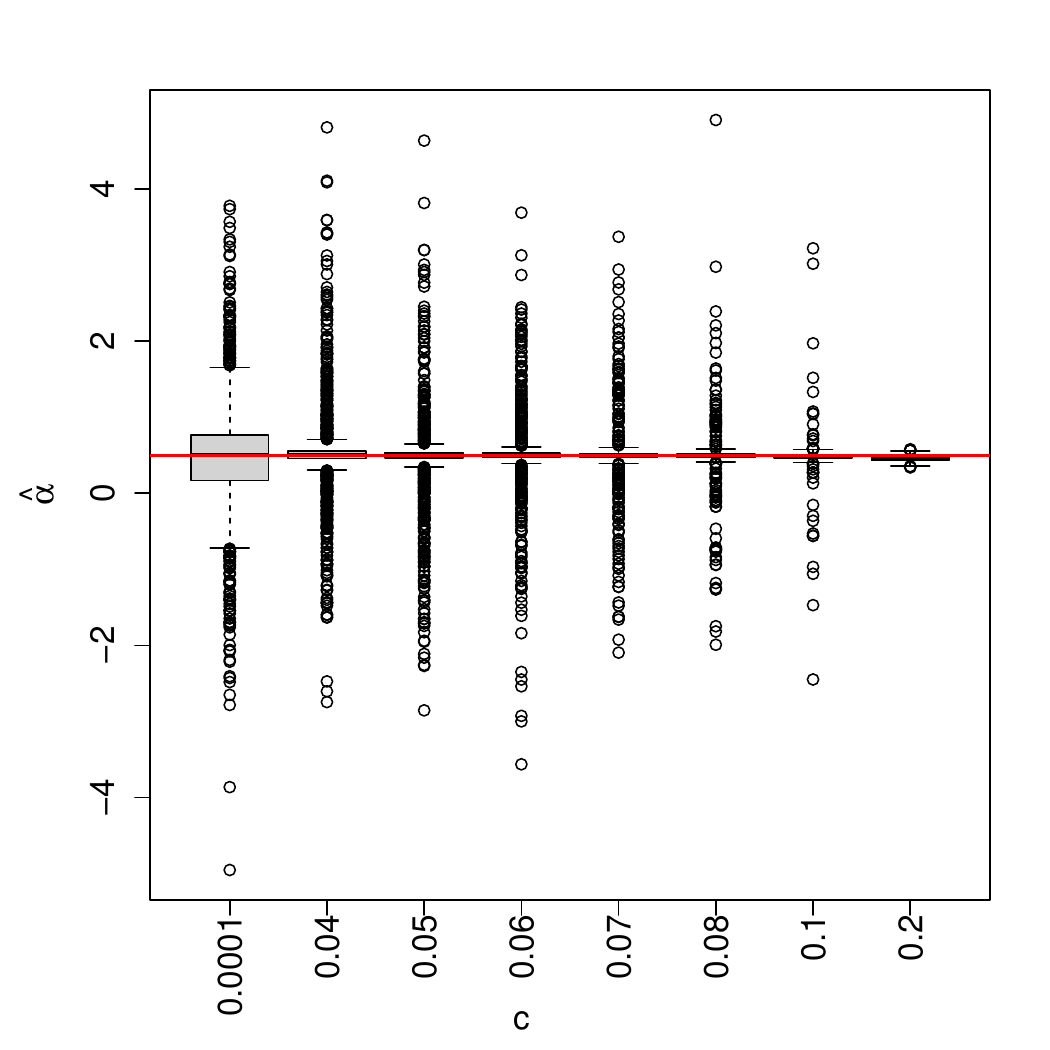}
\\

\multicolumn{3}{c}{$r_{\umbral_0,\delta}$, $\umbral_0=0.75$}\\[2ex]
\multicolumn{1}{c}{$\delta= \,-1$} & \multicolumn{1}{c}{$\delta= \,0$} & \multicolumn{1}{c}{$\delta= \,1$}\\[-2ex]
\includegraphics[scale=0.33]{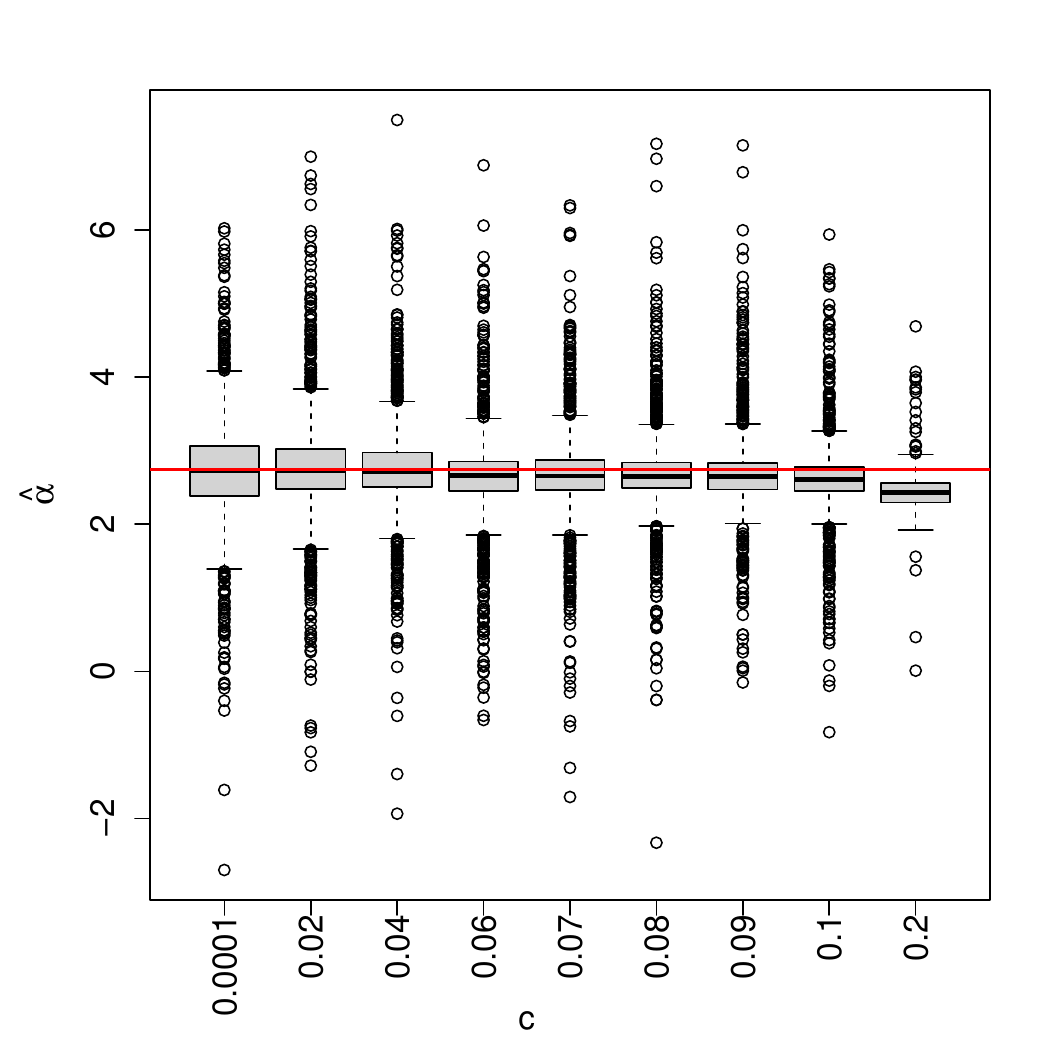} &
\includegraphics[scale=0.33]{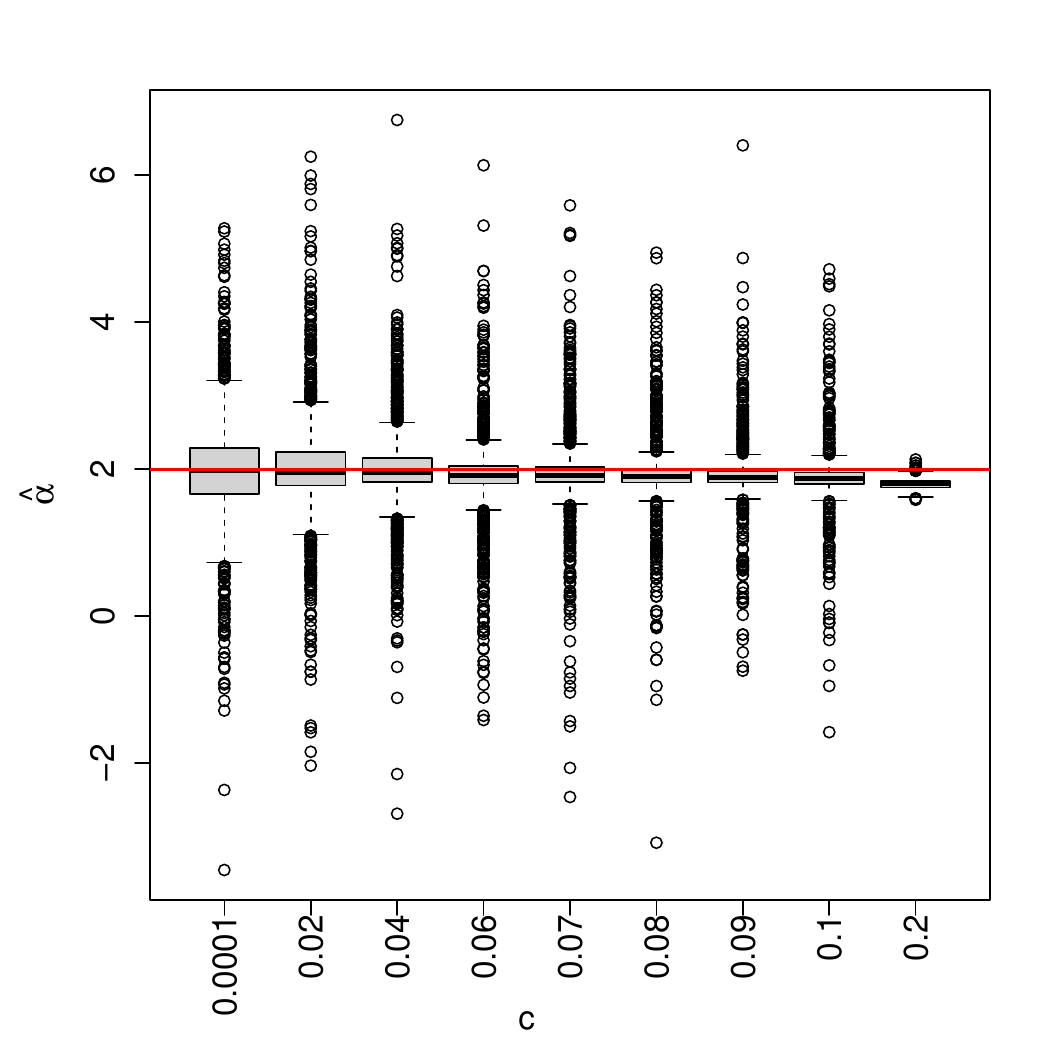} &
\includegraphics[scale=0.33]{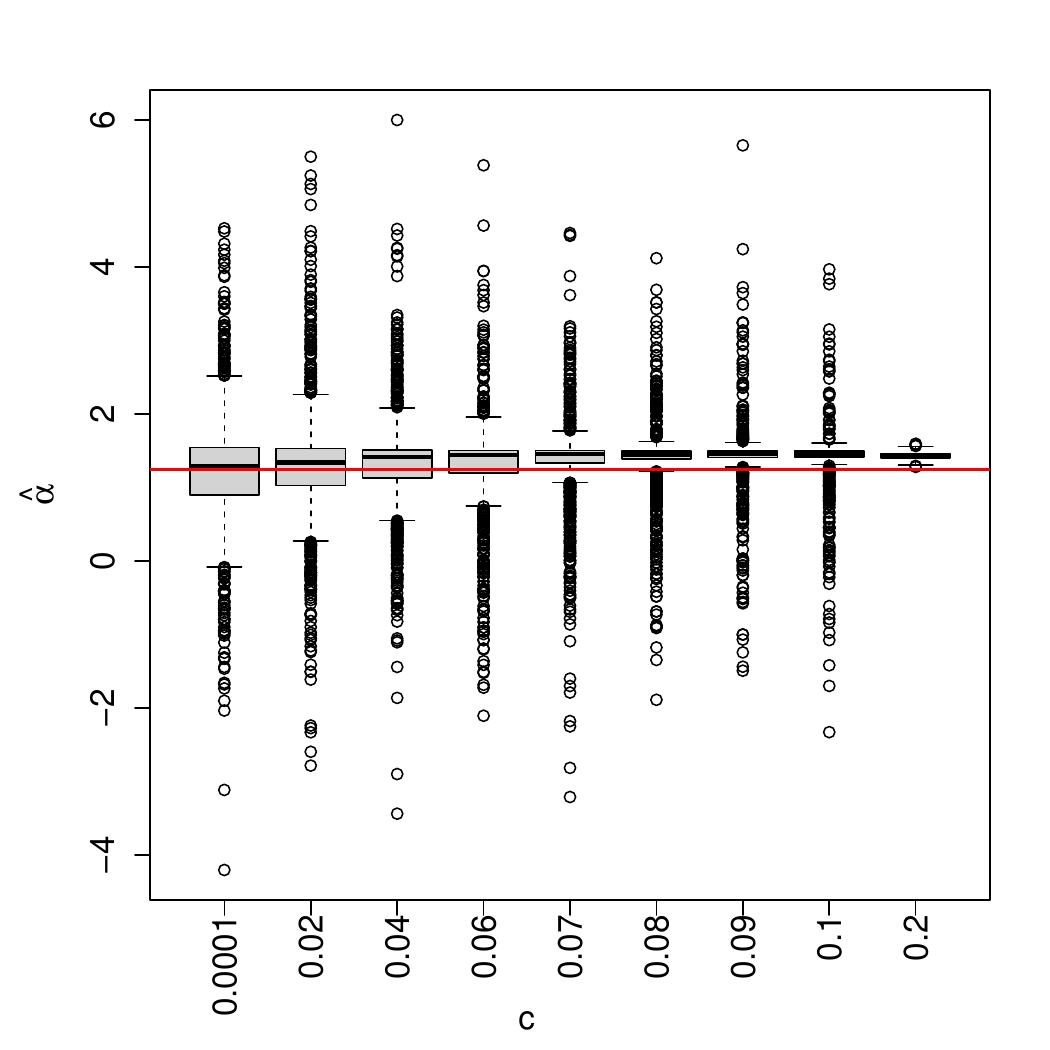}
\end{tabular}
\vskip-0.1in
\caption{\small \label{fig:bxp-alfa-10}	 Boxplots of the estimators $\widehat{\alpha}$ for different choices of the  penalizing constant $c$,   when $n=500$  and $\sigma=0.10$. The horizontal red line corresponds to the  value $\alpha_0$.} 
\end{center}
\end{figure}

\begin{figure}[ht!]
\begin{center}
\renewcommand{\arraystretch}{0.1}
\newcolumntype{G}{>{\centering\arraybackslash}m{\dimexpr.33\linewidth-1\tabcolsep}}
\begin{tabular}{GGG}
\multicolumn{3}{c}{$r_{\umbral_0,\delta}$, $\umbral_0=0.5$}\\[2ex]
\multicolumn{1}{c}{$\delta= \,-1$} & \multicolumn{1}{c}{$\delta= \,0$} & \multicolumn{1}{c}{$\delta= \,1$}\\[-2ex]
\includegraphics[scale=0.33]{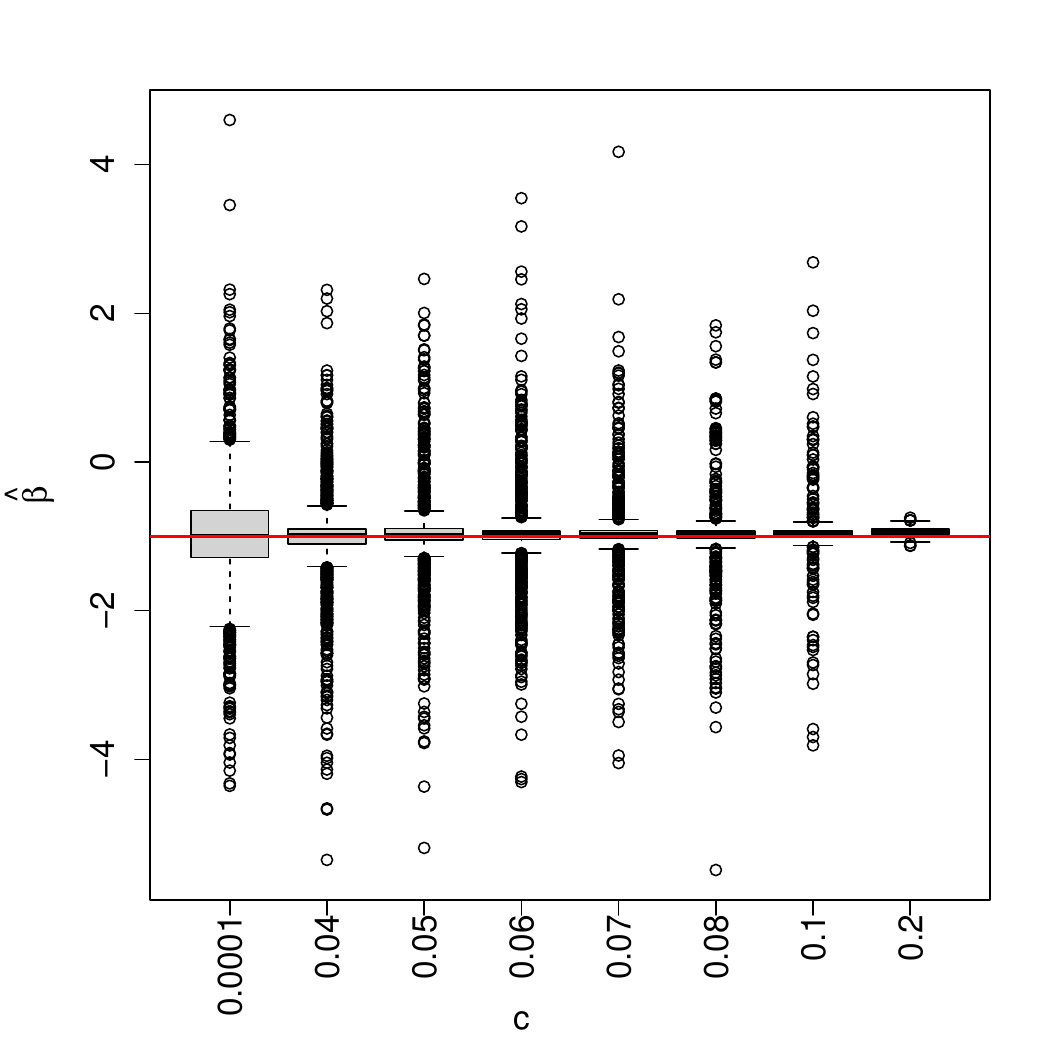} &
\includegraphics[scale=0.33]{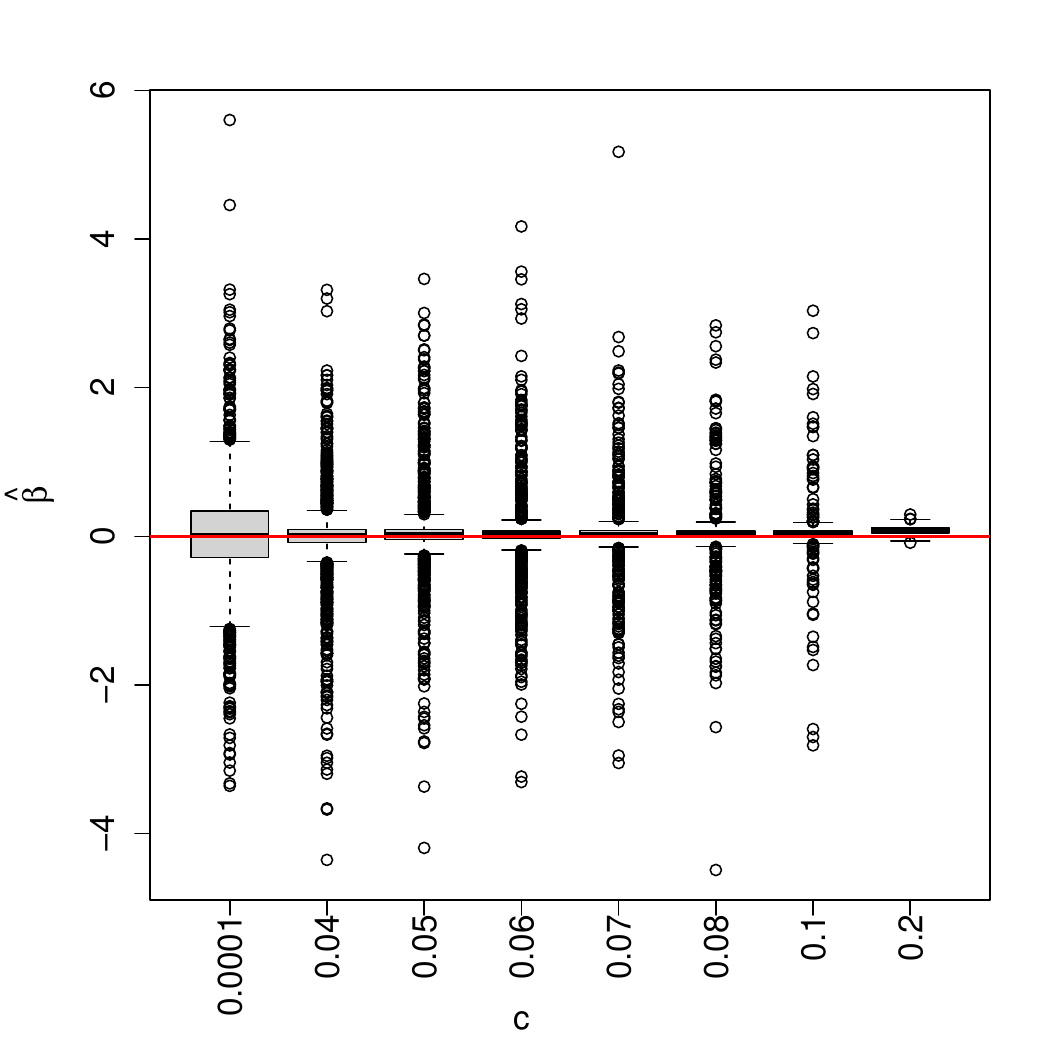}
&
\includegraphics[scale=0.33]{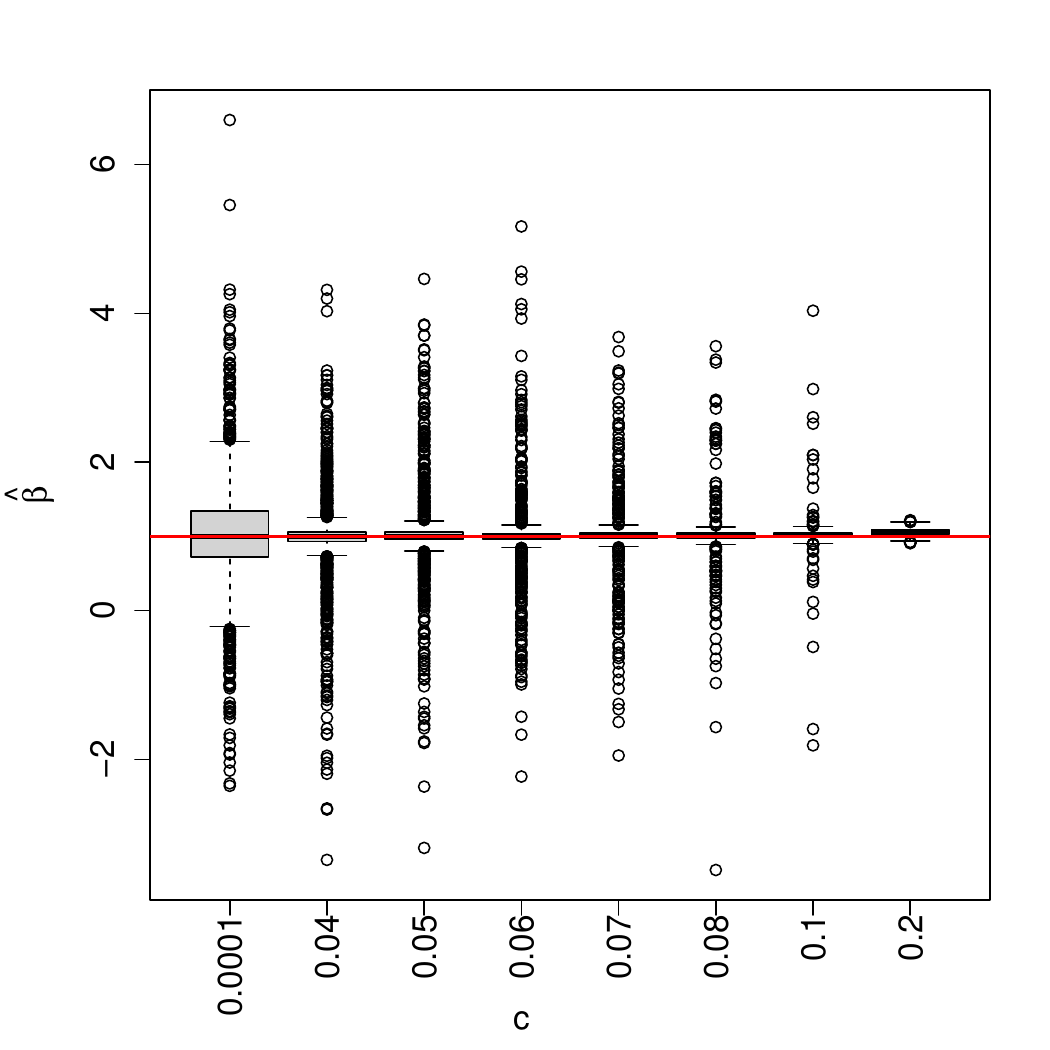}
\\

\multicolumn{3}{c}{$r_{\umbral_0,\delta}$, $\umbral_0=0.75$}\\[2ex]
\multicolumn{1}{c}{$\delta= \,-1$} & \multicolumn{1}{c}{$\delta= \,0$} & \multicolumn{1}{c}{$\delta= \,1$}\\[-2ex]
\includegraphics[scale=0.33]{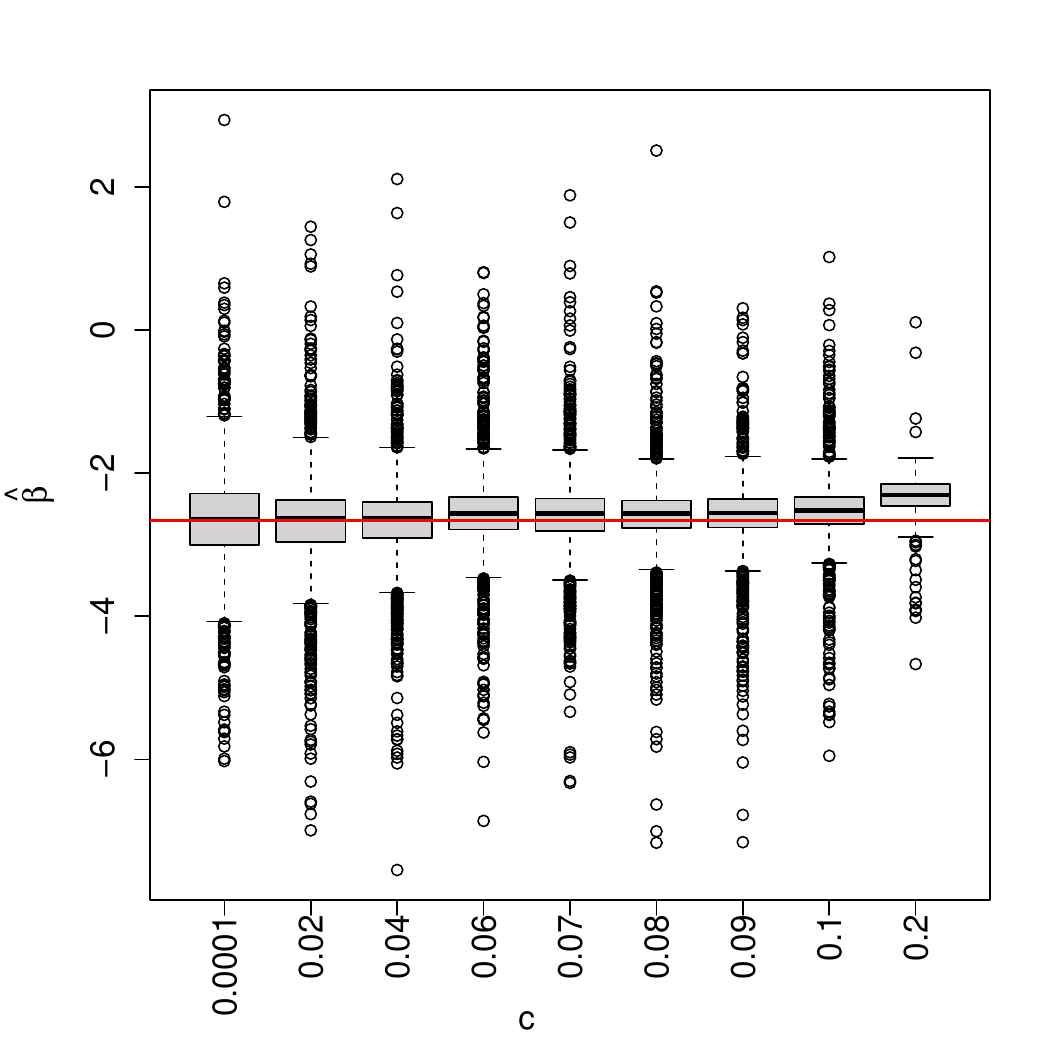} &
\includegraphics[scale=0.33]{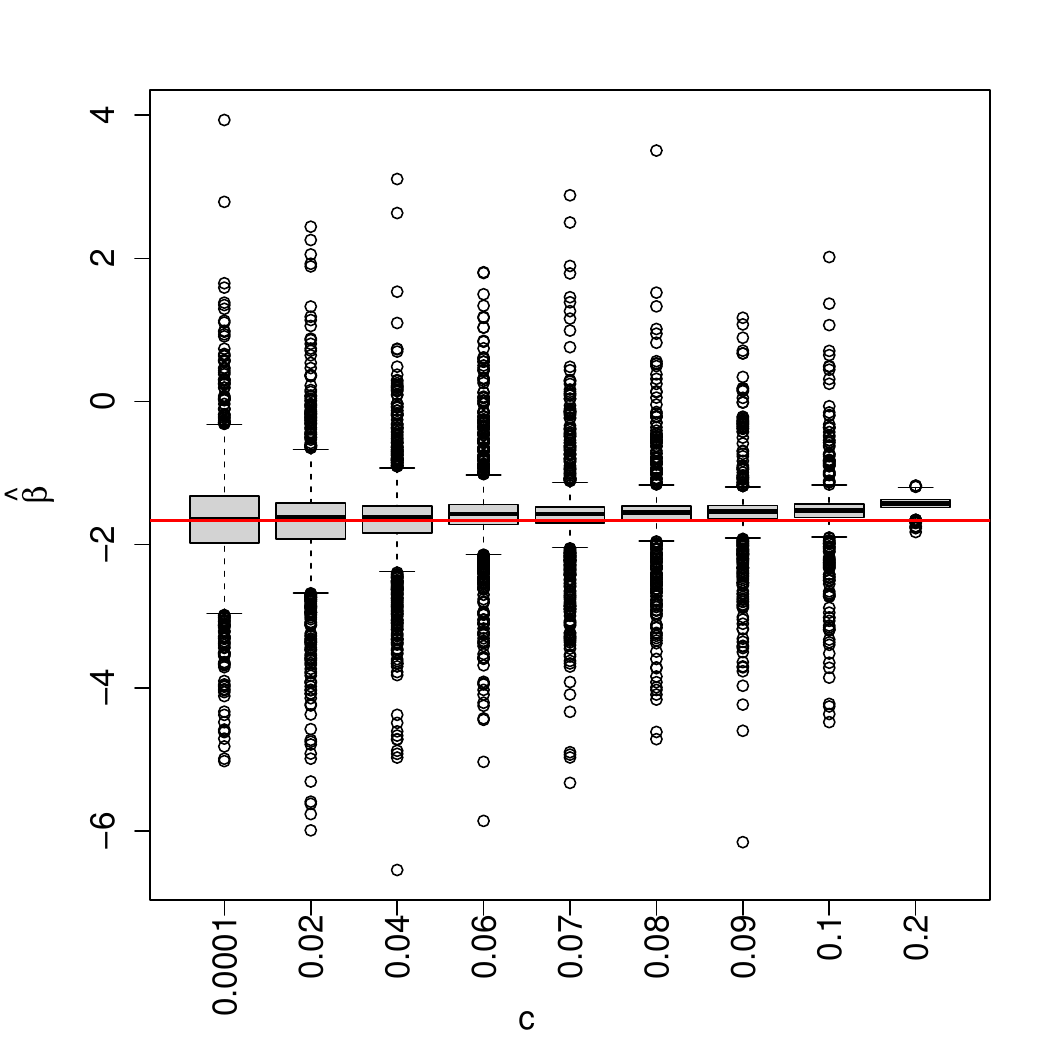}
&
\includegraphics[scale=0.33]{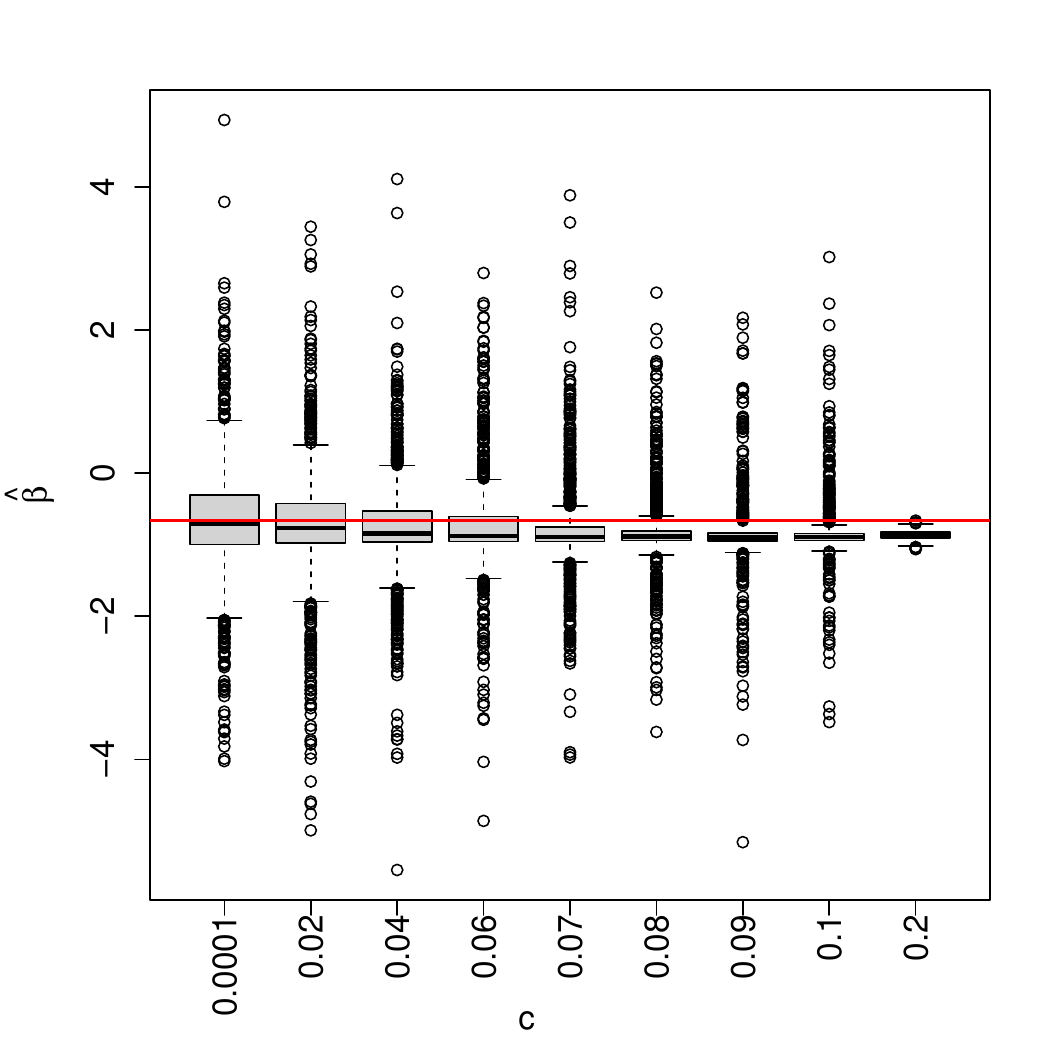}
\end{tabular}
\vskip-0.1in
\caption{\small \label{fig:bxp-beta-10}	 Boxplots of the estimators $\widehat{\beta}$ for different choices of the  penalizing constant $c$,   when $n=500$  and $\sigma=0.10$. The horizontal red line corresponds to the  value $\beta_0$.} 
\end{center}
\end{figure}

\clearpage
\section{A real data example}\label{sec:realdata}

In this section, we analyse the \texttt{airquality} real data set studied in  \citet{cleveland1985}. This data  set was used therein  as a wire conductor to present different graphical tools.  
As mentioned in the Introduction, the aim of the analysis is to study the relation between the ozone concentration and the wind speed,  that is, the response variable $Y$ is the ozone concentration and the predictor $X$ corresponds to the  wind speed. As mentioned in the Introduction, we consider the 111 cases studied in \citet{cleveland1985} which correspond to the data that do not contain missing observations.

\begin{figure}[ht!]
\renewcommand{\arraystretch}{0.1}
\newcolumntype{G}{>{\centering\arraybackslash}m{\dimexpr.33\linewidth-1\tabcolsep}}
\hskip-0.3in\begin{tabular}{GGG}
a) & b) & c)\\
\includegraphics[width=1.6in]{ozone-loess.pdf} &
\includegraphics[width=1.6in]{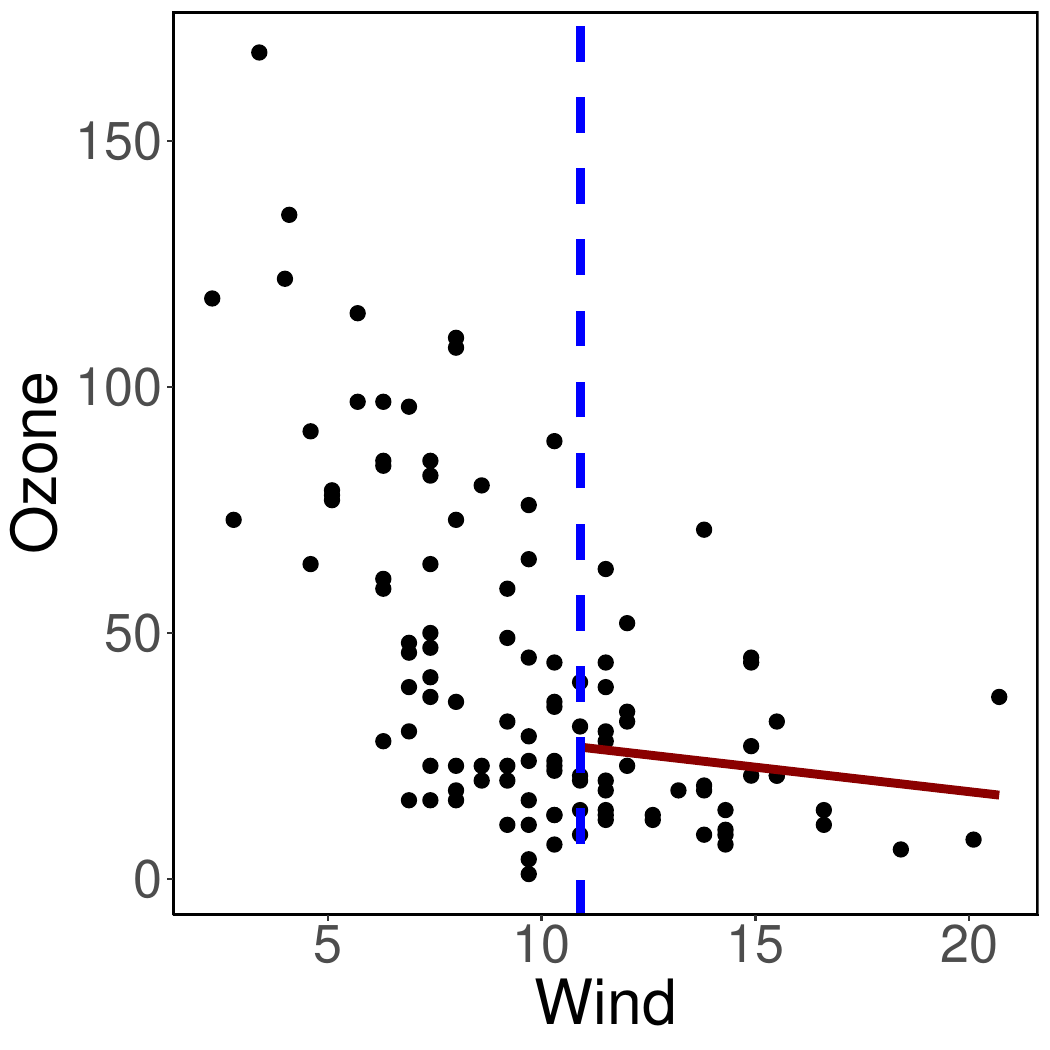} &
\includegraphics[width=1.6in]{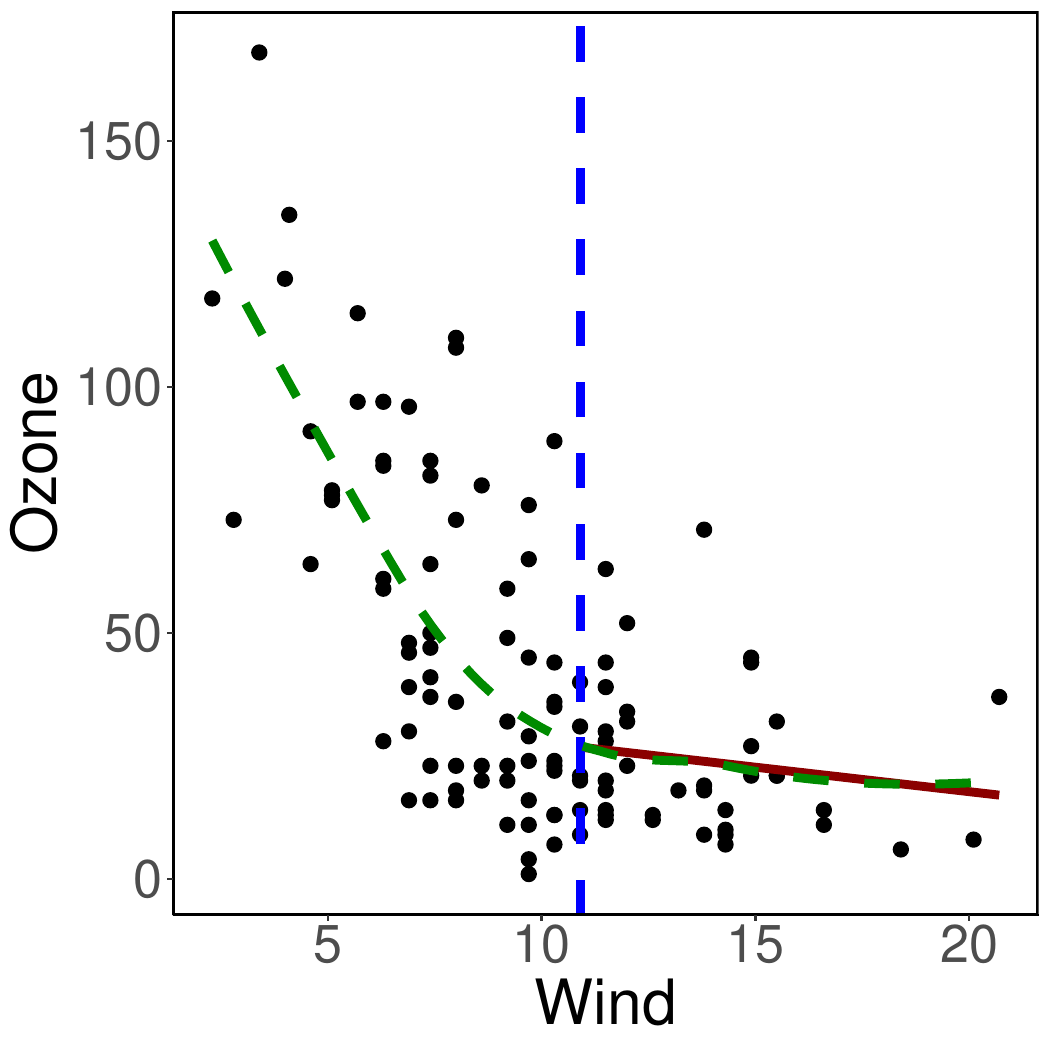}
\end{tabular}
\vskip-0.1in
\caption{\small\label{fig:los_dos}Air quality data (black points) with two different fits: a) \texttt{lowess} fit    and b)
threshold regression model with $\widehat{\umbral}=10.9$,  $\widehat{\alpha}_{\widehat{\umbral}}=37.658$ and $\widehat{\beta}_{\widehat{\umbral}}=\,-\,0.996$. Panel c) displays both fits over imposed.}
\end{figure}

The left panel  in Figure \ref{fig:los_dos} displays the data (in black points) together with a solid   green line corresponding to the fit obtained using local polynomials through the function \texttt{lowess} in \texttt{R}, see \citet{cleveland:1979,cleveland:1981}. This plot  mimics Figure 1.4 of the above mentioned book and as mentioned therein the nonparametric smoother suggests that the linear fit is not appropriate all over the support of the covariates. In this sense, an important point is to find the wind speed from which the model may be assumed to be linear and our approach though threshold regression models becomes useful.  

The middle panel in Figure \ref{fig:los_dos}  presents the data and the result obtained for the threshold regression model represented with a solid red line.  More precisely, we use  the threshold estimator introduced in this paper, to get the best wind speed above which the linear model can be used. For that purpose, the possible values of   $u$ vary  along the observed values of the predictor, up to its $98\%$ empirical quantile that corresponds to $u=18.4$.

\begin{figure}[ht!]
\centering
\includegraphics[scale=0.25]{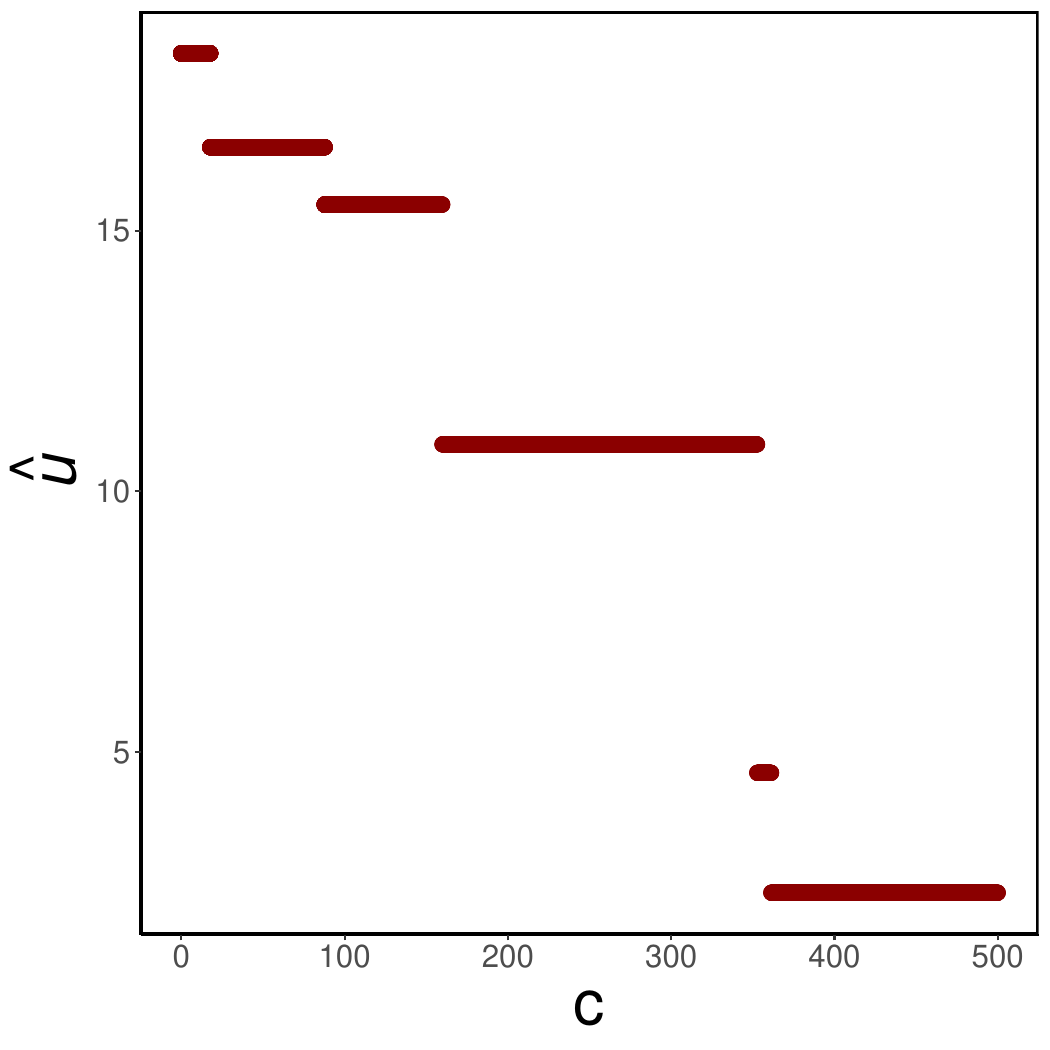}
\caption{\small \label{fig:estimaciones}Estimated thresholds (18.4, 16.6, 15.5, 10.9,  4.6 and 2.3) as a function of $c$, using the penalizing term $\lambda_n=cn^{-0.4}$ and $f(u)=u$. The cut-off    $\gamma_n$ is the  $98\%-$empirical quantile of the predictors.}
\end{figure}

Figure \ref{fig:estimaciones} displays the obtained values of the estimated threshold according to different values of the constant $c$ included the penalizing term defined as $\lambda_n f(u) = c\; n^{-0.4} f(u)$, where as in the simulation study $f(u)=u$ for $u\ge 0$ and $f(u)=0$ for $u\le 0$. The values of $c$ were chosen in the interval $[0,500]$, using an grid with  different grid steps, to be more precise, in the interval $[0,10]$ the increment between consecutive points was $0.001$, between 10.01 and  150 it was 0.01, while in the interval $[150.1, 500]$ the spacing was $0.1$. Figure \ref{fig:estimaciones} reveals that the threshold is a step function of $c$ leading to six possible thresholds which were obtained as $c$ varies along the considered interval.
As expected, for $c$ too small, we can not prevent an  \textsl{overfit phenomenon} which leads to   estimated thresholds equal to $\widehat{\umbral}=18.4$, $16.6$ and $15.5$ meaning that in almost all the support the model is considered nonparametric. In contrast, for very large values of $c$  the penalization term dominates the objective function, estimating the threshold as the left extreme of the interval ($\widehat{\umbral}=2.3$), suggesting that the linear regression could be used in the entire domain of the predictor. Interestingly, $\widehat{\umbral}=10.9$ is observed for an intermediate regime, indicating that 10.9 may be a possible value above which the linear model may be used. The fit displayed in the middle panel of Figure \ref{fig:los_dos} corresponds to this last value. Panel c) in Figure \ref{fig:los_dos}  displays the local smoother in dashed green lines together with the linear fit above the estimated  threshold plotted as a solid red line. The plot suggests that the considered linear fit is appropriate for values of the wind speed larger or equal than 10.9.

Taking into account the asymptotic behaviour provided in Theorem \ref{teo:distriasint} for the intercept and slope estimators, when considering values larger than the estimated threshold,  we  chose $\psi=1$ and we performed a linear fit  for those observations above    $\widehat{\umbral}+\psi=11.9$ obtaining $\widehat{\alpha}_{\widehat{\umbral}+\psi}= 42.096$ and $\widehat{\beta}_{\widehat{\umbral}+\psi}=\,-\, 1.280$. For this cut--off value,  the  slope estimator  $\widehat{\beta}_{\widehat{\umbral}+\psi}$  is not significantly different from 0 suggesting that for large values of the wind speed the ozone remains constant.

\section{Final comments}\label{sec:coments}
In this paper, we introduce a new model, called \textsl{threshold regression semiparametric model},  that attempts to overcome the limitations presented by the linear model while maintaining its simplicity and interpretability. The considered model tries to keep the validity of the linear model in the biggest possible domain for large values of the covariate, while no shape assumptions on the regression function are assumed  for values smaller than the threshold. This is an advantage over segmented linear regression or the so--called bent linear regression with an unknown change point, since the threshold regression semiparametric model does not require that the regression function is also linear before the change point.

A procedure to consistently estimate the threshold as well as the intercept and slope of the linear component of the regression function is introduced. The numerical results reflects that, as expected, the non--penalized threshold estimates do not converge to the true threshold and confirms the consistency results derived in Theorem \ref{teo:tiro_del_final} when a penalty term converging at a rate smaller than root-$n$ is included. Furthermore, the obtained results reflect that  the rate of convergence may depend on the regression smoothness. The benefits of our model are illustrated over the air quality data set  providing simple but flexible model. The threshold model allows to understand the relation between the ozone concentration and the wind speed for high values of the speed through a simple linear regression model. 

An important topic to be mentioned is that our approach allows for more general  semiparametric models, where non-parametric and parametric regimes are allowed under a unified framework. For instance, the model and the estimation method introduced may be easily extended to other parametric models beyond the simple linear regression postulated for the regression function beyond the threshold. Extensions to a threshold generalized linear model are also an important issue which is beyond the scope of the paper and will be object of future work. 

As is well known,  estimators based on least squares are usually affected by atypical data, that is, by observations with large residuals in particular if combined with  covariates with high leverage. In this sense, robust procedures may be important to mitigate the impact of outliers in the estimation process, both for linear or more general parametric regression models beyond a threshold. The result provided in Proposition \ref{prop:consistencia_ver2} suggests that other loss functions instead of the squared loss may be used to define resistant and consistent estimators. This interesting topic is now part of ongoing research. 

{\small 
\section*{Acknowledgments}
This paper was the result of the activities of the \emph{Research, 
Innovation and Dissemination Center for Neuromathematics}, Grant FAPESP 2013/07699-0, 
FAPESP's project \emph{Stochastic Modeling of Interacting Systems},  
Grant 2017/10555-0, S\~{a}o Paulo Research Foundation, Brazil.  This research was partially supported by    CNPq's research fellowship, Grant 311763/2020-0 (Florencia Leonardi),  grants 20020170100330BA  (Daniela Rodriguez and Mariela Sued) and 20020220200037BA (Graciela Boente) from Universidad de Buenos Aires,  PICT-2020-01302 (Daniela Rodriguez and Mariela Sued)  and   PICT-2021-I-A-00260 (Graciela Boente)  from ANPYCT,  and the Spanish Project {MTM2016-76969P} from the Ministry of Economy, Industry and Competitiveness, Spain (MINECO/AEI/FEDER, UE)  (Graciela Boente).  We would like to thank Professor Ricardo Fraiman for the helpful discussions and for  useful suggestions.}

 \setcounter{section}{0}
\renewcommand{\thesection}{\Alph{section}}

\setcounter{equation}{0}
\renewcommand{\theequation}{A.\arabic{equation}}

\section{Appendix: Proofs}
\subsection{Proof of Lemma \ref{lema-basico}}
In the sequel,  to alleviate the notation, we will use the sub index $\umbral$ to operate conditioned on $\{X\geq \umbral\}$. For instance,
$\esp_{\umbral}(W)$,  $\var_{\umbral} (X)$ and $\cov_{\umbral}(X,Y)$  stand for $\esp(W\mid X\geq \umbral)$,  $\esp(X^2\mid X\geq \umbral)-\esp^2(X\mid X\geq \umbral)$ and $\esp(XY \mid X\geq \umbral)-\esp(X \mid X\geq \umbral)\esp(Y\mid X\geq \umbral)$, respectively.  As mentioned above, if $\umbral=-\infty$ no restrictions are made  when conditioning.  

\vskip0.1in
\begin{proof}[Proof of Lemma \ref{lema-basico}]
Recall that, for $\umbral<b$,  $(\alpha_{\umbral}, \beta_{\umbral})$ are the best coefficient for linearly approximate $Y$ based on $X$ when $X\geq \umbral$, that is, they are defined through \eqref{eq:abu}. Therefore, we have that
$$\beta_{\umbral} =\frac{\cov_{\umbral}(X,Y)}{\var_{\umbral}(X)}\qquad \mbox{and}\qquad \alpha_{\umbral}=\esp_{\umbral}(Y)-\beta_{\umbral} \esp_{\umbral}(X)\,. 
$$
{Note that, as mentioned in Section \ref{sec:modelo}, when $\umbral_0\ne -\infty$, for any  $\umbral\geq \umbral_0$, we have that  $(\alpha_{\umbral}, \beta_{\umbral})=(\alpha_0,\beta_0)$, while in the case,   when $\umbral_0=-\infty$,   $(\alpha_0,\beta_0)=(\alpha_{\umbral}, \beta_{\umbral})$ for any   $\umbral <b$.}

a) Using that from \ref{ass:dist-continua}, the marginal distribution function of $X$ is continuous, from the Lebesgue Dominated  Convergence Theorem we obtain that  $\alpha_{\umbral}$ and $\beta_{\umbral}$ depend continuously on $\umbral$ and that   $\ell(\cdot)$ is also a continuous function on $(-\infty,b)$.

b) Using that $Y=r(X)+\varepsilon$, the independence between the errors and the covariates and that $\esp (\varepsilon) =0$ from \ref{ass:momento}, we obtain that   $\ell(\umbral)$ can be written as
\begin{equation}
\label{eq:elledeu}
\ell(\umbral)=\esp_{\umbral} \left[\left\{Y-(\alpha_{\umbral} +\beta_{\umbral} X)\right\}^2\right]=\esp_{\umbral}\left [\left\{r(X)-(\alpha_{\umbral} +\beta_{\umbral} X)\right\}^2\right]+\esp\left(\varepsilon^2\right)\,.
\end{equation}
{On the one hand, when  $\umbral_0\ne -\infty$, the desired result follows   from the facts that $\prob(r(X)=\alpha_0+\beta_0 \,X\mid X\ge\umbral)=1$ and  $(\alpha_0,\beta_0)=(\alpha_{\umbral}, \beta_{\umbral})$, for any  $\umbral\geq \umbral_0$. On the other hand,  when $\umbral_0=-\infty$,  we have that  $\prob(r(X)=\alpha_0+\beta_0 \,X )=1$ and $(\alpha_0,\beta_0)=(\alpha_{\umbral}, \beta_{\umbral})$, for any $\umbral <b$, which concludes the proof of b).}

c)  Let us consider the situation where $a<\umbral<\umbral_0$  and assume that $\ell(\umbral)=\ell(\umbral_0)$. Using \eqref{eq:elledeu}, we conclude that  $\esp_{\umbral} [\{r(X)-(\alpha_{\umbral} +\beta_{\umbral} X)\}^2]=0$ which leads to    $\prob\left(r(X)=\alpha_{\umbral} +\beta_{\umbral} X\mid X\geq \umbral\right)=1$  contradicting the definition of $\umbral_0$.

On the other hand, if  $a\ne -\infty$ and $\umbral\le a$, we get that we get that  $\alpha_u=\alpha_a$ and $\beta_u=\beta_a$, so
\begin{equation}
	\ell(\umbral)=\esp\left [\left\{r(X)-(\alpha_{a} +\beta_{a} X)\right\}^2\right]+\esp\left(\varepsilon^2\right)=\ell(a)\,,
\end{equation}
which is bigger than $\ell(\umbral_0)=\esp(\varepsilon^2)$, by \ref{ass:u0mayorquea}. Hence, $\ell(\umbral)>\ell(\umbral_0)$ for any $u<\umbral_0$, concluding the proof. 

d) We need to show that 
$$A_{\delta}=\inf_{\{u\le \umbral_0-\delta\}} \ell(u) > \ell(\umbral_0)\;.$$

To do so, let $\{\umbral_k\}_{k\in \natu}$  be a sequence  such that $\umbral_k\le \umbral_0-\delta$ and   $\lim_{k\to\infty} \ell(\umbral_k)=A_{\delta}$. If the sequence  $\{\umbral_k\}_{k\in \natu}$ is bounded, there exists a convergent subsequence  $\{\umbral_{k_j}\}_{j\in \natu}$. Let $\umbral^{\star}$ be such that $\lim_{j\to\infty} \umbral_{k_j}=\umbral^{\star}$, then $\umbral^{\star}\le \umbral_0-\delta$ and the continuity of the function $\ell$ implies that 
 $A_{\delta}=\lim_{j\to\infty} \ell(\umbral_{k_j})=\ell(\umbral^{\star})$, so using c) we obtain that $\ell(\umbral^{\star})>   \ell(\umbral_0)$, concluding the proof.

If   $a\ne \,-\,\infty$, the approximating sequence  $\{\umbral_k\}_{k\in \natu}$ may always be assumed to be bounded below since $\ell$ is constant for $\umbral\leq a$. In that case the proof is completed. 

 If $a= \,-\,\infty$ and the sequence  $\{\umbral_k\}_{k\in \natu}$ is not bounded below, there exists a subsequence  $\{\umbral_{k_j}\}_{j\in \natu}$ such that $\lim_{j\to\infty} \umbral_{k_j}=\,-\,\infty$. Denote
\begin{equation}
\label{eq:betastar}
\beta^{*} =\frac{\cov(X,Y)}{\var(X)}\qquad \mbox{and}\qquad \alpha^{*}=\esp(Y)-\beta^{*}\esp(X)\,.
\end{equation}
Arguing as in b), we obtain that 
\begin{align}
A_{\delta}&=\lim_{j\to\infty} \ell(\umbral_{k_j})= \esp \left\{Y-\left(\alpha^{*} +\beta^{*} X\right)\right\}^2
\nonumber \\
& =\esp \left\{r(X)-\left(\alpha^{*} +\beta^{*} X\right)\right\}^2 +\esp\left(\varepsilon^2\right)> \ell(\umbral_0)\,,
\label{eq:Adelta}
\end{align}
 where the last inequality follows from \ref{ass:u0mayorquea}, since $\prob\left(r(X)=\alpha^{*} +\beta^{*} X \right)<1$.
\end{proof}

\subsection{Proof of  Proposition \ref{prop:consistencia_ver2}}
In this section, we include the proof of Proposition \ref{prop:consistencia_ver2} which is a general result allowing to derive  Theorem \ref{teo:tiro_del_final}.

\begin{proof}[Proof of  Proposition \ref{prop:consistencia_ver2}]
Take $\widetilde{\gamma}>\gamma$, such that $1-F_X(\widetilde{\gamma})>0$ and define 
$$
W_n=\sup_{u\leq \umbral_0} \vert\widehat{\ell}(u, {\mathcal{Z}}_n)-\ell (u)\vert \quad \text{ and }\quad T_n=\sup_{u\in [\umbral_0,\widetilde{\gamma}]} \vert\widehat{\ell}(u, {\mathcal{Z}}_n)-\ell (u)\vert\,.
$$
Using that $\gamma_n \convpp \gamma$, we have that there exists a null probability set $\mathcal{N}$ such that $\gamma_n \to \gamma$, for any $\omega\notin \mathcal{N}$. Denote $\delta_0= (\gamma-\umbral_0)/2$, $\delta_0^{\star}= (\widetilde{\gamma}-\gamma)/2$   and $\delta_1=\min(\delta_0, \delta_0^{\star})$. Then,  we have that $\umbral_0+\delta_1 < \umbral_0+\delta_0 = \gamma- \delta_0$ and $\gamma+\delta_1<\gamma+\delta_0^{\star}=\widetilde{\gamma}$, so  
\begin{equation}
\label{eq:conv1}
\lim_{n\to \infty}	\prob(\mathcal{A}_n)=\prob( \umbral_0 +\delta_1 <\gamma_n<\widetilde{\gamma})=1\,,
\end{equation}
implying that for any  $\omega \in \mathcal{A}_n$, we have $\umbral_0 +\delta <\gamma_n<\widetilde{\gamma}$, for any $0<\delta<\delta_1$.

From \ref{ass:fcreciente}, there exists $\delta_2>0$  such that $f$ is strictly increasing in $[ \umbral_0-\delta_2, \umbral_0+\delta_2]$. We will show that for any $0<\delta<\min(\delta_1, \delta_2)$ there exists $\nu>0$ such that
\begin{equation}
\label{eq:Wnsubset}
\{W_n\leq \nu\}\;\subseteq\; \left\{\pl(\umbral _0,{\mathcal{Z}}_n)< \pl(\umbral ,{\mathcal{Z}}_n)\:,\quad \hbox{for $\umbral <\umbral_0-\delta$}\right\} 
\end{equation}
and that, for $n$ large enough, we have that 
\begin{equation}
\label{eq:Tnsubset}
\mathcal{A}_n\cap \{T_n\leq a_n\} \subseteq\; \left\{ \pl(\umbral _0,{\mathcal{Z}}_n)< \pl(\umbral ,{\mathcal{Z}}_n)\:,\quad \hbox{for $\umbral_0+\delta<\umbral <\widetilde{\gamma}$}  \right\}\;.    
\end{equation}
Note that from \eqref{eq:Wnsubset} and  \eqref{eq:Tnsubset} using that $\widehat{\umbral}_n < \gamma_n$ , we obtain that
$$
\mathcal{A}_n\,\cap\, \{W_n\leq \nu\}\,\cap\,  \{T_n\leq a_n\}\;\subseteq\;
\left\{\, \vert \widehat{\umbral}_n-\umbral_0\vert \leq \delta\,\right\}\,,
$$
which, together with \eqref{eq:conv1} and the convergences assumed in (a) and (b), entail  that \linebreak $\prob\left(\vert \widehat{\umbral}_n-\umbral_0\vert \leq \delta\right)\to 1$ as $n\to \infty$. Then,   if we prove \eqref{eq:Wnsubset} and  \eqref{eq:Tnsubset} we obtain the desired result.

\vskip0.1in
Let us begin by showing that \eqref{eq:Tnsubset} holds.   Taking into account that $\ell$ satisfies \eqref{eq:l_minimiza},  we have that $\ell(\umbral )=\ell(\umbral_0)$,  for any  $\umbral_0\le \umbral\leq \widetilde{\gamma}$, therefore  we obtain that  
\begin{align*}
	\pl(\umbral ,{\mathcal{Z}}_n)-\pl(\umbral _0,{\mathcal{Z}}_n)&\;=\;\widehat{\ell}(\umbral,  {\mathcal{Z}}_n)-\widehat{\ell}(\umbral_0,  {\mathcal{Z}}_n)
	+\lambda_n(f(\umbral)-f(\umbral_0))\\
	&\;=\;\widehat{\ell}(\umbral,  {\mathcal{Z}}_n)-\ell(\umbral)+\ell(\umbral_0)-\widehat{\ell}(\umbral_0,  {\mathcal{Z}}_n)
	+\lambda_n(f(\umbral)-f(\umbral_0))\\
	&\;\geq\; -2a_n+\lambda_n(f(\umbral)-f(\umbral_0)) \,,
\end{align*}	
where the last inequality follows from the fact that $T_n\leq a_n$.  
Note that   \ref{ass:fcreciente} implies that $f(\umbral)\ge f(\umbral_0+\delta)$, for any $\umbral>\umbral_0+\delta$. Furthermore, using that $\delta<\delta_2$, we get that  $f(\umbral_0+\delta) - f(\umbral_0)>0$. Then, using that $\mathcal{A}_n\subset \{\umbral_0+\delta < \gamma_n<\widetilde{\gamma}\}$, we get that on $\mathcal{A}_n\cap\{T_n\leq a_n\}$, for any  $\umbral_0+\delta<\umbral< \widetilde{\gamma} $, we have that
\begin{align*}
\pl(\umbral ,{\mathcal{Z}}_n)-\pl(\umbral _0,{\mathcal{Z}}_n)
& \ge 
\lambda_n \left\{ (f(\umbral_0+\delta) - f(\umbral_0))-2 \frac{a_n}{\lambda_n}  \right\}>0\,,
\end{align*}	
for  $n$ large enough,  since $a_n/\lambda_n\to 0$ as $n\to\infty$.

It remains to prove that \eqref{eq:Wnsubset} holds. Note that   since $\ell$ satisfies \eqref{eq:lmayorlu0}, we have that 
\begin{equation}
\inf_{\umbral\leq \umbral_0-\delta} \ell(\umbral)-\ell(\umbral_0)=\nu_\delta >0\,.
\end{equation}
Choosing $\nu=\nu_\delta/4$ and  using that $f$ is a non-negative function, we obtain that, for any  $\umbral<\umbral_0-\delta$, if $ W_n\leq \nu $ then
\begin{align*}
\pl(\umbral ,{\mathcal{Z}}_n)-\pl(\umbral _0,{\mathcal{Z}}_n)&\;=\;\widehat{\ell}(\umbral,  {\mathcal{Z}}_n)-\widehat{\ell}(\umbral_0,  {\mathcal{Z}}_n)
	+\lambda_n(f(\umbral)-f(\umbral_0)) \\
&\;=\;\widehat{\ell}(\umbral,  {\mathcal{Z}}_n)-\ell(\umbral)+\ell(\umbral_0)-\widehat{\ell}(\umbral_0,  {\mathcal{Z}}_n)+\ell(\umbral)-\ell(\umbral_0) \\
& \hskip0.4in +\lambda_n(f(\umbral)-f(\umbral_0)) \\
&\;\geq\; -2 W_n +\ell(\umbral)-\ell(\umbral_0)-\lambda_n f(\umbral_0)\\
&\;\geq\; -2 \alpha +\nu_\delta -\lambda_n f(\umbral_0) = \frac{\nu_\delta}2 -\lambda_n f(\umbral_0)>0
\end{align*}
for $n$ large enough, since $\lambda_n\to 0$. 
\end{proof}

\subsection{Some preliminary results}
In this section, we include two results which are a key step to derive Theorem \ref{teo:tiro_del_final} from  Proposition \ref{prop:consistencia_ver2}. They will also be useful to prove Corollary \ref{coro:au} and Theorem \ref{teo:distriasint}. Their proof use standard  empirical processes tools.
 
In the sequel,   we adopt the following notation that strengths  the dependence on the regression coefficients
\begin{align*}
\widehat{\ell} (\umbral, \alpha,\beta)&=\frac{\prob_n \left[\{Y-( \alpha + \beta X)\}^2 \indica_{\{X \geq \umbral\}}\right]}{\prob_n\left(\indica_{\{X\geq \umbral\}}\right)}\;,\\
\ell(\umbral,  \alpha, \beta)&= \frac{\prob\left[\{Y-( \alpha +\beta X)\}^2\indica_{\{X\geq \umbral\}}\right]}{\prob\left(\indica_{\{X\geq \umbral\}}\right)}\,. 
\end{align*}
In this way, we get that  $\widehat{\ell} (\umbral,{\mathcal{Z}}_n)=\widehat{\ell} (\umbral, \widehat{\alpha}_{\umbral},\widehat{\beta}_{\umbral})$ and $ \ell(\umbral)=\ell(\umbral, \alpha_{\umbral}, \beta_{\umbral})$.

The proof of Lemma \ref{lema:todo}  needs also a basic result stated below.

\begin{lemma}{\label{lema:limite_empiricos}}
Let $\{(X_i, Y_i): i\geq 1\}$ are i.i.d., distributed as  $(X, Y)$ and fix $s, m\in \natu$.  Then,   if $\esp(X^{2s}\, Y^{2m})<\infty$, we have that the class of functions ${\mathcal{H}}=\{h(x,y)= x^s y^m \indica_{\{x\geq \umbral\}},\; \umbral\in \real\}$  is $P-$Donsker, , where $P$ stands for the joint probability measure of $(X,Y)$, so
\begin{equation}
\label{eq:clase_buena}
\sqrt{n}\displaystyle\sup_{\umbral\in \mathbb R} \left|\, \prob_n\left(X^s Y^m\indica_{\{X\geq \umbral\}}\right)-\prob \left(X^s Y^m\indica_{\{X\geq \umbral\}}\right)\, \right|=O_{\prob}(1)\;. 
\end{equation}
\end{lemma}

\begin{proof}
Lemma 2.6.15 in \citet{vander:wellner:1996} implies that the class of functions $\mathcal{G}=\{g(x,y)=x-\umbral,\; \umbral\in \real\}$ is a VC-subgraph class of functions with VC-index smaller or equal that 3.  
Then, from Lemmas 9.8 and 9.9 (iii)  in \citet{kosorok:2008}, we conclude that  $\mathcal{G}^{\star}=\{g(x,y)=\indica_{\{x\geq \umbral\}}, \;\umbral\in \real\}$ is also  a VC-subgraph class of functions. 
Using the permanence properties of  VC-classes stated in Lemma 9.9 (iv)  from \citet{kosorok:2008}, we obtain that  ${\mathcal{H}}$ is also  a VC-subgraph class of functions. 
Note that ${\mathcal{H}}$ has envelope $H(X,Y)=|X|^s |Y|^m$ which belongs to $L^2(P)$, since we assumed that $\esp(X^{2s}\, Y^{2m})<\infty$. Thus, from Lemma 2.6.7 in \citet{vander:wellner:1996}, we get that $\mathcal{H}$ satisfies the uniform entropy bound, which together with Theorem 2.5.2 in \citet{vander:wellner:1996} implies that $\mathcal{H}$ is a $P-$Donsker's class, concluding the proof of   \eqref{eq:clase_buena}. 
 \end{proof}

\bigskip


\begin{lemma}{\label{lema:todo}}
Assume $\{(X_i, Y_i): i\geq 1\}$ are i.i.d., distributed as  $(X, Y)$   satisfying the regression model \eqref{eq:regression} and let $\umbral_0$ be defined as in \eqref{eq:u_0}..    Assume that \ref{ass:momento},  \ref{ass:dist-continua}  and \ref{ass:u0mayorquea} to  \ref{ass:momento4}   hold. 
Let $\widehat{\umbral}_n$ be defined as in \eqref{eq:def_hat_umbral}  and  $\widehat{\ell}(\umbral, {\mathcal{Z}}_n)$  defined in \eqref{eq:ell_empirica}.  Then, for any fixed $\widetilde{\gamma}=F_X^{-1}(1-\eta)$ with $\eta>0$, we have that
	\begin{enumerate}
		\item[(a)  ] $\sup_{\umbral \leq \widetilde{\gamma}}\sqrt{n}( \widehat{\alpha}_{\umbral}-\alpha_{\umbral}) =O_{\prob}(1)\;\quad\hbox{and}\quad 
		\sup_{\umbral \leq \widetilde{\gamma}}\sqrt{n}( \widehat{\beta}_{\umbral}-\beta_{\umbral}) =O_{\prob}(1)$.
		\item[(b)  ] $\sup_{\umbral \leq \widetilde{\gamma}}\sqrt{n}\vert \widehat{\ell}  (\umbral,\widehat{\alpha}_{\umbral},\widehat{\beta}_{\umbral})- \widehat\ell(\umbral,\alpha_{\umbral},\beta_{\umbral})\vert =O_{\prob}(1)$.
		\item[(c)  ] $\sup_{\umbral \leq \widetilde{\gamma}}\sqrt{n}\vert  \widehat\ell(\umbral,\alpha_{\umbral},\beta_{\umbral})- \ell  (\umbral,\alpha_{\umbral},\beta_{\umbral})\vert =O_{\prob}(1)$.
	\end{enumerate}
\end{lemma}
It is worth mentioning that when the support of $\itS=[a,b]$ is bounded below, that is, when $a\ne -\infty$, the supremum over $\umbral \leq \widetilde{\gamma}$ equals that computed over $a\le \umbral \leq \widetilde{\gamma}$, since for values of $\umbral$ smaller than $a$, $\indica_{\{X\ge \umbral\}}=1$. For that reason, in the proof below we consider the limits when $\umbral\to a$ instead of $\umbral\to-\infty$. The case  of unbounded supports
is included having in mind that in such situation $a= -\infty$.

\begin{proof}[Proof of Lemma \ref{lema:todo}] 
To prove (a) note that 
$$\sqrt{n}\,(\widehat{\alpha}_{\umbral}-\alpha_{\umbral})= A_{1,n}(\umbral) - \beta_{\umbral}\, A_{2,n}(\umbral) -\sqrt{n}\,\left(\widehat{\beta}_{\umbral}-\beta_{\umbral}\right)\frac{\prob_n\left(Y\indica_{\{X\geq \umbral\}}\right)}
{\prob_n \left(\indica_{\{X\geq \umbral\}}\right)}\,,$$
where
\begin{align*}
A_{1,n}(\umbral) &=\sqrt{n}\,\left\{\frac{\prob_n\left(Y\indica_{\{X\geq \umbral\}}\right)}{\prob_n\left(\indica_{\{X\geq \umbral\}}\right)}-\frac{\prob\left( Y\indica_{\{X\geq \umbral\}}\right)}{\prob\left(\indica_{\{X\geq \umbral\}}\right)}\right\}\\
& = \frac{\sqrt{n}\left\{\prob_n\left(Y\indica_{\{X\geq \umbral\}}\right)-\prob\left(Y\indica_{\{X\geq \umbral\}}\right)\right\}\, \prob\left(\indica_{\{X\geq \umbral\}}\right)-\prob\left(Y\indica_{\{X\geq \umbral\}}\right)\, \sqrt{n}\left\{\prob_n\left(\indica_{\{X\geq \umbral\}}\right) -\prob\left(\indica_{\{X\geq \umbral\}}\right)\right\}}{\prob\left(\indica_{\{X\geq \umbral\}}\right)\prob_n\left(\indica_{\{X\geq \umbral\}}\right)}\,,\\
A_{2,n}(\umbral) &= \sqrt{n}\,\left\{ \frac{\prob_n\left(X\indica_{\{X\geq \umbral\}}\right)}{\prob_n\left(\indica_{\{X\geq \umbral\}}\right)} -
\frac{\prob\left( X\indica_{\{X\geq \umbral\}}\right)}{\prob\left(\indica_{\{X\geq \umbral\}}\right)}\right\}\\
& = \frac{\sqrt{n}\left\{\prob_n\left(X\indica_{\{X\geq \umbral\}}\right)-\prob\left(X\indica_{\{X\geq \umbral\}}\right)\right\}\, \prob\left(\indica_{\{X\geq \umbral\}}\right)-\prob\left(X\indica_{\{X\geq \umbral\}}\right)\,\sqrt{n} \left\{\prob_n\left(\indica_{\{X\geq \umbral\}}\right) -\prob\left(\indica_{\{X\geq \umbral\}}\right)\right\}}{\prob\left(\indica_{\{X\geq \umbral\}}\right)\prob_n\left(\indica_{\{X\geq \umbral\}}\right)}\,.
\end{align*}
Therefore, noticing that from Lemma \ref{lema:limite_empiricos}, 
\begin{align*}
\sqrt{n}\sup_{\umbral \leq \widetilde{\gamma}}\left|\prob_n\left(Y\indica_{\{X\geq \umbral\}}\right)- \prob\left(Y\indica_{\{X\geq \umbral\}}\right)\right|&=O_{\prob}(1)\,, \\
\sqrt{n}\sup_{\umbral \leq \widetilde{\gamma}}\left|\prob_n\left(X\indica_{\{X\geq \umbral\}}\right)- \prob\left(X\indica_{\{X\geq \umbral\}}\right)\right|&=O_{\prob}(1)\,, \\
\sqrt{n}\sup_{\umbral \leq \widetilde{\gamma}}\left|\prob_n\left( \indica_{\{X\geq \umbral\}}\right)- \prob\left( \indica_{\{X\geq \umbral\}}\right)\right|&=O_{\prob}(1)\,, 
\end{align*}
since $\esp(X^2)<\infty$ and $\esp(Y^2)<\infty$, to show that 
  $\sup_{\umbral \leq \widetilde{\gamma}}\sqrt{n}( \widehat{\alpha}_{\umbral}-\alpha_{\umbral}) =O_{\prob}(1)$ it will be enough to show that
   $\sup_{\umbral \leq \widetilde{\gamma}}\sqrt{n}( \widehat{\beta}_{\umbral}-\beta_{\umbral}) =O_{\prob}(1)$.

For that purpose, notice that
$$
\beta_{\umbral}\;=\;\frac{\prob\left(X\,Y\,\indica_{\{X\geq \umbral\}}\right)\;\prob\left(\indica_{\{X\geq \umbral\}}\right)-\prob\left(X \, \indica_{\{X\geq \umbral\}}\right)\;\prob\left(Y\,\indica_{\{X\geq \umbral\}}\right)}{\prob\left(X^2\,\indica_{\{X\geq \umbral\}}\right)\;\prob\left(\indica_{\{X\geq \umbral\}}\right)-\left\{\prob\left(X\,\indica_{\{X\geq \umbral\}}\right)\right\}^2}=\frac{A(\umbral)}{D(\umbral)}\,,
$$
and
$$
\widehat{\beta}_{\umbral}\;=\;\frac{\prob_n\left(X\,Y\,\indica_{\{X\geq \umbral\}}\right)\;\prob_n\left(\indica_{\{X\geq \umbral\}}\right)-
\prob_n\left(X \, \indica_{\{X\geq \umbral\}}\right)\;\prob_n\left(Y\,\indica_{\{X\geq \umbral\}}\right)}
{\prob_n\left(X^2\,\indica_{\{X\geq \umbral\}}\right)\;\prob_n\left(\indica_{\{X\geq \umbral\}}\right)-\left\{\prob_n\left(X\,\indica_{\{X\geq \umbral\}}\right)\right\}^2}=\frac{A_n(\umbral)}{D_n(\umbral)}\,.
$$
Hence,
\begin{align}
\sqrt{n}( \widehat{\beta}_{\umbral}-\beta_{\umbral}) &= \frac{1}{D_n(\umbral)\, D(\umbral)}\;\sqrt{n} \left\{A_n(\umbral) D(\umbral)- A(\umbral) D_n(\umbral)\right\}
\nonumber\\
& = \frac{1}{D_n(\umbral)\, D(\umbral)} \left\{ D(\umbral)\, \sqrt{n}\left(A_n(\umbral)- A(\umbral)\right) - A(\umbral) \, \sqrt{n}\left(D_n(\umbral)- D(\umbral)\right)\right\}
\label{eq:numerobeta}
\end{align}
First of all, observe that $\var_{\umbral}(X)$ is a continuous function of $\umbral$ which converges to $\var (X)$ when  $\umbral\to a$, hence $\inf_{u\leq \widetilde{\gamma}}\var_{\umbral}(X)>0$. Thus, for all $\umbral\leq \widetilde{\gamma}$ we have that
 \begin{align*}
 D(\umbral)=\prob\left(X^2\indica_{\{X\geq \umbral\}}\right)\prob\left(\indica_{\{X\geq \umbral\}}\right)-
 \left\{\prob\left(X\indica_{\{X\geq \umbral\}}\right)\right\}^2\;&=\;\left\{\prob \left(X\geq \umbral\right)\right\}^2\; \var_{\umbral}(X)\\
 &\ge \; \left[1-F_X(\widetilde{\gamma})\right]^2\inf_{u\leq \widetilde{\gamma}}\var_{\umbral}(X) >\;0\;, 
 \end{align*}
 since $1-F_X(\widetilde{\gamma})=\eta>0$, meaning that $D=\inf_{\umbral\leq \widetilde{\gamma}}  D(\umbral)>\;0$.  
 
 Lemma \ref{lema:limite_empiricos} entails that  $\sup_{\umbral\leq \widetilde{\gamma}} |D_n(\umbral)-D(\umbral)| \convprob 0$, so $\lim_{n\to \infty} \prob\left(\inf_{\umbral\leq \widetilde{\gamma}}  D_n(\umbral)>D/2\right)=1$. Thus, taking into account that   $\sup_{\umbral\leq \widetilde{\gamma}}  D(\umbral)\le  \esp(X^2)$ and $\sup_{\umbral\leq \widetilde{\gamma}} | A(\umbral)|\le \esp |X\, Y| + \esp |X|\,\esp |Y|<\infty$, from \eqref{eq:numerobeta} we conclude that to show that $\sqrt{n}( \widehat{\beta}_{\umbral}-\beta_{\umbral})=O_{\prob}(1)$, it will be enough to prove that
 $$\sqrt{n}\left(A_n(\umbral)- A(\umbral)\right)=O_{\prob}(1)\qquad \mbox{and} \qquad \sqrt{n}\left(D_n(\umbral)- D(\umbral)\right)=O_{\prob}(1)\,,$$
which follow easily from Lemma \ref{lema:limite_empiricos}, since $\esp(X^2Y^2)<\infty$, $\esp(X^2)<\infty$, $\esp(Y^2)<\infty$, see equation \eqref{eq:clase_buena}.

\noindent b) To prove (b), observe that 
\begin{equation}
\label{eq:difelles}
\widehat\ell(\umbral,\alpha_{\umbral},\beta_{\umbral})-\widehat{\ell}  (\umbral,\widehat{\alpha}_{\umbral},\widehat{\beta}_{\umbral}
	) \; = \; \frac{\prob_n\left(\widehat{\Delta}(u)\,\indica_{\{X\geq \umbral\}}\right)}{\prob_n\left(\indica_{\{X\geq \umbral\}}\right)}\,,
\end{equation}	
where
$$ \widehat{\Delta}(u)= \Big\{Y-(\alpha_{\umbral} +\beta_{\umbral} X)\Big\}^2 -\left\{Y-(\widehat{\alpha}_{\umbral} +\widehat{\beta}_{\umbral} X)\right\}^2\,.$$
Straightforward calculations allow to see that
\begin{equation*}
\widehat{\Delta}(u)=\left(\alpha^2_{\umbral}-\widehat{\alpha}^2_{\umbral}\right)+ \left(\beta^2_{\umbral} -\widehat{\beta}^2_{\umbral}\right)\,X^2 -2\left(\alpha_{\umbral} -\widehat{\alpha}_{\umbral}\right)\, Y\,-\,2\,\left(\beta_{\umbral}-\widehat{\beta}_{\umbral}\right)Y X\,+\,2\left(\alpha_{\umbral} \beta_{\umbral}-\widehat{\alpha}_{\umbral} \widehat{\beta}_{\umbral}\right)\, X\,.
\end{equation*}
Then, replacing in \eqref{eq:difelles} we obtain that 
\begin{align*} 
\widehat\ell(\umbral,\alpha_{\umbral},\beta_{\umbral})& -\widehat{\ell}  (\umbral,\widehat{\alpha}_{\umbral},\widehat{\beta}_{\umbral}) \; 
=  \; \frac{1}{\prob_n\left(\indica_{\{X\geq \umbral\}}\right)}\Bigl[\left(\beta^2_{\umbral} -\widehat{\beta}^2_{\umbral}\right)\,\prob_n \left(X^2\indica_{\{X\geq \umbral\}}\right)   -2\left(\alpha_{\umbral} -\widehat{\alpha}_{\umbral}\right) \prob_n\left(Y\indica_{\{X\geq \umbral\}}\right) \\
	&-\,2\,\left(\beta_{\umbral}-\widehat{\beta}_{\umbral}\right)\prob_n\left(Y\, X\, \indica_{\{X\geq \umbral\}}\right)
	\,+\,2\left(\alpha_{\umbral} \beta_{\umbral}-\widehat{\alpha}_{\umbral} \widehat{\beta}_{\umbral}\right)\prob_n \left(X\indica_{\{X\geq \umbral\}}\right)\Bigr]+ \left(\alpha^2_{\umbral}-\widehat{\alpha}^2_{\umbral}\right) \,.
\end{align*}
Combining the results established in (a), Lemma \ref{lema:limite_empiricos} and the lower bound for $\prob(X\geq \umbral)$   when $\umbral\leq \widetilde{\gamma}$ given by $\prob(X\geq \umbral)\ge \eta>0$,  the desired result is obtained. 

\noindent c) Finally, to prove (c) note that 
\begin{align*}
\widehat{\ell}(\umbral, \alpha_{\umbral}, \beta_{\umbral})	&=  \frac{\prob_n \left(\left[Y^2-2 Y \alpha_{\umbral}\,-\,2\,Y \, X\,\beta_{\umbral}+\,\alpha_{\umbral}^2\,+\, \beta_{\umbral}^2 \, X^2\, +\,2\,\alpha_{\umbral} \beta_{\umbral} X\right] \indica_{\{X\geq \umbral\}}\right)}{\prob_n\left(\indica_{\{X\geq \umbral\}}\right)}=\frac{B_n(u)}{D_n(u)}\,, 
\end{align*}	
while 
\begin{align*}
	\ell(\umbral, \alpha_{\umbral}, \beta_{\umbral})&=  \frac{\prob\left(\left[Y^2-2 Y \alpha_{\umbral}\,-\,2\,Y \, X\,\beta_{\umbral}+\,\alpha_{\umbral}^2\,+\, \beta_{\umbral}^2 \, X^2\, +\,2\,\alpha_{\umbral} \beta_{\umbral} X\right] \indica_{\{X\geq \umbral\}}\right)}{\prob\left(\indica_{\{X\geq \umbral\}}\right)}=\frac{B(u)}{D(u)}\,.
\end{align*}	
 Denote $\widehat{\Delta}_n=\sup_{\umbral \leq \widetilde{\gamma}}\sqrt{n}\vert  \widehat\ell(\umbral,\alpha_{\umbral},\beta_{\umbral})- \ell  (\umbral,\alpha_{\umbral},\beta_{\umbral})\vert  $ and define $\itI=[a, \widetilde{\gamma}]$ when  $a\ne -\infty$  and $\itI=(-\infty, \widetilde{\gamma}]$, otherwise. Moreover, let  $\xi( \widetilde{\gamma})=\sup_{\umbral \in \itI} \vert B(\umbral)\vert$. To bound $\widehat{\Delta}_n$ recall that 
$$\sup_{\umbral \leq \widetilde{\gamma}}\sqrt{n}\vert  \widehat\ell(\umbral,\alpha_{\umbral},\beta_{\umbral})- \ell  (\umbral,\alpha_{\umbral},\beta_{\umbral})\vert =  \sup_{a\le \umbral \leq \widetilde{\gamma}}\sqrt{n}\vert  \widehat\ell(\umbral,\alpha_{\umbral},\beta_{\umbral})- \ell  (\umbral,\alpha_{\umbral},\beta_{\umbral})\vert\, ,$$
when $a\ne -\infty$ and that from  the proof of Lemma \ref{lema-basico}, $\alpha_{\umbral}$ and $\beta_{\umbral}$ are continuous functions of $\umbral$. Hence, when  $a\ne -\infty$, $\alpha_{\umbral}$ and $\beta_{\umbral}$ are bounded in $[a, \widetilde{\gamma}] $, so $\xi( \widetilde{\gamma}) <\infty$. In contrast, when $a= -\infty$,   $\alpha_{\umbral}\to \alpha^{*}$ and $\beta_{\umbral}\to \beta^{*}$ when $\umbral\to -\infty$, where $\alpha^{*}$ and $ \beta^{*}$ are defined in \eqref{eq:betastar}, hence using this convergence and the boundedness of $\alpha_{\umbral}$ and $\beta_{\umbral}$ in any compact set, we obtain that $\alpha_{\umbral}$ and $\beta_{\umbral}$ are bounded in  $ \{ \umbral \leq \widetilde{\gamma} \}$ implying that we also have $\xi( \widetilde{\gamma}) <\infty$, when  $a= -\infty$. 

 Therefore, using that  $\inf_{ \{ \umbral \leq \widetilde{\gamma} \}}\prob(X\geq \umbral)\ge \eta>0$, we obtain that
\begin{align}
\widehat{\Delta}_n
& \le
\frac{\sup_ {\umbral \in \itI} \sqrt{n} \vert B_n(\umbral)-B(\umbral) \vert +\sup_{\umbral \in \itI}  \sqrt{n} \vert D_n(\umbral)-D(\umbral) \vert\, 
\xi( \widetilde{\gamma})}{\eta \;\inf_{ \{ \umbral \leq \widetilde{\gamma} \}}D_n(\umbral)}\,.
\label{eq:cota}
\end{align}
Note that   Lemma \ref{lema:limite_empiricos} entails that $\lim_{n\to \infty} \prob\left(\inf_{\umbral\leq \widetilde{\gamma}}  D_n(\umbral)>\eta/2\right)=1$. Furthermore, invoking  again Lemma \ref{lema:limite_empiricos} and using the boundedness of $\alpha_{\umbral}$ and $\beta_{\umbral}$ on $\itI$, we get that  
\begin{align*}
\sup_{a\le \umbral \leq \widetilde{\gamma}} \sqrt{n} \vert B_n(\umbral)-B(\umbral) \vert   =O_{\prob}(1)\qquad \mbox{ and }\qquad 
\sup_{a\le \umbral \leq \widetilde{\gamma}} \sqrt{n} \vert D_n(\umbral)-D(\umbral) \vert   =O_{\prob}(1)\;.
\end{align*}
 The result follows now using \eqref{eq:cota}.  
\end{proof}

\subsection{Proof of Theorem \ref{teo:tiro_del_final}, Corollary \ref{coro:au} and Theorem \ref{teo:distriasint}}

\begin{proof}[Proof of Theorem \ref{teo:tiro_del_final}]
To obtain the weak consistency of  $\widehat{\umbral}_n$, it will be enough to show that the assumptions in Proposition \ref{prop:consistencia_ver2} hold.

First note that Lemma \ref{lema-basico} implies that \eqref{eq:l_minimiza}, \eqref{eq:l_minimiza2} and \eqref{eq:lmayorlu0} hold. Then, from Proposition \ref{prop:consistencia_ver2} to derive that $\widehat{\umbral}_n \convprob \umbral_0$ it will be enough to show that the function $\widehat{\ell}(\umbral, {\mathcal{Z}}_n)$  defined in \eqref{eq:ell_empirica} satisfies the requirements (a) and (b) of Proposition \ref{prop:consistencia_ver2}, for some sequence $a_n$ such that $a_n/\lambda_n\to 0$. Taking into account that $\lambda_n \to 0$,  and   
$n^{1/2}\lambda_n/ \varphi_n\to +\infty$, where  $\varphi_n\to +\infty$, we have that  $n^{1/2}/ \varphi_n\to +\infty$ , so the choice $a_n=n^{-1/2}\,  \varphi_n$,   ensures that  $a_n/\lambda_n\to 0$. 

Let $0<\eta<1$ be fixed  and define $\widetilde{\gamma}=F_X^{-1}(1-\eta)$, combining the results given in  items (b) and (c) of  Lemma \ref{lema:todo}, we obtain that
\begin{equation}
\label{eq:procesos_empiricos}
\sqrt{n}\sup_{u\leq \widetilde{\gamma}} \vert \widehat{\ell}(\umbral, {\mathcal{Z}}_n)-\ell(u)\vert =O_{\prob}(1)\,.  
\end{equation}
Hence, using that $a_n=n^{-1/2}\,  \varphi_n$ and $\varphi_n\to +\infty$, we get that $a_n^{-1}\, n^{-1/2}=1/ \varphi_n\to 0$, so from \eqref{eq:procesos_empiricos} we conclude that
\begin{equation}{\label{eq:condiciona}}
\lim_{n\to \infty} \prob\left(\sup_{\umbral\leq \widetilde{\gamma}} \vert \widehat{\ell}(u, {\mathcal{Z}}_n)-\ell (u)\vert\leq a_n\right)=1\;.
\end{equation}
Therefore, condition (b) of Proposition \ref{prop:consistencia_ver2} is satisfied.  Note that assumptions  \ref{ass:u0mayorquea} and \ref{ass:cuantil} entail that $0<F_X(u_0)<1$, hence condition (a) follows immediately from \eqref{eq:condiciona} taking $\eta=F_X(u_0)$.
\end{proof}

\begin{proof}[Proof of Corollary \ref{coro:au}]
Taking into account that $\widehat{\umbral}_n \convprob \umbral_0$	and that  $\alpha_{\umbral}$ and $\beta_{\umbral}$ depend continuously on $\umbral$ (see the proof of Lemma \ref{lema-basico}), we have that
\begin{equation}
\label{eq:alphauhat}
\alpha_{\widehat{\umbral}_n }\convprob \alpha_{\umbral_0} \quad \mbox{and} \quad \beta_{\widehat{\umbral}_n }\convprob \beta_{\umbral_0}\,.
\end{equation}
The  uniform convergences stated at item (a)  of Lemma \ref{lema:todo}  implies that for any fixed $\widetilde{\gamma}>\gamma$ such that $1-F_X(\widetilde{\gamma})>0$, we have 
$$\sup_{\umbral \leq \widetilde{\gamma}} ( \widehat{\alpha}_{\umbral}-\alpha_{\umbral}) \convprob 0\;\quad\hbox{and}\quad 
		\sup_{\umbral \leq \widetilde{\gamma}} ( \widehat{\beta}_{\umbral}-\beta_{\umbral}) \convprob 0\,.$$
By \eqref{eq:conv1}, $\prob(\mathcal{A})=\prob( \exists n_0: \forall n\ge n_0 \quad \gamma_n<\widetilde{\gamma})=1$, so using that $\widehat{\umbral}_n\le \gamma_n$, we get that
\begin{equation}
\label{eq:hatalphauhat}
\widehat{\alpha}_{\widehat{\umbral}_n }-\alpha_{\widehat{\umbral}_n } \convprob 0 \quad \mbox{and} \quad \widehat{\beta}_{\widehat{\umbral}_n }-\beta_{\widehat{\umbral}_n } \convprob 0 \,.
\end{equation}
Combining \eqref{eq:alphauhat} and 	\eqref{eq:hatalphauhat}, the result follows.
	\end{proof}

\begin{proof}[Proof of Theorem \ref{teo:distriasint}] 
	In what follows, we denote $\theta_{\umbral}=(\alpha_{\umbral},\beta_{\umbral})$ and $\widehat{\theta}_{\umbral} =(\widehat{\alpha}_{\umbral},\widehat{\beta}_{\umbral})$. To prove the result it is enough to show that $\widehat{\theta}_{\widehat{\umbral}+\psi}$ is asymptotically equivalent to $\widehat{\theta}_{\umbral_0+\psi}$, meaning that 
$\sqrt{n}\{\widehat{\theta}_{\widehat{\umbral}+\psi}-\widehat{\theta}_{\umbral_0+\psi}\}=o_{\prob}(1)$. 
Recall that  $\theta_{\umbral}=\theta_{\umbral_0}$ for $u\geq \umbral_0$ and 
	$\lim_{n\to \infty}\prob\left(\widehat{\umbral}_n+\psi > \umbral_0\right)=1$, since $\prob(|\widehat{\umbral}_n- \umbral_0|<\psi/2)\to 1$. 
	Then, we have that 
	\begin{equation}
	\label{eq:difthetas}
	\lim_{n\to \infty}\prob\left(\theta_{\widehat{\umbral}+\psi}=\theta_{\umbral_0}=\theta_{\umbral_0+\psi}\right)=1\,.
	\end{equation}
Define $E_n(u)=\sqrt{n}\{\widehat{\theta}_{\umbral}-\theta_{\umbral}\}$ and $R_n= \sqrt{n}\left(\theta_{\widehat{\umbral}+\psi}-\theta_{\umbral_0+\psi}\right)$ and notice that
\begin{equation}
	\sqrt{n}\left(\widehat{\theta}_{\widehat{\umbral}+\psi}-\widehat{\theta}_{\umbral_0+\psi}\right)=E_n(\widehat{\umbral}+\psi)-E_n(\umbral_0+\psi) + R_n \,.
\end{equation}
On the one hand, \eqref{eq:difthetas} implies that $R_n=o_{\prob}(1)$.
On the other hand, from   Lemma \ref{lema:limite_empiricos}, for  $s,m=0,1,2$, $s+m\le 2$,  the classes of functions $\mathcal{H}_{s,m} = \left\{h(x,y)=x^s\,y^m\indica_{\{x\geq \umbral\}}; \umbral \in \real\right\}$,  are Donsker classes. Therefore, if we define    ${\mathbb G}_n h =\sqrt{n}(\prob_n h(X,Y)- \prob h(X,Y))$  we get that the empirical processes $\{{\mathbb G}_n h : h\in \mathcal{H}_{s,m}\} $ are asymptotically equicontinuous. This fact, along with the expansions used in the proof of Lemma \ref{lema:todo} for $\alpha_{\umbral}$, $\beta_{\umbral}$, $\widehat{\alpha}_{\umbral}$ and $\widehat{\beta}_{\umbral}$ and the consistency of $\widehat{\umbral}$   to $\umbral_0$, allows us to deduce that $E_n(\widehat{\umbral}+\psi) - E_n(\umbral_0+\psi) = o_{\prob}(1)$, concluding the proof. 
\end{proof}

\bibliographystyle{apalike}

\bibliography{biblio_reg_change}

\end{document}